\documentclass{amsart}
\usepackage{amsfonts}
\usepackage{amsmath,amscd}
\usepackage{amsthm}
\usepackage{amssymb}
\usepackage{latexsym}

\numberwithin{equation}{section}

\usepackage{color}

\newtheorem{thm}{Theorem}[section]
\newtheorem{cor}[thm]{Corollary}
\newtheorem{lem}[thm]{Lemma}

\newtheorem{prop}[thm]{Proposition}

\newtheorem{example}[thm]{Example}
\newtheorem{defn}[thm]{Definition}

\newtheorem{rem}[thm]{Remark}

\numberwithin{equation}{section}
\begin{document}
\newcommand{\beqa}{\begin{eqnarray}}
\newcommand{\eeqa}{\end{eqnarray}}
\newcommand{\thmref}[1]{Theorem~\ref{#1}}
\newcommand{\secref}[1]{Sect.~\ref{#1}}
\newcommand{\lemref}[1]{Lemma~\ref{#1}}
\newcommand{\propref}[1]{Proposition~\ref{#1}}
\newcommand{\corref}[1]{Corollary~\ref{#1}}
\newcommand{\remref}[1]{Remark~\ref{#1}}
\newcommand{\er}[1]{(\ref{#1})}
\newcommand{\nc}{\newcommand}
\newcommand{\rnc}{\renewcommand}

\nc{\cal}{\mathcal}

\nc{\goth}{\mathfrak}
\rnc{\bold}{\mathbf}
\renewcommand{\frak}{\mathfrak}
\renewcommand{\Bbb}{\mathbb}

\newcommand{\id}{\text{id}}
\nc{\Cal}{\mathcal}
\nc{\Xp}[1]{X^+(#1)}
\nc{\Xm}[1]{X^-(#1)}
\nc{\on}{\operatorname}
\nc{\ch}{\mbox{ch}}
\nc{\Z}{{\bold Z}}
\nc{\J}{{\mathcal J}}
\nc{\C}{{\bold C}}
\nc{\Q}{{\bold Q}}
\nc{\oC}{{\widetilde{C}}}
\nc{\oc}{{\tilde{c}}}
\nc{\ocI}{ \overline{\cal I}}
\nc{\og}{{\tilde{\gamma}}}
\nc{\lC}{{\overline{C}}}
\nc{\lc}{{\overline{c}}}
\nc{\Rt}{{\tilde{R}}}

\nc{\tW}{{\textsf{W}}}
\nc{\tG}{{\textsf{G}}}

\nc{\tw}{{\textsf{w}}}
\nc{\tg}{{\textsf{g}}}

\nc{\tx}{{\textsf{x}}}
\nc{\tho}{{\textsf{h}}}
\nc{\tk}{{\textsf{k}}}
\nc{\tep}{{\bf{\cal E}}}

\nc{\te}{{\textsf{e}}}
\nc{\tf}{{\textsf{f}}}
\nc{\tK}{{\textsf{K}}}

\nc{\odel}{{\overline{\delta}}}

\def\pr#1{\left(#1\right)_\infty}  

\renewcommand{\P}{{\mathcal P}}
\nc{\N}{{\Bbb N}}
\nc\beq{\begin{equation}}
\nc\enq{\end{equation}}
\nc\lan{\langle}
\nc\ran{\rangle}
\nc\bsl{\backslash}
\nc\mto{\mapsto}
\nc\lra{\leftrightarrow}
\nc\hra{\hookrightarrow}
\nc\sm{\smallmatrix}
\nc\esm{\endsmallmatrix}
\nc\sub{\subset}
\nc\ti{\tilde}
\nc\nl{\newline}
\nc\fra{\frac}
\nc\und{\underline}
\nc\ov{\overline}
\nc\ot{\otimes}

\nc\ochi{\overline{\chi}}
\nc\bbq{\bar{\bq}_l}
\nc\bcc{\thickfracwithdelims[]\thickness0}
\nc\ad{\text{\rm ad}}
\nc\Ad{\text{\rm Ad}}
\nc\Hom{\text{\rm Hom}}
\nc\End{\text{\rm End}}
\nc\Ind{\text{\rm Ind}}
\nc\Res{\text{\rm Res}}
\nc\Ker{\text{\rm Ker}}
\rnc\Im{\text{Im}}
\nc\sgn{\text{\rm sgn}}
\nc\tr{\text{\rm tr}}
\nc\Tr{\text{\rm Tr}}
\nc\supp{\text{\rm supp}}
\nc\card{\text{\rm card}}
\nc\bst{{}^\bigstar\!}
\nc\he{\heartsuit}
\nc\clu{\clubsuit}
\nc\spa{\spadesuit}
\nc\di{\diamond}
\nc\cW{\cal W}
\nc\cG{\cal G}
\nc\cZ{\cal Z}
\nc\ocW{\overline{\cal W}}
\nc\ocZ{\overline{\cal Z}}
\nc\al{\alpha}
\nc\bet{\beta}
\nc\ga{\gamma}
\nc\de{\delta}
\nc\ep{\epsilon}
\nc\io{\iota}
\nc\om{\omega}
\nc\si{\sigma}
\rnc\th{\theta}
\nc\ka{\kappa}
\nc\la{\lambda}
\nc\ze{\zeta}

\nc\vp{\varpi}
\nc\vt{\vartheta}
\nc\vr{\varrho}

\nc\odelta{\overline{\delta}}
\nc\Ga{\Gamma}
\nc\De{\Delta}
\nc\Om{\Omega}
\nc\Si{\Sigma}
\nc\Th{\Theta}
\nc\La{\Lambda}

\nc\boa{\bold a}
\nc\bob{\bold b}
\nc\boc{\bold c}
\nc\bod{\bold d}
\nc\boe{\bold e}
\nc\bof{\bold f}
\nc\bog{\bold g}
\nc\boh{\bold h}
\nc\boi{\bold i}
\nc\boj{\bold j}
\nc\bok{\bold k}
\nc\bol{\bold l}
\nc\bom{\bold m}
\nc\bon{\bold n}
\nc\boo{\bold o}
\nc\bop{\bold p}
\nc\boq{\bold q}
\nc\bor{\bold r}
\nc\bos{\bold s}
\nc\bou{\bold u}
\nc\bov{\bold v}
\nc\bow{\bold w}
\nc\boz{\bold z}

\nc\ba{\bold A}
\nc\bb{\bold B}
\nc\bc{\bold C}
\nc\bd{\bold D}
\nc\be{\bold E}
\nc\bg{\bold G}
\nc\bh{\bold H}
\nc\bi{\bold I}
\nc\bj{\bold J}
\nc\bk{\bold K}
\nc\bl{\bold L}
\nc\bm{\bold M}
\nc\bn{\bold N}
\nc\bo{\bold O}
\nc\bp{\bold P}
\nc\bq{\bold Q}
\nc\br{\bold R}
\nc\bs{\bold S}
\nc\bt{\bold T}
\nc\bu{\bold U}
\nc\bv{\bold V}
\nc\bw{\bold W}
\nc\bz{\bold Z}
\nc\bx{\bold X}

\nc\ca{\mathcal A}
\nc\cb{\mathcal B}
\nc\cc{\mathcal C}
\nc\cd{\mathcal D}
\nc\ce{\mathcal E}
\nc\cf{\mathcal F}
\nc\cg{\mathcal G}
\rnc\ch{\mathcal H}
\nc\ci{\mathcal I}
\nc\cj{\mathcal J}
\nc\ck{\mathcal K}
\nc\cl{\mathcal L}
\nc\cm{\mathcal M}
\nc\cn{\mathcal N}
\nc\co{\mathcal O}
\nc\cp{\mathcal P}
\nc\cq{\mathcal Q}
\nc\car{\mathcal R}
\nc\cs{\mathcal S}
\nc\ct{\mathcal T}
\nc\cu{\mathcal U}
\nc\cv{\mathcal V}
\nc\cz{\mathcal Z}
\nc\cx{\mathcal X}
\nc\cy{\mathcal Y}

\nc\e[1]{E_{#1}}
\nc\ei[1]{E_{\delta - \alpha_{#1}}}
\nc\esi[1]{E_{s \delta - \alpha_{#1}}}
\nc\eri[1]{E_{r \delta - \alpha_{#1}}}
\nc\ed[2][]{E_{#1 \delta,#2}}
\nc\ekd[1]{E_{k \delta,#1}}
\nc\emd[1]{E_{m \delta,#1}}
\nc\erd[1]{E_{r \delta,#1}}

\nc\ef[1]{F_{#1}}
\nc\efi[1]{F_{\delta - \alpha_{#1}}}
\nc\efsi[1]{F_{s \delta - \alpha_{#1}}}
\nc\efri[1]{F_{r \delta - \alpha_{#1}}}
\nc\efd[2][]{F_{#1 \delta,#2}}
\nc\efkd[1]{F_{k \delta,#1}}
\nc\efmd[1]{F_{m \delta,#1}}
\nc\efrd[1]{F_{r \delta,#1}}

\nc\fa{\frak a}
\nc\fb{\frak b}
\nc\fc{\frak c}
\nc\fd{\frak d}
\nc\fe{\frak e}
\nc\ff{\frak f}
\nc\fg{\frak g}
\nc\fh{\frak h}
\nc\fj{\frak j}
\nc\fk{\frak k}
\nc\fl{\frak l}
\nc\fm{\frak m}
\nc\fn{\frak n}
\nc\fo{\frak o}
\nc\fp{\frak p}
\nc\fq{\frak q}
\nc\fr{\frak r}
\nc\fs{\frak s}
\nc\ft{\frak t}
\nc\fu{\frak u}
\nc\fv{\frak v}
\nc\fz{\frak z}
\nc\fx{\frak x}
\nc\fy{\frak y}

\nc\fA{\frak A}
\nc\fB{\frak B}
\nc\fC{\frak C}
\nc\fD{\frak D}
\nc\fE{\frak E}
\nc\fF{\frak F}
\nc\fG{\frak G}
\nc\fH{\frak H}
\nc\fJ{\frak J}
\nc\fK{\frak K}
\nc\fL{\frak L}
\nc\fM{\frak M}
\nc\fN{\frak N}
\nc\fO{\frak O}
\nc\fP{\frak P}
\nc\fQ{\frak Q}
\nc\fR{\frak R}
\nc\fS{\frak S}
\nc\fT{\frak T}
\nc\fU{\frak U}
\nc\fV{\frak V}
\nc\fZ{\frak Z}
\nc\fX{\frak X}
\nc\fY{\frak Y}
\nc\tfi{\ti{\Phi}}
\nc\bF{\bold F}
\rnc\bol{\bold 1}

\nc\ua{\bold U_\A}

\nc\qinti[1]{[#1]_i}
\nc\q[1]{[#1]_q}
\nc\xpm[2]{E_{#2 \delta \pm \alpha_#1}}  
\nc\xmp[2]{E_{#2 \delta \mp \alpha_#1}}
\nc\xp[2]{E_{#2 \delta + \alpha_{#1}}}
\nc\xm[2]{E_{#2 \delta - \alpha_{#1}}}
\nc\hik{\ed{k}{i}}
\nc\hjl{\ed{l}{j}}
\nc\qcoeff[3]{\left[ \begin{smallmatrix} {#1}& \\ {#2}& \end{smallmatrix}
\negthickspace \right]_{#3}}
\nc\qi{q}
\nc\qj{q}

\nc\ufdm{{_\ca\bu}_{\rm fd}^{\le 0}}


\nc\isom{\cong} 

\nc{\pone}{{\Bbb C}{\Bbb P}^1}
\nc{\pa}{\partial}
\def\H{\mathcal H}
\def\L{\mathcal L}
\nc{\F}{{\mathcal F}}
\nc{\Sym}{{\goth S}}
\nc{\A}{{\mathcal A}}
\nc{\arr}{\rightarrow}
\nc{\larr}{\longrightarrow}

\nc{\ri}{\rangle}
\nc{\lef}{\langle}
\nc{\W}{{\mathcal W}}
\nc{\uqatwoatone}{{U_{q,1}}(\su)}
\nc{\uqtwo}{U_q(\goth{sl}_2)}
\nc{\dij}{\delta_{ij}}
\nc{\divei}{E_{\alpha_i}^{(n)}}
\nc{\divfi}{F_{\alpha_i}^{(n)}}
\nc{\Lzero}{\Lambda_0}
\nc{\Lone}{\Lambda_1}
\nc{\ve}{\varepsilon}
\nc{\bepsilon}{\bar{\epsilon}}
\nc{\bak}{\bar{k}}
\nc{\phioneminusi}{\Phi^{(1-i,i)}}
\nc{\phioneminusistar}{\Phi^{* (1-i,i)}}
\nc{\phii}{\Phi^{(i,1-i)}}
\nc{\Li}{\Lambda_i}
\nc{\Loneminusi}{\Lambda_{1-i}}
\nc{\vtimesz}{v_\ve \otimes z^m}

\nc{\asltwo}{\widehat{\goth{sl}_2}}
\nc\ag{\widehat{\goth{g}}}  
\nc\teb{\tilde E_\boc}
\nc\tebp{\tilde E_{\boc'}}

\newcommand{\LR}{\bar{R}}
\newcommand{\eeq}{\end{equation}}
\newcommand{\ben}{\begin{eqnarray}}
\newcommand{\een}{\end{eqnarray}}

\title[The alternating presentation of $U_q(\widehat{gl_2})$]{ The alternating presentation of $U_q(\widehat{gl_2})$ \\ from Freidel-Maillet algebras} 
\author{Pascal Baseilhac}
\address{Institut Denis-Poisson CNRS/UMR 7013 - Universit\'e de Tours - Universit\'e d'Orl\'eans
Parc de Grammont, 37200 Tours, 
FRANCE}
\email{pascal.baseilhac@idpoisson.fr}

\begin{abstract} 
An infinite dimensional algebra denoted  $\bar{\cal A}_q$ that is isomorphic to a central extension of $U_q^+$ -  the positive part of $U_q(\widehat{sl_2})$ -  has been recently proposed by Paul Terwilliger.  It provides an `alternating' Poincar\'e-Birkhoff-Witt (PBW) basis besides the known Damiani's PBW basis built from positive  root vectors. In this paper, a presentation of $\bar{\cal A}_q$ in terms of a Freidel-Maillet type algebra is obtained. Using this presentation:  (a) finite dimensional tensor product representations for $\bar{\cal A}_q$ are constructed; (b) explicit   isomorphisms from $\bar{\cal A}_q$ to certain  Drinfeld type `alternating' subalgebras of $U_q(\widehat{gl_2})$ are obtained; (c) the image  in $U_q^+$ of all the generators  of $\bar{\cal A}_q$ in terms of Damiani's root vectors is obtained. A new tensor product decomposition for $U_q(\widehat{sl_2})$ in terms of Drinfeld type `alternating' subalgebras follows. The specialization $q\rightarrow 1$ of  $\bar{\cal A}_q$ is also introduced and  studied in details. In this case, a presentation is given as a non-standard Yang-Baxter algebra. {\it This paper is dedicated to Paul Terwilliger for his 65th birthday.} 
\end{abstract}

\maketitle

\vskip -0.5cm



\section{Introduction}
Quantum affine algebras are known to admit at least three presentations. For $U_q(\widehat{sl_2})$, the first presentation originally introduced in \cite{Jim,Dr0} -  referred as the Drinfeld-Jimbo presentation in the literature - is given in terms of generators $\{E_i,F_i,K_i^{\pm 1}
|i=0,1\}$ and relations, see Appendix \ref{apA}. The so-called Drinfeld second presentation was found later on  \cite{Dr}, given in terms of generators $\{{\tx}_k^{\pm}, \tho_{\ell},\tK^{\pm 1}, C^{\pm 1/2}
|k\in {\mathbb Z},\ell\in {\mathbb Z}\backslash
\{0\} \}$ and relations. The third one, obtained in \cite{RS}, takes the form of a  Faddeev-Reshetikhin-Takhtajan (FRT) presentation \cite{FRT89}. 
In these definitions, note that the so-called derivation generator is ommited (see \cite[Remark  2, p. 393]{CPb}). 
  In the following, we denote respectively  $U_q^{DJ}$, $U_q^{Dr}$ and $U_q^{RS}$ these presentations of $U_q(\widehat{sl_2})$. In addition, for $U_q(\widehat{sl_2})$ note that a fourth presentation called `equitable', denoted $U_q^{IT}$, has been introduced in \cite{IT03}. It is generated by $\{y_i^\pm,k_i^\pm|i=0,1\}$. For the explicit isomorphism $U_q^{IT} \rightarrow U_q^{DJ}$, see \cite[Theorem 2.1]{IT03}.

	\vspace{1mm}

The construction of a  Poincar\'e-Birkhoff-Witt (PBW) basis for $U_q(\widehat{sl_2})$ \cite{Da,Beck} on one hand, and the FRT presentation of Ding-Frenkel \cite{DF93} on the other hand brought major contributions to the subject, by establishing the explicit isomorphisms between  $U_q^{DJ}$, $U_q^{Dr}$ and $U_q^{RS}$ (see also \cite{Ji,Da14}). To motivate the goal of the present paper, as a preliminary let us briefly review the main results of \cite{Da,Beck} and \cite{DF93}.\vspace{1mm}

$\bullet$ To establish the isomorphism  between  $U_q^{DJ}$ and $U_q^{Dr}$, the main ingredient is the construction of a PBW basis. In \cite{Da}, it is shown that
the so-called positive part of  $U_q(\widehat{sl_2})$ denoted $U_q^{DJ,+}$ - cf. Notation 1.2 - is generated by positive (real and imaginary)  root vectors \cite[Section 3.1]{Da}. The root vectors are obtained using  Lusztig's braid group action on  $U_q^{DJ}$ \cite{L93}.
 Based on the structure of the commutation relations among the root vectors,  a PBW basis for $U_q^{DJ,+}$ is first obtained  \cite[Section 4]{Da}. Then, introduce  the subalgebras $U_q^{DJ,-}$, $U_q^{DJ,0}$ of $U_q^{DJ}$.  Thanks to the tensor product decomposition  $U_q^{DJ} \cong U_q^{DJ,+}\otimes U_q^{DJ,0}\otimes U_q^{DJ,-}$ \cite{L93}
 and some automorphism of $U_q^{DJ}$,  the  PBW basis for $U_q^{DJ,+}$ induces a PBW basis for $U_q(\widehat{sl_2})$ \cite[Section 5]{Da}. Then, the explicit isomorphism  $U_q^{Dr}\rightarrow U_q^{DJ}$ \cite{Beck} maps Drinfeld generators to  root vectors.  See \cite[Lemma 1.5]{BCP}, \cite{Da14}. \vspace{1mm}

$\bullet$  To establish the explicit isomorphism between  $U_q^{RS}$ and $U_q^{Dr}$, the main ingredient in  \cite{DF93} is the construction of a FRT presentation for $U_q(\widehat{gl_2})$, which can be interpreted as a central extension of $U_q(\widehat{sl_2})$ \cite{FMu}. In this approach, the defining relations are written in the form of a Yang-Baxter algebra. Namely, two quantum Lax operators $L^\pm(z)$ which entries are generating functions with coefficients in two different subalgebras of $U_q^{Dr}$ are introduced. They satisfy certain functional relations  (the so-called `RTT' relations) characterized by an R-matrix.  The explicit  isomomorphism   $U_q^{RS}\rightarrow U_q^{Dr}$ is obtained as a corollary of the FRT presentation of $U_q(\widehat{gl_2})$. \vspace{1mm}

In these works, Damiani's root vectors  (or equivalently the Drinfeld generators), associated PBW bases and the Yang-Baxter algebra play a central role. Later on, these objects found several applications. For instance, the universal R-matrix is built from elements in PBW bases of $U_q(\widehat{sl_2})$ subalgebras \cite{DaR}. Also, irreducible  finite dimensional representations of $U_q(\widehat{sl_2})$ are classified using $U_q^{Dr}$ \cite{CP}. A natural question is the following: for $U_q(\widehat{sl_2})$, is it possible to construct a different `triplet' of mutually isomorphic algebras other than $U_q^{DJ}$ (or $U_q^{IT}$), $U_q^{Dr}$ and $U_q^{RS}$? \vspace{1mm}

Recent works by Paul Terwilliger  bring a new light on this subject, and give a starting point for a precise answer. Indeed,
in \cite{Ter18,Ter19} Terwilliger investigated the description of  PBW bases of $U_q(\widehat{sl_2})$ from the perspective of combinatorics,  using a $q$-shuffle algebra ${ \mathbb V}$ introduced earlier by Rosso \cite{Ro}. Remarkably, using an injective algebra homomorphism $U_q^{DJ,+}\rightarrow {\mathbb V}$   a  closed form for the images in ${ \mathbb V}$ of Damiani's  root vectors of $U_q^{DJ,+}$ - the basic building elements of Damiani's PBW basis  - was obtained in terms of Catalan words \cite[Theorem 1.7]{Ter18}. Then, in \cite{Ter19}, he introduced a set of elements $\{W_{-k},W_{k+1},G_{k+1},\tilde G_{k+1}|k\in{\mathbb N}  \}$ into the $q$-shuffle algebra  named as `alternating' words. It was shown that the alternating words generate an algebra denoted $U$ \cite[Section 5]{Ter19} for which a PBW basis was constructed \cite[Theorem 10.1,10.2]{Ter19}. Considering the preimage in $U_q^{DJ,+}$ of the  alternating words of $U$, a
 new PBW in  basis  - called alternating -  for $U_q^{DJ,+}$ arises,  besides Damiani's one \cite[Theorem 2]{Da}. 
  A comparison between the images in  ${ \mathbb V}$ of both PBW  bases was done, see \cite[Section 11]{Ter19}.
 More recently  \cite{Ter19b}, a central extension of the preimage of the algebra $U$  arising  from the exchange relations between alternating words, denoted ${\cal U}^+_q$, has been introduced.  Its generators are in bijection with  `alternating' generators recursively  built in $U_q^{DJ,+}$  and form an `alternating' PBW basis   for the new algebra ${\cal U}^+_q$ \cite[Section 10]{Ter19b}. \vspace{1mm}

 In this paper, we  investigate further these new `alternating' algebras motivated by the construction of a  new triplet of presentations for $U_q(\widehat{sl_2})$. To this aim,  following \cite{Ter19b} we introduce  the algebra  $\bar{\cal A}_q$ with  generators  
$ \{{\tW}_{-k}, {\tW}_{k+1},  {\tG}_{k+1},$ $\tilde{\tG}_{k+1}|k\in {\mathbb N}\}$ - see Definition \ref{defAqp}. 
Note that to enable a non-trivial specialization  $q \rightarrow 1$, the definitions of $\bar{\cal A}_q$ and ${\cal U}^+_q$ slightly differ. However,  for $q \neq 1$ $\bar{\cal A}_q$ and ${\cal U}^+_q$ are essentially the same object. Also,   the center ${\cal Z}$ of $\bar{\cal A}_q$ is introduced. Adapting the results of \cite{Ter19b}, the `alternating' PBW basis of $\bar{\cal A}_q$ is given, see Theorem \ref{pbwcAbar}.  Following \cite{Ter19}, similarly we introduce the algebra $\bar{A}_q$ with  generators  
$\{W_{-k},W_{k+1},G_{k+1},\tilde G_{k+1}|k\in{\mathbb N}  \}$.
One has:
\beqa
\bar{\cal A}_q \cong  \bar{A}_q \otimes  {\cal Z} \ .
\eeqa
Let $\langle W_0,W_1 \rangle$ denote the subalgebra of  $\bar{A}_q$ generated by $W_0,W_1$. The simplest   relations satisfied by $W_0,W_1 $ are the $q$-Serre relations   (\ref{qS1}), (\ref{qS2}), of  ${U}^{DJ,+}_q$ - see (\ref{defUqDJp}). Actually, according to \cite{Ter19},  $\bar{A}_q \cong {U}^{DJ,+}_q \cong {U}^{DJ,-}_q$. So, having in mind the isomorphic pair consisting of  $U_q^{DJ,+}$  (or $U_q^{DJ,-}$) and  certain subalgebras of $U_q^{Dr}$ \cite{Beck,BCP},  an  `alternating' isomorphic pair is provided by $\langle \cW_0,\cW_1 \rangle$  and  $\bar{A}_q$.  Furthermore, by analogy with \cite{Beck}, the explicit isomorphism  $\bar{A}_q \rightarrow \langle W_0,W_1 \rangle$ follows from Lemma \ref{lem3} using a map $\gamma: \bar{\cal A}_q \rightarrow  \bar{A}_q $.  Details are reviewed in Section \ref{sec2}. For completeness, the specialization  $q \rightarrow 1$ of $\bar{\cal A}_q$, denoted $\bar{\cal A}$, is also introduced.
\vspace{1mm}

The main result of this paper  is a presentation for $\bar{\cal A}_q$ which sits into the family of Freidel-Maillet type algebras\footnote{See also \cite{NC92,Bab92,KS}.} \cite{FM91} for generic $q$,  see Theorem \ref{thm1}.
For the specialization  $\bar{\cal A}$, a FRT type presentation is obtained. It sits into the family of non-standard Yang-Baxter algebras, see Proposition \ref{prop32}.  
 This is done in Section \ref{secFM}.
This Freidel-Maillet type presentation of $\bar{\cal A}_q$ gives an efficient framework for studying in more details this algebra and clarifying its relation with $U_q^{DJ}$ (or $U_q^{IT}$), $U_q^{Dr}$ and $U_q^{RS}$. The following results are obtained:\vspace{1mm}

(a) Tensor product realizations of $\bar{\cal A}_q$ in $U_q(sl_2)^{\otimes N }$  are explicitly constructed. They generate certains quotients of $\bar{\cal A}_q$, characterized by a set of linear relations satisfied by the fundamental generators. See  Proposition \ref{propKN}. This is done in Section \ref{secTPR}. \vspace{1mm}  

(b) Explicit isomorphisms between $\bar{\cal A}_q$ and certain `alternating' subalgebras of  $U_q(\widehat{gl_2})$,
denoted $U_q(\widehat{gl_2})^{\triangleright,+}$ and $U_q(\widehat{gl_2})^{\triangleleft,-}$, are obtained. See Propositions \ref{map1}, \ref{map2}. The main ingredient in the analysis is the use of the Ding-Frenkel isomorphism \cite{DF93}.
 As a corollary, similar results for  $\bar{ A}_q$ and the `alternating' subalgebras of  $U_q(\widehat{sl_2})$  follow. Also, it is shown that $\bar{ \cal A}_q$ can be regarded  as a left (or right) comodule of alternating subalgebras of $U_q(\widehat{gl_2})$. An example of  coaction map is given in  Lemma \ref{coprodform}. See Example \ref{excop}. 
\vspace{1mm}  

(c) The explicit isomorphism $\iota: \langle \tW_0,\tW_1\rangle \rightarrow U_q^{DJ,+}$ given by  (\ref{mappos}) is extended to the whole set of generators of $\bar{\cal A}_q$: a set of functional equations that determine the explicit relation between  Damiani's root vectors $\{E_{n\delta+\alpha_i}, E_{n\delta}|i=0,1\}\in U_q^{DJ,+}$  (or $\{F_{n\delta+\alpha_i}, F_{n\delta}|i=0,1\}\in U_q^{DJ,-}$) and  the  generators of  $\bar{A}_q$ is derived, see Proposition \ref{prop:Aroot}.
\vspace{1mm}

The results (b) and (c) are given in  Section \ref{secDr}. All together,  if we denote $\bar A_q^{FM}$ as the Freidel-Maillet type presentation of $\bar A_q$, we get the isomorphic `triplet' 
\beqa
U_q^{DJ,+} \cong \bar A_q \cong \bar A_q^{FM}\ .\nonumber
\eeqa

In the last section, we point out a straightforward application of  \cite{Ter19,Ter19b} combined with the results of Section \ref{secDr}. One has the `alternating' tensor product decomposition of $U_q(\widehat{sl_2})$: 
\beqa
U_q(\widehat{sl_2}) \cong\bar{A}_q^\triangleright \otimes  U_q^{DJ,0} \otimes  \bar{A}_q^\triangleleft \  ,\label{newdecintro}
\eeqa
where $\bar{A}_q^{\triangleright(\triangleleft)}(\cong U_q^{DJ,+(-)})$ are certain alternating subalgebras of $U_q^{Dr}$.
The corresponding `alternating' PBW basis is given in Theorem \ref{thmfin}. \vspace{1mm}

Let us conclude this introduction with some additional comments. In the literature, it is known that solutions of  the Yang-Baxter equation find many applications in the theory of quantum integrable systems such as vertex models, spin chains,... They can be obtained by specializing solutions of the universal Yang-Baxter equation, the so-called universal R-matrices. As already mentioned, the construction of a universal R-matrix for $U_q(\widehat{sl_2})$  (and similarly for higher rank cases) essentially relies on the tensor product decomposition
\beqa
U_q(\widehat{sl_2}) \cong  U_q^{DJ,+}  \otimes  U_q^{DJ,0} \otimes  U_q^{DJ,-} , \  
\eeqa
and the use of root vectors  \cite{KiR,KT, DaR,FMu,JLM,JLMBD}. Now, the `alternating' tensor product decomposition (\ref{newdecintro}) rises the question of an `alternating'  universal K-matrix built from a product of  solutions to a universal Freidel-Maillet type equation. See   \cite{CG92,Parm,BKo15,RV16,AV} for related problems.  In view of the importance of the R-matrix in mathematical physics, it looks as an interesting problem that might be considered elsewhere. 
\vspace{1mm}

It should be mentioned that the analysis here presented is also motivated by the subject of the $q$-Onsager algebra $O_q$ \cite{Ter03,Bas2} and its applications to quantum integrable systems. See e.g. \cite{BK14a,BK14b,BB16,BT17,Ts18,Ts19,BP19}. 
The original presentation of $O_q$ is given in terms of generators $A,B$ satisfying a pair of $q$-Dolan-Grady relations. The algebra $\bar{\cal A}_q$ studied in this paper can be viewed as a limiting case of the algebra ${\cal A}_q$ introduced in \cite{BSh1,a-BasBel17p}. For  ${\cal A}_q$, the original presentation \cite{BK} takes the form of a reflection algebra introduced by Sklyanin \cite{Skly88}, see \cite{BSh1}. Let us denote this presentation by ${\cal A}_q^S$. Using ${\cal A}_q^S$,
  it has been conjectured that ${\cal A}_q$ is  a central extension of $O_q$. Initial supporting evidences were based on a comparison between the `zig-zag'  basis of   $O_q$  \cite{IT} and the one conjectured for  ${\cal A}_q$ \cite[Conjecture 1]{a-BasBel17p}. Other evidences are also given in \cite{T21a}. More recently, the conjecture is finally proved \cite{T21b}.  
	So, using a surjective homomorphism ${\cal A}_q \rightarrow O_q$, one gets a triplet of isomorphic algebras $ O_q \cong A_q \cong A_q^S$. 
Independently,  more recently the analog of Lusztig's automorphism and Damiani's root vectors  denoted $B_{n\delta+\alpha_0},B_{n\delta+\alpha_1}, B_{n\delta}$ for  the $q$-Onsager algebra  have been obtained \cite{BKo17} (see also \cite{Ter17}).  In terms of the root vectors, a PBW basis  has been constructed. In addition, a Drinfeld type presentation is now identified \cite{MW}. However, at the moment the precise relation between the presentation of $O_q$ given in  \cite{BKo17} or  its Drinfeld type presentation denoted $O_q^{Dr}$ \cite{MW} and  ${ A}_q$ is yet to  be clarified. 
To prove $O_q \cong { A}_q \cong { A}^S_q$ provides an `alternating' triplet of presentation for the $q$-Onsager algebra and $O_q^{Dr}\cong  { A}_q $, the analysis here presented  sketches the strategy  that may be considered elsewhere.
\vspace{1mm}

Clearly,   alternating subalgebras  for higher rank affine Lie algebras and corresponding generalizations of (\ref{newdecintro})  may be considered as well following a similar approach. \vspace{2mm}

{\bf Notation 1.1.}
{\it
Recall the natural numbers ${\mathbb N} = \{0, 1, 2, \cdots\}$ and integers ${\mathbb Z} = \{0, \pm 1, \pm 2, \cdots\}$.
Let ${\mathbb K}$ denote an algebraically closed field of characteristic $0$.   ${\mathbb K}(q)$ denotes the field of rational functions in an indeterminate $q$. The $q$-commutator $\big[X,Y\big]_q=qXY-q^{-1}YX$ is introduced. We denote $[x]= (q^x-q^{-x})/(q-q^{-1})$.} \vspace{2mm}

{\bf Notation 1.2.}
{\it
 $U_q^{DJ}$ is the Drinfeld-Jimbo presentation of $U_q(\widehat{sl_2})$.
 $U_q^{DJ,+}, U_q^{DJ,0}, U_q^{DJ,-}$ are the subalgebras of $U_q^{DJ}$ generated respectively by $\{E_0,E_1\}$, $\{K_0,K_1\}$, $\{F_0,F_1\}$. 
We also introduce the subalgebras $U_q^{DJ,+,0}$ (resp. $U_q^{DJ,-,0}$) generated by $\{E_0,E_1,K_0,K_1\}$ (resp. $\{F_0,F_1,K_0,K_1\}$).
}

\section{The algebra $\bar{\cal A}_q$ and its specialization $q \rightarrow 1$}\label{sec2}
In this section, the algebra $\bar{\cal A}_q$ and its specialization $q \rightarrow 1$ denoted $\bar{\cal A}$  are introduced. The algebra $\bar{\cal A}_q$ is nothing but a slight modification of the algebra ${\cal U}_q^+$ introduced in \cite[Section 3]{Ter19b}.   Compared with ${\cal U}_q^+$, the modification here considered aims to ensure that the specialization  $q \rightarrow 1$ of $\bar{\cal A}_q$ is non-trivial. Also, the parameter $\bar\rho$ is introduced for normalization convenience. So,  part of the material  in this section is mainly adapted from \cite{Ter19b}. Besides, Lemma \ref{lem1} and Lemma \ref{lem12}  solve \cite[Problem 13.1]{Ter19}.  At the end of this section, we prepare the discussion for Sections \ref{secFM} and \ref{secDr}.

\subsection{Defining relations} 
We refer the  reader to \cite[Definition 3.1]{Ter19b} for the definition of  ${\cal U}_q^+$. We now introduce  the algebra $\bar{\cal A}_q$. 
\begin{defn}\label{defAqp} Let $\bar\rho\in{\mathbb K}(q)$. $\bar{\cal A}_q$ is the associative algebra over ${\mathbb K}(q)$  generated by $ \{{\tW}_{-k}, {\tW}_{k+1},  {\tG}_{k+1}, \tilde{\tG}_{k+1}|k\in {\mathbb N}\}$
subject to the following relations:
\begin{align}
&
 \lbrack  {\tW}_0,  {\tW}_{k+1}\rbrack= 
\lbrack  {\tW}_{-k},  {\tW}_{1}\rbrack=
\frac{({\tilde{\tG}}_{k+1} -  {\tG}_{k+1})}{q+q^{-1}},
\label{def1}
\\
&
\lbrack  {\tW}_0,  {\tG}_{k+1}\rbrack_q= 
\lbrack {{\tilde {\tG}}}_{k+1},  {\tW}_{0}\rbrack_q= 
 \bar\rho{\tW}_{-k-1},
\label{def2}
\\
&
\lbrack {\tG}_{k+1},  {\tW}_{1}\rbrack_q= 
\lbrack  {\tW}_{1}, { {\tilde {\tG}}}_{k+1}\rbrack_q= 
\bar\rho {\tW}_{k+2},
\label{def3}
\\
&
\lbrack  {\tW}_{-k},  {\tW}_{-\ell}\rbrack=0,  \qquad 
\lbrack  {\tW}_{k+1},  {\tW}_{\ell+1}\rbrack= 0,
\label{def4}
\\
&
\lbrack  {\tW}_{-k},  {\tW}_{\ell+1}\rbrack+
\lbrack {\tW}_{k+1},  {\tW}_{-\ell}\rbrack= 0,
\label{def5}
\\
&
\lbrack  {\tW}_{-k},  {\tG}_{\ell+1}\rbrack+
\lbrack {\tG}_{k+1},  {\tW}_{-\ell}\rbrack= 0,
\label{def6}
\\
&
\lbrack {\tW}_{-k},  {\tilde {\tG}}_{\ell+1}\rbrack+
\lbrack  {\tilde {\tG}}_{k+1},  {\tW}_{-\ell}\rbrack= 0,
\label{def7}
\\
&
\lbrack  {\tW}_{k+1},  {\tG}_{\ell+1}\rbrack+
\lbrack   {\tG}_{k+1}, {\tW}_{\ell+1}\rbrack= 0,
\label{def8}
\\
&
\lbrack  {\tW}_{k+1},  {\tilde {\tG}}_{\ell+1}\rbrack+
\lbrack  {\tilde {\tG}}_{k+1},  {\tW}_{\ell+1}\rbrack= 0,
\label{def9}
\\
&
\lbrack  {\tG}_{k+1},  {\tG}_{\ell+1}\rbrack=0,
\qquad 
\lbrack {\tilde {\tG}}_{k+1},  {\tilde {\tG}}_{\ell+1}\rbrack= 0,
\label{def10}
\\
&
\lbrack {\tilde {\tG}}_{k+1},  {\tG}_{\ell+1}\rbrack+
\lbrack  {\tG}_{k+1},  {\tilde {\tG}}_{\ell+1}\rbrack= 0\ .
\label{def11}
\end{align}
\end{defn}

\begin{rem}\label{Rem1} The defining relations of $\bar{\cal A}_q$ coincide with the defining relations  (30)-(40) in \cite{Ter19b} of the algebra ${\cal U}_q^+$  for the identification:
\beqa
&&{\tW}_{-k}\mapsto  \cW_{-k}\ ,\quad  {\tW}_{k+1} \mapsto  \cW_{k+1}\ ,\label{id1}\\
&&{\tG}_{k+1} \mapsto  q^{-1}(q^2-q^{-2})\cG_{k+1}\ ,\quad  {\tilde{\tG}}_{k+1}\mapsto q^{-1}(q^2-q^{-2})\tilde{\cG}_{k+1}\ ,\label{id2}\\
&& \bar\rho\mapsto q^{-1}(q^2-q^{-2})(q-q^{-1})\ .\label{id3}
\eeqa
\end{rem}

Note that there exists an automorphism $\sigma$ and an antiautomorphism $S$ (for ${\cal U}_q^+$, see \cite[Lemma 3.9]{Ter19b})  such that:
\beqa
\sigma: && {\tW}_{-k}\mapsto {\tW}_{k+1}\ ,\quad {\tW}_{k+1}\mapsto {\tW}_{-k}\ ,\quad {\tG}_{k+1}\mapsto \tilde{\tG}_{k+1}\ ,\quad \tilde{\tG}_{k+1}\mapsto {\tG}_{k+1}\ ,\label{sig}\\
S: && {\tW}_{-k}\mapsto {\tW}_{-k}\ ,\quad {\tW}_{k+1}\mapsto {\tW}_{k+1}\ ,\quad {\tG}_{k+1}\mapsto \tilde{\tG}_{k+1}\ ,\quad \tilde{\tG}_{k+1}\mapsto {\tG}_{k+1}\ .\label{autS}
\eeqa
\vspace{1mm}

For completeness  (see \cite[Note 2.6]{Ter19b}) and the discussion in the next section, a set of additional relations can be derived from the defining relations (\ref{def1})-(\ref{def11}), given in Lemmas \ref{lem1}, \ref{lem12} below. 
\begin{lem}
\label{lem1} In  $\bar{\cal A}_q$,  the following
relations hold:
\begin{align}
&\lbrack {\tW}_{-k}, {\tG}_{\ell}\rbrack_q = 
\lbrack {\tW}_{-\ell}, {\tG}_{k}\rbrack_q,
\qquad \quad
\lbrack {\tG}_k, {\tW}_{\ell+1}\rbrack_q = 
\lbrack {\tG}_\ell, {\tW}_{k+1}\rbrack_q,
\label{wg1}
\\
&
\lbrack \tilde {\tG}_k, {\tW}_{-\ell}\rbrack_q = 
\lbrack \tilde G_\ell, {\tW}_{-k}\rbrack_q,
\qquad \quad 
\lbrack {\tW}_{\ell+1}, \tilde {\tG}_{k}\rbrack_q = 
\lbrack {\tW}_{k+1}, \tilde {\tG}_{\ell}\rbrack_q.
\label{wg2}
\end{align}
\end{lem}
\begin{proof}
Consider the first equation in (\ref{wg1}). For convenience, substitute $\ell \rightarrow \ell+1$ and multiply by $\bar\rho$ the equality. From the r.h.s. of the resulting equation, using (\ref{def2}) one has:
\beqa
 \lbrack \underbrace{\bar\rho {\tW}_{-\ell -1}}_{=  \lbrack {\tW}_0, {\tG}_{\ell+1}  \rbrack_q }, {\tG}_{k}\rbrack_q &=& q^2 {\tW}_{0}\underbrace{{\tG}_{\ell+1}{\tG}_{k}}_{={\tG}_{k}{\tG}_{\ell+1}} - {\tG}_{\ell+1}{\tW}_{0}{\tG}_{k} - {\tG}_{k}{\tW}_{0}{\tG}_{\ell+1} +q^{-2}\underbrace{{\tG}_{k}{\tG}_{\ell+1}}_{={\tG}_{\ell+1}{\tG}_{k}}{\tW}_0 \qquad \mbox{by (\ref{def10})} \nonumber\\
&=&  \underbrace{q^2 {\tW}_{0}{\tG}_{k}{\tG}_{\ell+1}}_{=q \lbrack {\tW}_0, {\tG}_{k}\rbrack_q {\tG}_{\ell+1}  +{\tG}_{k}{\tW}_0{\tG}_{\ell+1}   }
-  \!\!\!\!\!\!\!\!\!\!\! \!\!\!\!\!\!\!\!\!\!\!\!\!\!\!\!\!\!\underbrace{ {\tG}_{\ell+1}{\tW}_{0}{\tG}_{k}}_{\quad \qquad\qquad\qquad = q^{-1}{\tG}_{\ell+1}\lbrack {\tW}_0, {\tG}_{k}\rbrack_q + q^{-2} {\tG}_{\ell+1}{\tG}_k{\tW}_0 }
  \!\!\!\!\!\!\!\!\!\!\!\!\!\!\!\!\!\!  \!\!\!\!\!\!\!\!\!\!\! - {\tG}_{k}{\tW}_{0}{\tG}_{\ell+1} +q^{-2}{\tG}_{\ell+1}{\tG}_{k}{\tW}_0  \nonumber\\
&=&    \lbrack  \lbrack {\tW}_0, {\tG}_{k}  \rbrack_q , {\tG}_{\ell+1}  \rbrack_q    \nonumber\\
  &  =&       \bar\rho \lbrack{\tW}_{-k}, {\tG}_{\ell+1}\rbrack_q \nonumber \ ,
\eeqa
which coincides with the l.h.s. The three other equations are shown similarly.
\end{proof}
\begin{lem}
\label{lem12} In  $\bar{\cal A}_q$, the following
relations hold:
\begin{align}
&\lbrack {\tG}_{k}, \tilde  {\tG}_{\ell+1}\rbrack -
\lbrack  {\tG}_{\ell}, \tilde  {\tG}_{k+1}\rbrack =
\bar\rho (q+q^{-1})\left(\lbrack  {\tW}_{-\ell},  {\tW}_{k+1}\rbrack_q-
\lbrack  {\tW}_{-k},  {\tW}_{\ell+1}\rbrack_q\right),
\label{gg1}
\\
&\lbrack \tilde  {\tG}_{k},   {\tG}_{\ell+1}\rbrack -
\lbrack \tilde  {\tG}_{\ell},   {\tG}_{k+1}\rbrack =
\bar\rho (q+q^{-1})\left( \lbrack  {\tW}_{\ell+1},  {\tW}_{-k}\rbrack_q-
\lbrack  {\tW}_{k+1},  {\tW}_{-\ell}\rbrack_q \right),
\label{gg2}
\\
&\lbrack  {\tG}_{k+1}, \tilde  {\tG}_{\ell+1}\rbrack_q -
\lbrack  {\tG}_{\ell+1}, \tilde  {\tG}_{k+1}\rbrack_q =
\bar\rho (q+q^{-1})\left(\lbrack  {\tW}_{-\ell},  {\tW}_{k+2}\rbrack-
\lbrack  {\tW}_{-k},  {\tW}_{\ell+2}\rbrack \right),
\label{gg3}
\\
&\lbrack \tilde  {\tG}_{k+1},   {\tG}_{\ell+1}\rbrack_q -
\lbrack \tilde  {\tG}_{\ell+1},   {\tG}_{k+1}\rbrack_q =
\bar\rho (q+q^{-1})\left( \lbrack  {\tW}_{\ell+1},  {\tW}_{-k-1}\rbrack-
\lbrack  {\tW}_{k+1},  {\tW}_{-\ell-1}\rbrack \right).
\label{gg4}
\end{align}
\end{lem}
\begin{proof} Consider (\ref{gg1}). One has:
\beqa
 \lbrack  {\tG}_{k}, \tilde{\tG}_{\ell+1}\rbrack &=&   \lbrack  {\tG}_{k}, \tilde{\tG}_{\ell+1}- {\tG}_{\ell+1}\rbrack = (q+q^{-1}) \lbrack  {\tG}_{k}, \lbrack {\tW}_0,{\tW}_{\ell+1} \rbrack\rbrack \quad \quad \mbox{by (\ref{def1})}\nonumber\\
&=&  (q+q^{-1})  \left(   \underbrace{{\tG}_{k}{\tW}_0}_{=q^2 {\tW}_0{\tG}_{k} -  \bar\rho q {\tW}_{-k}}\!\!\!\!\!\!\!\!\!\!\! {\tW}_{\ell+1}  -  {\tG}_{k} {\tW}_{\ell+1} {\tW}_{0}  - {\tW}_0 {\tW}_{\ell+1} {\tG}_k  + {\tW}_{\ell+1}\!\!\!\!\!\!\!\!\!\!\!\!\!\!\underbrace{{\tW}_0 {\tG}_{k}}_{=q^{-2}{\tG}_{k} {\tW}_0 +  \bar\rho q^{-1} {\tW}_{-k}}  \right)\nonumber\\
&=&  (q+q^{-1}) \left(  \lbrack {\tW}_0 , \lbrack {\tG}_k,{\tW}_{\ell+1} \rbrack_q \rbrack_q  - \bar\rho \lbrack {\tW}_{-k},{\tW}_{\ell+1} \rbrack_q \right)\nonumber\ .
\eeqa
It follows:
 \beqa
 \lbrack  {\tG}_{k}, \tilde{\tG}_{\ell+1}\rbrack -  \lbrack  {\tG}_{\ell}, \tilde{\tG}_{k+1}\rbrack = \bar\rho (q+q^{-1})\left(  \lbrack {\tW}_{-\ell},{\tW}_{k+1} \rbrack_q  -  \lbrack {\tW}_{-k},{\tW}_{\ell+1} \rbrack_q \right) + (q+q^{-1}) \lbrack {\tW}_0 ,\underbrace{ \lbrack {\tG}_k,{\tW}_{\ell+1}  \rbrack_q  - 
  \lbrack {\tG}_\ell,{\tW}_{k+1}\rbrack_q}_{=0 \ \ \mbox{by (\ref{wg1})}} \rbrack_q \ \nonumber
\eeqa
which reduces to (\ref{gg1}). One shows (\ref{gg2}) similarly.
\vspace{1mm}

Consider (\ref{gg4}). One has:
\beqa
\bar\rho \lbrack  {\tW}_{\ell+1}, {\tW}_{-k-1}\rbrack &=&  \lbrack  {\tW}_{\ell+1}, \lbrack {\tW}_0,{\tG}_{k+1} \rbrack_q \rbrack = q {\tW}_0 \lbrack{\tW}_{\ell+1},{\tG}_{k+1} \rbrack +    q^{-1} \lbrack{\tG}_{k+1},{\tW}_{\ell+1} \rbrack  {\tW}_0 \nonumber\\
&& \qquad \qquad \qquad \qquad \qquad + \frac{q^{-1}}{(q+q^{-1})} {\tG}_{k+1}\tilde{\tG}_{\ell+1}- \frac{q}{(q+q^{-1})} \tilde{\tG}_{\ell+1}{\tG}_{k+1} + \frac{(q-q^{-1})}{(q+q^{-1})}{\tG}_{k+1}{\tG}_{\ell+1}\ ,\nonumber
\eeqa
where (\ref{def2}),  (\ref{def1}) and (\ref{def10}) have been used successively. Using (\ref{def8}) it follows:
\beqa
\bar\rho \left(\lbrack  {\tW}_{\ell+1}, {\tW}_{-k-1}\rbrack - \lbrack  {\tW}_{k+1}, {\tW}_{-\ell-1}\rbrack \right)&=&
 \frac{q^{-1}}{(q+q^{-1})} \left({\tG}_{k+1}\tilde{\tG}_{\ell+1}- {\tG}_{\ell+1}\tilde{\tG}_{k+1}  \right) \label{interww}\\
&&- \frac{q}{(q+q^{-1})} \left(\tilde{\tG}_{\ell+1}{\tG}_{k+1}  - \tilde{\tG}_{k+1}{\tG}_{\ell+1}\right)\ .\nonumber
\eeqa
From (\ref{def11}), note that:
\beqa
\tilde{\tG}_{\ell+1}{\tG}_{k+1}  - \tilde{\tG}_{k+1}{\tG}_{\ell+1} =  {\tG}_{k+1}\tilde{\tG}_{\ell+1}- {\tG}_{\ell+1}\tilde{\tG}_{k+1}  \nonumber
\eeqa
which implies:
\beqa
(q-q^{-1})\left( {\tG}_{\ell+1}\tilde{\tG}_{k+1}  - {\tG}_{k+1}\tilde{\tG}_{\ell+1}\right) = \lbrack \tilde  {\tG}_{k+1},   {\tG}_{\ell+1}\rbrack_q -
\lbrack \tilde  {\tG}_{\ell+1},   {\tG}_{k+1}\rbrack_q \ . \nonumber
\eeqa
Using this last equality in the r.h.s. of  (\ref{interww}), eq. (\ref{gg4}) follows. The other relation (\ref{gg3}) is shown similarly.\vspace{1mm}
 \end{proof}
\begin{rem} The relations  (41)-(46) in \cite{Ter19b} follow from  Lemmas \ref{lem1}, \ref{lem12}, using 
the identification (\ref{id1})-(\ref{id3}).
\end{rem}
\vspace{1mm}

\subsection{The center $\cal Z$ }
For the algebra ${\cal U}_q^+$, central elements denoted $Z_{n+1}^\vee$ are known \cite[eq. (52) and Lemma 5.2]{Ter19b} (see also equivalent expressions \cite[Corollary 8.4]{Ter19b}). With minor modifications using the correspondence (\ref{id1})-(\ref{id3}), central elements in $\bar{\cal A}_q$ are obtained in a straightforward manner. Thus, we omit the proof of the following lemma and refer the reader to \cite[Section 13]{Ter19b} for details.
\begin{lem} For $n\in {\mathbb N}$, the element
\beqa
&&\quad Y_{n+1} =   \tG_{n+1}q^{-n-1} + \tilde{\tG}_{n+1}q^{n+1} - (q^2-q^{-2})\sum_{k=0}^{n} q^{-n+2k}\tW_{-k}\tW_{n+1-k} + \frac{(q-q^{-1})}{\bar\rho}\sum_{k=0}^{n-1} q^{-n+1+2k} \tilde{\tG}_{k+1}\tG_{n-k} \label{Yn}
\eeqa
 is central in $\bar{\cal A}_q$.
\end{lem}
\begin{rem} Central elements for the algebra ${\cal U}_q^+$  \cite[Lemma 5.2, Corollary 8.4]{Ter19b} are obtained using the identification (\ref{id1})-(\ref{id3}):
\beqa
Y_{n+1} \mapsto q^{-1}(q^2-q^{-2})Z_{n+1}^\vee \ .\label{YZ}
\eeqa
\end{rem}
Note that the central elements are fixed under the action of (anti)automorphisms of $\bar{\cal A}_q$. Applying $\sigma$ and $S$ according to (\ref{sig}), (\ref{autS}), three other expressions for  $Y_{n+1}$ follow (for ${\cal U}_q^+$, see \cite[Corollary 8.4]{Ter19b}). In particular, for further convenience, define the combination:
\beqa
\Delta_{n+1} = \frac{1}{q^{n+1}+q^{-n-1}} \left(Y_{n+1} + \sigma(Y_{n+1})\right)\ .\label{deltadef}
\eeqa 
Using (\ref{def5}), one has $S(\Delta_{n+1})=\Delta_{n+1}$. Thus, $\Delta_{n+1}$  is invariant under the action of $\sigma$, $S$.
\begin{example}
\beqa
\qquad \qquad \Delta_1&=& \tG_{1} + \tilde{{\tG}}_{1} -(q-q^ {-1})\big(\tW_0\tW_1+\tW_1\tW_0\big) ,\label{delta1}\\ 
\Delta_2&=&  \tG_{2} + \tilde{{\tG}}_{2} - \frac{(q^2-q^{-2})}{(q^2+q^{-2})}( q^{-1}\tW_0\tW_2 + q\tW_2\tW_0 + q^{-1}\tW_1\tW_{-1} + q\tW_{-1}\tW_{1} ) \,   \label{delta2}\\
&&\qquad\qquad + \frac{(q-q^{-1})}{(q^2+q^{-2})}\,\Big(\frac{\tilde{\tG_1}\tG_1 + \tG_1\tilde{\tG}_1}{\bar\rho}\Big),\nonumber
\eeqa
\beqa
\Delta_3&=&  \tG_{3} + \tilde{{\tG}}_{3}  - \frac{(q-q^{-1})}{(q^2+q^{-2}-1)}( q^{-2}\tW_0\tW_3 + q^2\tW_{3}\tW_0 + q^{-2}\tW_1\tW_{-2} + q^2\tW_{-2}\tW_{1}) \label{delta3}\\
&&\qquad\qquad - \frac{(q-q^{-1})}{(q^2+q^{-2}-1)} (\tW_{2}\tW_{-1} + \tW_{-1}\tW_{2} )   \nonumber\\
&&\qquad\qquad + \frac{(q-q^{-1})}{(q^2+q^{-2}-1)}\,\Big(\frac{\tilde{\tG_2}\tG_1 + \tG_2\tilde{\tG_1}}{\bar\rho}\Big) \ .\nonumber 
\eeqa
\end{example}
 By construction, the elements $\Delta_{n+1}$ are central in  $\bar{\cal A}_q$.
 Let ${\cal Z}$ denote the subalgebra of $\bar{\cal A}_q$ generated by $\{\Delta_{n+1}\}_{n\in {\mathbb N}}$.  By \cite[Proposition 6.2]{Ter19b}, ${\cal Z}$ is the center of $\bar{\cal A}_q$.
\vspace{1mm}

\subsection{Generators and recursive relations}
Following \cite{Ter19b}, combining the defining relations (\ref{def1})-(\ref{def3}) together with (\ref{deltadef}) it follows:
\begin{lem}\label{lem3}  In $\bar{\cal A}_q$, the following recursive relations hold:
\beqa
\tG_{n+1} &=&   \frac{(q^2-q^{-2})}{2(q^{n+1}+q^{-n-1})}\sum_{k=0}^{n} q^{-n+2k}\left(\tW_{-k}\tW_{n+1-k} + \tW_{k+1}\tW_{k-n} \right) \,   \label{recGn}\\
&& - \frac{(q-q^{-1})}{2\bar\rho(q^{n+1}+q^{-n-1})}\sum_{k=0}^{n-1} q^{-n+1+2k}\left(  \tG_{k+1}\tilde{\tG}_{n-k}+  \tilde{\tG}_{k+1}\tG_{n-k} \right)\nonumber\\
 && + \frac{(q+q^{-1})}{2}\big[\tW_{n+1},\tW_0\big] + \frac{1}{2}\Delta_{n+1}\ ,\nonumber
\eeqa
\beqa
\tilde{\tG}_{n+1} &=&  {\tG}_{n+1} + (q+q^{-1})\big[\tW_0,\tW_{n+1}\big] \label{recGtn}  \ ,\\
 {\tW}_{-n-1}&=&\frac{1}{\bar\rho}\big[{\tW}_0,{\tG}_{n+1}\big]_q\ ,\ \label{recWmn} \\
 {\tW}_{n+2}&=&\frac{1}{\bar\rho}\big[{\tG}_{n+1},{\tW}_1\big]_q\ .\ \label{recWnp1}
\eeqa
\end{lem}

Iterating the recursive formulae (\ref{recGn}),  (\ref{recGtn}), (\ref{recWmn}),  (\ref{recWnp1}), given $n$ fixed,  the corresponding generator  is a polynomial in $\tW_0,\tW_1$ and $\{\Delta_{k+1}|k=0,...,n\}$. 
\begin{example}\label{ex1} The first generators read:
\beqa
\quad && \ {\tG}_{1} = q{\tW_1}\tW_0-q^{-1}{\tW_0}\tW_1 + \frac{1}{2}\Delta_1 \ ,\label{defel}\\
&&{\tW}_{-1} = \frac{1}{\bar\rho}\left( (q^2+q^{-2})\tW_0\tW_1\tW_0 -\tW_0^2\tW_1 - \tW_1 \tW_0^2\right) + \frac{1}{2} \frac{\Delta_1(q-q^{-1})}{\bar\rho}\tW_0\ ,\label{Wm1}\\
&&{\tG}_{2} = \frac{1}{\bar\rho(q^2+q^{-2})} \Big( (q^{-3}+q^{-1}) \tW_0^2{\tW_1}^2 - (q^{3}+q){\tW_1}^2\tW_0^2 + (q^{-3}-q^{3})(\tW_0{\tW_1}^2\tW_0 + {\tW_1}\tW_0^2{\tW_1}) \\
&&\qquad \qquad  - (q^{-5}+q^{-3} +2q^{-1}) \tW_0{\tW_1}\tW_0{\tW_1} + (q^{5}+q^{3} +2q){\tW_1}\tW_0{\tW_1}\tW_0 
\Big)  \nonumber \\
&&\qquad \qquad +\ \frac{1}{2}\frac{\Delta_1(q-q^{-1})}{\bar\rho}\big(q{\tW_1}\tW_0-q^{-1}{\tW_0}\tW_1 \big)  - \frac{1}{4}\frac{\Delta_1^2(q-q^{-1})}{\bar\rho(q^2+q^{-2})} +  \frac{1}{2}\Delta_2\ .\nonumber 
\eeqa
Expressions of ${\tilde{\tG}}_{1},{\tW}_{2},{\tilde{\tG}}_{2}$  are obtained using the automorphism $\sigma$.
\end{example}
\begin{cor}\label{cor21} The algebra  $\bar{\cal A}_q$  is generated by $\tW_0, \tW_1$ and ${\cal Z}$.
\end{cor}
\vspace{1mm}

\subsection{ PBW basis}
Following \cite[Lemma 3.10]{Ter19b}, the algebra $\bar{\cal A}_q$ has an $\mathbb N^2$-grading. Define ${\rm deg}: \bar{\cal A}_q \rightarrow \mathbb N \times \mathbb N$.
For the alternating generators one has:
\beqa
&& {\rm deg}(\tW_{-k}) = (k+1,k) \ ,\quad  {\rm deg}(\tW_{k+1}) = (k,k+1)  \ , \nonumber\\
&&   {\rm deg}(\tG_{k+1})={\rm deg}(\tilde\tG_{k+1}) = (k+1,k+1) \ . \nonumber
\eeqa
Note that the expressions in Lemma \ref{lem3} are homogeneous with respect to the grading assigment. The dimension $d_{i,j}$ of the vector space spanned by linearly independent vectors  of the same degree  $(i,j)$  is obtained from the formal power series in the indeterminates $\lambda,\mu$:
\beqa
\Phi (\lambda,\mu) &=&\cal H (\lambda,\mu) \cal \cal Z (\lambda,\mu) \ , \nonumber\\
&=& \sum_{(i,j)\in {\mathbb N}} d_{i,j} \lambda^i\mu^j   \quad \mbox{for $|\lambda|,|\mu|< 1$} \nonumber
\eeqa
with
\beqa
\cal H (\lambda,\mu) &=& \prod_{\ell=1}^\infty \frac{1}{1-\lambda^\ell \mu^{\ell-1}} \frac{1}{1-\lambda^{\ell-1} \mu^{\ell}} \frac{1}{1-\lambda^\ell \mu^{\ell}}\ ,\quad
\cal Z (\lambda,\mu) = \prod_{\ell=1}^\infty \frac{1}{1-\lambda^\ell \mu^{\ell}}\ . \nonumber
\eeqa

In \cite[Section 10]{Ter19b}, a PBW basis for  ${\cal U}_q^+$ is obtained. The proof solely uses the defining relations corresponding to (\ref{def1})-(\ref{def11}). The following theorem is a straightforward adaptation of \cite[Theorem 10.2]{Ter19b}.
\begin{thm}\label{pbwcAbar} (see \cite{Ter19b})  A PBW basis for $\bar{\cal A}_q$ is obtained by its alternating generators
\beqa
\{\tW_{-k}\}_{k\in {\mathbb N}}\ ,\quad \{\tG_{\ell+1}\}_{\ell\in {\mathbb N}}\ ,\quad \{\tilde\tG_{m+1}\}_{m\in {\mathbb N}} \ ,\quad \{\tW_{n+1}\}_{n\in {\mathbb N}}  \nonumber
\eeqa
in any linear order $<$  that satisfies 
\beqa
&&\tW_{-k}<   \tG_{\ell+1}< \tilde\tG_{m+1} < \tW_{n+1}\ , k,\ell,m,n \in {\mathbb N}\ .\nonumber
\eeqa
\end{thm}
Note that combining $\sigma$, $S$ given by (\ref{sig}), (\ref{autS}),  other PBW bases can be obtained. 
\vspace{1mm}

\subsection{ The algebra $\bar{A}_q$ }
By construction \cite{Ter19b}, the algebra $U_q^+$  studied in \cite{Ter19} is a quotient of the algebra $\cal U_q^+$. This quotient is characterized by the fact that the images of all the central elements $Z_n^\vee$ of \cite[Definition 5.1]{Ter19b}  in  $U_q^+$ are vanishing, see \cite[Lemma 2.8]{Ter19b}. Recall (\ref{YZ}),  (\ref{deltadef}). 
\begin{defn}\label{Aq}  
The algebra $\bar{A}_q$ is defined as the quotient of the algebra $\bar{\cal A}_q$ by the ideal generated from the relations $\{\Delta_{k+1}=0|\forall k\in {\mathbb N}\}$. The generators are  $ \{{W}_{-k}, {W}_{k+1},  {G}_{k+1}, \tilde{G}_{k+1}|k\in {\mathbb N}\}$.
\end{defn}
Following \cite[Lemma 3.3]{Ter19b}, let us denote  by $\gamma: \bar{\cal A}_q \rightarrow {\cal A}_q$
 the corresponding surjective homomorphism. It is such that:
\beqa
\gamma: \quad \tW_{-k} \mapsto W_{-k}\ ,\quad {\tW}_{k+1} \mapsto W_{k+1}\ ,\quad {\tG}_{k+1} \mapsto G_{k+1}\ ,\quad \tilde{\tG}_{k+1}\mapsto \tilde G_{n+1}\ .\label{mapgam}
\eeqa
So, they can be obtained as polynomials in $W_0,W_1$ applying $\gamma$ to the expressions given in Lemma \ref{lem3}, where $\gamma(\Delta_{k+1})=0$ for all $k$.

\vspace{1mm}

In \cite{Ter19,Ter19b}, the embedding of $U_q^{DJ,+}$ into  a $q-$shuflle algebra leads to $\bar A_q$, providing an `alternating' presentation for  $U_q^{DJ,+}$. Adapting this result to our conventions, it follows: 
\begin{prop}\label{prop:AqUDJ} (see (\cite{Ter19,Ter19b}) $\bar A_q \cong U_q^{DJ,+} \cong U_q^{DJ,-}$\ .
\end{prop}
In \cite[Section 10]{Ter19}, an alternating' PBW basis for  $U_q^{DJ,+}$ is obtained. We refer to \cite[Theorem 10.1]{Ter19}.
\begin{thm}\label{pbwAbar} (see \cite{Ter19})  A PBW basis for $\bar{A}_q$ is obtained by its alternating generators
\beqa
\{W_{-k}\}_{k\in {\mathbb N}}\ ,\quad \{G_{\ell+1}\}_{\ell\in {\mathbb N}}\ ,\quad \{W_{n+1}\}_{n\in {\mathbb N}}  \nonumber
\eeqa
in any linear order $<$  that satisfies 
\beqa
&& W_{-k}<   G_{\ell+1}<  W_{n+1}\ , k,\ell,n \in {\mathbb N}\ ;\nonumber\\
&& W_{k+1}<   G_{\ell+1}<  W_{-n}\ , k,\ell,n \in {\mathbb N}\ .\nonumber
\eeqa
\end{thm}
Using  automorphisms of $\bar A_q$, other PBW bases can be obtained. 
\vspace{1mm}

\subsection{The specialization $q\rightarrow 1$ and the algebra  $\bar{\cal A}$}
For the specialization  $q\rightarrow 1$, according to the identification  (\ref{id2}), (\ref{id3}),  the defining relations \cite[Definition 3.1]{Ter19b} of the algebra ${\cal U}_q^+$ drastically simplify to those of a commutative algebra. Instead, the specialization  $q\rightarrow 1$ of the defining relations of the algebra $\bar{\cal A}_q$ lead to an associative algebra called  $\bar{\cal A}$, as explained below. To define properly the specialization, we follow the  method described in e.g.  \cite[Section 10]{Kolb} (see also references therein).\vspace{1mm}

Let $\textsf{A}= {\mathbb K}\big[q\big]_{q-1}$ ($={\cal S}^{-1}{\mathbb K}\big[q\big]$ where ${\cal S}={\mathbb K}\big[q\big]\backslash(q-1)$). 
Let ${\cal U}_\textsf{A}$ be the $\textsf{A}$-subalgebra of $\bar{\cal A}_q$ generated by   $\{{\tW}_{-k},{\tW}_{k+1},{\tG}_{k+1},\tilde{\tG}_{k+1}|k\in {\mathbb N}\}$. Note that contrary to $U_q(\widehat{sl_2})$ \cite[page 289]{CPb}, according to the structure of the defining relations (\ref{def1})-(\ref{def11}) for the specialization $q\rightarrow 1$ of $\bar{\cal A}_q$ there is no need to introduce other generators. One has the natural isomorphism of $\textsf{A}$-algebras \ 
${\cal U}_\textsf{A}\otimes_\textsf{A} {\mathbb K}(q) \rightarrow \bar{\cal A}_q$. Consider ${\mathbb K}$ as an $\textsf{A}$-module via evaluation at $q=1$. 
The algebra 
\beqa
{\cal U}_1 = {\cal U}_\textsf{A}\otimes_\textsf{A} {\mathbb K}  \nonumber
\eeqa
is the specialization of $\bar{\cal A}_q$ at $q=1$. Similarly, one defines ${\cal Z}_\textsf{A}$, and ${\cal Z}_1 = {\cal Z}_\textsf{A}\otimes_\textsf{A} {\mathbb K}$.
\begin{defn}\label{def:Acl} 
$\bar{\cal A}$ is the associative algebra  over ${\mathbb K}$ with unit and generators $\{{\tw}_{-k},{\tw}_{k+1},{\tg}_{k+1},\tilde{\tg}_{k+1}|k\in {\mathbb N}\}$ satisfying the following relations:
\begin{align}
&  \big[{\tw}_{-\ell},{\tw}_{k+1}\big]=\frac{1}{2}(\tilde{\tg}_{k+\ell+1}-{\tg}_{k+\ell+1})\ ,\label{po1}\\
&\big[{\tilde{\tg}}_{k+1},{\tw}_{-l}\big]= \big[{\tw}_{-l},{\tg}_{k+1}\big]=16{\tw}_{-k-\ell-1}\ ,\label{po2}\\
&\big[{\tw}_{\ell+1},{\tilde{\tg}}_{k+1}\big]=\big[{\tg}_{k+1},{\tw}_{\ell+1}\big]=16{\tw}_{\ell+k+2}\ ,\label{po3}\\
&\big[{\tw}_{-k},{\tw}_{-\ell}\big]=0\ ,\quad 
\big[{\tw}_{k+1},{\tw}_{\ell+1}\big]=0\ ,\quad  \big[{\tg}_{k+1},{\tg}_{\ell+1}\big]=0\ ,\quad \big[{\tilde{\tg}}_{k+1},\tilde{{\tg}}_{\ell+1}\big]=0\ .\label{po4}\quad 
\end{align}
\end{defn}
\begin{rem} An overall parameter $\bar\rho_c\in {\mathbb K}^*$ may be introduced in the r.h.s. of (\ref{po2}), (\ref{po3}). 
\end{rem}
\begin{prop}\label{mapspe} There exists an algebra isomorphism $ {\cal U}_1 \rightarrow  \bar{\cal A}$ such that:
\beqa
\tW_{-k}\mapsto {\tw}_{-k}  \ ,\quad \tW_{k+1}\mapsto {\tw}_{k+1}\ ,\quad \tG_{k+1}\mapsto {\tg}_{k+1} \ ,\quad \tilde \tG_{k+1}\mapsto \tilde{\tg}_{k+1} \ ,\quad \bar\rho \mapsto 16\ ,  \quad q \mapsto 1  \ .     \label{mapc}
\eeqa
\end{prop}
\begin{proof}  First, we show how to obtain the defining relations for  $\bar{\cal A}$
  from those of  $\bar{\cal A}_q$  at $q=1$ and $\bar\rho= 16$. From eqs.  (\ref{def4}), (\ref{def10}), one  immediatly obtains the four equations in (\ref{po4}).
From (\ref{deltadef}),  one  gets  
\beqa
\delta_{k+1} = {\tg}_{k+1} + \tilde{\tg}_{k+1} \ ,\label{deltacl}
\eeqa
where $\{\delta_k\}_{k\in{\mathbb N}}$ are central with respect to the algebra generated by $\{{\tw}_{-k},{\tw}_{k+1},{\tg}_{k+1},\tilde{\tg}_{k+1}|k\in {\mathbb N}\}$.  This implies  the first equalities in (\ref{po2}),  (\ref{po3}).
The second equalities in (\ref{po2}),  (\ref{po3}) are obtained from elementary computation using the Jacobi identity together with (\ref{def5})-(\ref{def10}) and (\ref{def1})-(\ref{def4}). For instance:
\beqa
 \big[{\tw}_{-1},{\tw}_{k+1}\big] &=& \frac{1}{16} \big[ \big[{\tw}_{0},{\tg}_{1}\big] ,{\tw}_{k+1}\big] =  - \frac{1}{16} \big[ \underbrace{\big[{\tg}_{1},{\tw}_{k+1} \big]}_{\!\!\!\!\!\!\!\!\!\!\!\!\!\!\!\!\!\!\!\!\!\!\!\!=  \big[{\tg}_{k+1},{\tw}_{0}\big]=16{\tw}_{k+2}} ,{\tw}_{0}\big]  - \frac{1}{16} \big[ \underbrace{\big[ {\tw}_{k+1},{\tw}_{0}\big] }_{\!\!\!\!\!\!\!\!\!\!\!\!\!\!\!\!\!\!\!= -\frac{1}{2}(\tilde{\tg}_{k+1}-{\tg}_{k+1})},{\tg}_{1}\big]= \big[{\tw}_{0},{\tw}_{k+2}\big] \nonumber\\
&=&   \frac{1}{2}(\tilde{\tg}_{k+2}-{\tg}_{k+2})\ .\nonumber
\eeqa
By induction, it follows:
\beqa
\big[{\tw}_{-\ell},{\tw}_{k+1}\big] &=& \big[{\tw}_{-\ell+1},{\tw}_{k+2}\big]= \cdots = \big[{\tw}_{0},{\tw}_{k+\ell + 1}\big] =\frac{1}{2}(\tilde{\tg}_{k+\ell+1}-{\tg}_{k+\ell+1}) \ .\nonumber
\eeqa
Similarly, by induction one easily finds:
\beqa
\big[\tilde {\tg}_{k+1},{\tw}_{-\ell}\big] &=& \big[\tilde {\tg}_{k+2},{\tw}_{-\ell+1}\big]= \cdots =  \big[\tilde {\tg}_{k+\ell +1},{\tw}_{0}\big] = 16{\tw}_{-k-\ell-1}\ ,\nonumber\\
\big[{\tw}_{\ell+1},\tilde {\tg}_{k+1}\big] &=&  \big[{\tw}_{\ell},\tilde {\tg}_{k+2}\big]  = \cdots =   \big[{\tw}_{1},\tilde {\tg}_{k+\ell +1}\big] = 16{\tw}_{\ell+k+2}
\ .\nonumber
\eeqa
Thus, the defining relations (\ref{po1})-(\ref{po4}) of  $\bar{\cal A}$ are recovered from the specialization $q\rightarrow 1$, $\bar\rho\rightarrow 16$ of the defining relations (\ref{def1})-(\ref{def11}) of  $\bar{\cal A}_q$. The converse statement is easily checked. 
\end{proof}
In the following, we call  $\bar{\cal A}$ the specialization $q \rightarrow 1$ of  $\bar{\cal A}_q$.

\subsection{Relation with $U_q^{DJ,\pm}$ and specialization}
 The following comments give some motivation for Sections \ref{secFM} and \ref{secDr}. 
We first  describe the relation  between  $\bar{\cal A}_q$ and   $U_q(\widehat{sl_2})$ with respect to the Drinfeld-Jimbo presentation, adapting directly the results of   \cite{Ter19b}. On one hand, recall that the defining relations for $U_q^{DJ,+}$, $U_q^{DJ,-}$ are respectively given by (\ref{defUqDJp}),  (\ref{defUqDJm}). On the other hand,  inserting (\ref{Wm1}) in (\ref{def4}) for $k=0,\ell=1$  one finds that  ${\tW_0},\tW_1$ satisfy  the $q$-Serre relations:
\beqa
 \lbrack \tW_0 , \lbrack \tW_0,  \lbrack \tW_0 ,\tW_1  \rbrack_q  \rbrack_{q^{-1}} \rbrack &=& 0\ ,\label{qS1}\\
\lbrack  \tW_1,  \lbrack \tW_1,  \lbrack \tW_1, \tW_0  \rbrack_q  \rbrack_{q^{-1}} \rbrack &=& 0\ .\label{qS2}
\eeqa
Let $\langle\tW_0,\tW_1 \rangle$ denote the subalgebra of $\bar{\cal A}_q$ generated by $\tW_0,\tW_1$.
According to \cite[Proposition 6.4]{Ter19b} combined with Remark \ref{Rem1}, it follows that the map $\langle \tW_0,\tW_1\rangle \rightarrow U_q^{DJ,+}$ :
\beqa
\tW_0 \mapsto E_1\ ,\qquad \tW_1 \mapsto E_0\  \label{mappos}
\eeqa
is an algebra isomorphism. Obviously, a similar statement holds for $U_q^{DJ,-}$. Let $ \cal Z ^+$ denote the image of  $ \cal Z $ by the map (\ref{mappos}), and similarly $ \cal Z^-$ the image  associated with the negative part. In both cases, it is a polynomial algebra \cite[Section 4]{Ter19b}.
 Adapting  \cite[Proposition 6.5]{Ter19b} and using Remark \ref{Rem1}, by Corollary \ref{cor21} one concludes: 
\beqa
\bar{\cal A}_q \cong  U_q^{DJ,+} \otimes \cal Z^+  \cong U_q^{DJ,-} \otimes \cal Z^- \ .\label{AqpdefUqp}
\eeqa
For this reason, $\bar{\cal A}_q$ is called the central extension of  $U_q^{DJ,+}$ (or $U_q^{DJ,-}$).\vspace{1mm}

Let us also add the following comment. In view of the isomorphism   (\ref{mappos}),   $\bar{\cal A}_q$ can be equipped with a comodule structure  \cite{CPb}. For instance, examples of left (or right) coaction maps can be considered for the subalgebra $\langle \tW_0,\tW_1 \rangle$. Define  the `left' coaction such that
\beqa
 \bar{\cal A}_q \rightarrow U_q^{DJ,+,0} \otimes \bar{\cal A}_q\ .\label{Hisopart}
\eeqa
Consider its restriction  to $\langle \tW_0,\tW_1 \rangle \cong U_q^{DJ,+}$. As an example of coaction, we may consider:
\beqa
\tW_0 \rightarrow E_0 \otimes  1 + K_0 \otimes \tW_0\ , \label{coprod0A1} \\
\tW_1 \rightarrow E_1 \otimes  1 + K_1 \otimes \tW_1\ .\label{coprod0A2} 
\eeqa  
A `right' coaction could be introduced similarly, as well as a coaction \  $ \bar{\cal A}_q \rightarrow  U_q^{DJ,-,0} \otimes \bar{\cal A}_q$\ . In Section \ref{secDr}, a comodule algebra homomorphism $\delta$  is obtained, see Lemma \ref{coprodform}.\vspace{1mm}

The relation between $\bar{\cal A}$ and the Lie algebra $\widehat{sl_2}^{SC}$  can be considered  through specialization. Recall the isomorphism ${\cal U}_\textsf{A} \otimes_\textsf{A} {\mathbb K}(q) \rightarrow \bar{\cal A}_q$ and similarly for $\langle \tW_0,\tW_1\rangle$ and $\cal Z$. One has  the injection    $\langle \tW_0,\tW_1\rangle_\textsf{A} \otimes_\textsf{A}  \cal Z_\textsf{A}\rightarrow  {\cal U}_\textsf{A}$ by \cite[Lemma 10.6]{Kolb}. By Lemma \ref{lem3}, the latter map is also surjective. Using the fact that $\langle \tW_0,\tW_1\rangle_\textsf{A}$ and  $\cal Z_\textsf{A}$ are free $\textsf{A}$-modules, one calculates:
\beqa
{\cal U}_1 ={\cal U}_\textsf{A} \otimes_\textsf{A} {\mathbb K} &=& \left(  \langle \tW_0,\tW_1\rangle_\textsf{A}  \otimes_\textsf{A} \cal Z_\textsf{A}   \right)       \otimes_\textsf{A}   {\mathbb K} \nonumber\\
&=& \left(  \langle \tW_0,\tW_1\rangle_\textsf{A} \otimes_\textsf{A} {\mathbb K}\right)  \otimes_{\mathbb K}    \left( \cal Z_\textsf{A} \otimes_\textsf{A} {\mathbb K}\right) \ . \nonumber
\eeqa
Let $\langle\tw_0,\tw_1 \rangle$ denote the subalgebra of $\bar{\cal A}$. By Proposition \ref{mapspe} one has $\langle\tw_0,\tw_1 \rangle \cong   \langle \tW_0,\tW_1\rangle_\textsf{A} \otimes_\textsf{A} {\mathbb K}$.  The generators  ${\tw}_{0},{\tw}_{1}$ satisfy the Serre relations (i.e. (\ref{qS1})-(\ref{qS2}) for $q= 1$). Recall  the Lie algebra $\widehat{sl_2}^{SC}$ in the Serre-Chevalley presentation of $\widehat{sl_2}$ with defining relations reported in Appendix \ref{apA}. Denote $\widehat{sl_2}^{SC,+}$ (resp. $\widehat{sl_2}^{SC,-})$ the subalgebra generated by $\{e_0,e_1\}$ (resp. $\{f_0,f_1\})$.  Combining the isomorphism (\ref{mappos}) and the well-known result about the specialization  $q\rightarrow 1$ of $ U_q^{DJ,+}$ given by  $U(\widehat{sl_2}^{SC,+})$, it follows that
the map $\langle\tw_0,\tw_1 \rangle \rightarrow U(\widehat{sl_2}^{SC,+}) $ is an isomorphism. Also, $\cal Z$ is a polynomial ring in the $\{\Delta_{k+1}\}_{k\in{\mathbb N}}$.
 ${\cal Z}_1=\cal Z_\textsf{A} \otimes_\textsf{A} {\mathbb K}=U(\textsf{z})$ where $\textsf{z}$ is the linear span of  $\{\delta_{k+1}\}_{k\in{\mathbb N}}$,  see (\ref{deltacl}). Denote  $\textsf{z}^\pm$ the images of  $\textsf{z}$ in $\widehat{sl_2}^{SC,\pm}$. 
 It follows:
\beqa
\bar{\cal A} \cong U\left( \widehat{sl_2}^{SC,+}   \oplus  \textsc{z}^+\right)  \cong U\left(\widehat{sl_2}^{SC,-}   \oplus  \textsc{z}^-\right)\ . \label{Apdefsl2p}
\eeqa

The structure of the isomorphisms (\ref{AqpdefUqp}) and (\ref{Apdefsl2p}) suggests a close relationship between $\bar{\cal A}_q$ (resp. $\bar{\cal A}$) and certain subalgebras of the quantum universal enveloping algebra $U_q(\widehat{gl_2})$ (resp.  its specialization $U(\widehat{gl_2})$). To clarify this relation in Section \ref{secDr}, a new presentation for  $\bar{\cal A}_q$ (and $\bar{\cal A}$) is given in the next section.

\section{A Freidel-Maillet type presentation for $\bar{\cal A}_q$ and its specialization $q\rightarrow 1$}\label{secFM}
In this section, it is shown that the algebra $\bar{\cal A}_q$ introduced in Definition \ref{defAqp} admits a presentation in the form of a  $K$-matrix satisfying the defining relations of a quadratic algebra within the family introduced by Freidel and Maillet \cite{FM91}, see Theorem \ref{thm1}. In this framework, by Theorem \ref{thm1} and Proposition \ref{propqdet}, several results obtained in \cite{Ter19b} for ${\cal U}_q^+$ are derived in a straightforward manner. For the specialization $q\rightarrow 1$, a presentation of  the Lie algebra $\bar{\cal A}$ - see Definition \ref{def:Acl} - is obtained in terms of a non-standard classical Yang-Baxter algebra, see Proposition \ref{prop32}.
 
\subsection{A quadratic algebra of Freidel-Maillet type} Let  $R(u)$ be the intertwining operator (called quantum $R-$matrix) between the tensor product of two fundamental representations ${\cal V}_1 \otimes {\cal V}_2$ for ${\cal V}={\mathbb C}^2$ associated with the algebra $U_q(\widehat{sl_2})$. The element $R(u)$ depends on the deformation parameter $q$ and is defined by \cite{Baxter}
\begin{align}
R(u) =\left(
\begin{array}{cccc} 
 uq -  u^{-1}q^{-1}    & 0 & 0 & 0 \\
0  &  u -  u^{-1} & q-q^{-1} & 0 \\
0  &  q-q^{-1} & u -  u^{-1} &  0 \\
0 & 0 & 0 & uq -  u^{-1}q^{-1}
\end{array} \right) \ ,\label{R}
\end{align}
where $u$  is an indeterminate,  called  `spectral parameter' in the literature on integrable systems. It is known that $R(u)$ satisfies the quantum Yang-Baxter equation in the space ${\cal V}_1\otimes {\cal V}_2\otimes {\cal V}_3$. Using the standard notation 
\beqa
R_{ij}(u)\in \mathrm{End}({\cal V}_i\otimes {\cal V}_j),\label{notR}
\eeqa
the Yang-Baxter equation reads 
\begin{align}
R_{12}(u/v)R_{13}(u)R_{23}(v)=R_{23}(v)R_{13}(u)R_{12}(u/v)\ .\label{YB}
\end{align}
 As usual, intoduce the permutation operator  $P=R(1)/(q-q^{-1})$. Here, note that  $R_{12}(u)=PR_{12}(u)P=R_{21}(u)$.\vspace{1mm}

We now show that the algebra $\bar{\cal A}_q$ is isomorphic to a quadratic algebra of Freidel-Maillet type \cite{FM91}, which can be viewed as a limiting case of the  standard quantum  reflection equation (also called the boundary quantum Yang-Baxter equation) introduced in the context of boundary quantum inverse scattering theory \cite{Cher,Skly88}. In addition to (\ref{R}), define:
\beqa
 R^{(0)}=diag(1,q^{-1},q^{-1},1)\ .\label{R0}
\eeqa
Define the generating functions:
\begin{align}
{\cW}_+(u)=\sum_{k\in {\mathbb N}}{\tW}_{-k}U^{-k-1} \ , \quad {\cW}_-(u)=\sum_{k\in  {\mathbb N}}{\tW}_{k+1}U^{-k-1} \ ,\label{c1}\\
 \quad {\cG}_+(u)=\sum_{k\in {\mathbb N}}{\tG}_{k+1}U^{-k-1} \ , \quad {\cG}_-(u)=\sum_{k\in {\mathbb N}}\tilde{{\tG}}_{k+1}U^{-k-1} \ ,\label{c2}
\end{align}
where  the shorthand notation $U=qu^2/(q+q^{-1})$
is used.  Let $k_\pm$ be non-zero scalars in ${\mathbb K}(q)$ such that
\beqa
\bar\rho=k_+k_-(q+q^{-1})^2\ .\label{rho}
\eeqa
\begin{thm}\label{thm1} The algebra $\bar{\cal A}_q$ has a presentation of Freidel-Maillet type.
 Let $K(u)$ be a square matrix such that 

\beqa
&&  K(u)=
       \begin{pmatrix}
      uq \cW_+(u) &  \frac{1}{k_-(q+q^{-1})}\cG_+(u) + \frac{k_+(q+q^{-1})}{(q-q^{-1})} \\
     \frac{1}{k_+(q+q^{-1})}\cG_-(u) + \frac{k_-(q+q^{-1})}{(q-q^{-1})}  & uq \cW_-(u) 
      \end{pmatrix} \label{K}
\eeqa
with (\ref{c1})-(\ref{c2}).
The defining relations are given by:
\begin{align} R(u/v)\ (K(u)\otimes I\!\!I)\ R^{(0)}\ (I\!\!I \otimes K(v))\
= \ (I\!\!I \otimes K(v))\  R^{(0)}\ (K(u)\otimes I\!\!I)\ R(u/v)\ .
\label{RE} \end{align}
\end{thm}
\begin{proof}
Inserting (\ref{K}) into (\ref{RE}), the system of (sixteen in total) independent equations for the entries $(K(u))_{ij}$ coming from the Freidel-Maillet type quadratic algebra (\ref{RE}) leads to a system of commutation relations between the generating functions  $\cW_\pm(u),\cG_\pm(u)$. Using the identification (\ref{rho}), after simplifications these commutation relations read:
\begin{align}
&&\big[{\cW}_\pm(u),{\cW}_\pm(v)\big]=0\ ,\qquad\qquad\qquad\qquad\qquad\qquad\qquad\label{ec1}\\
&&\big[{\cW}_+(u),{\cW}_-(v)\big]+\big[{\cW}_-(u),{\cW}_+(v)\big]=0\ ,\qquad\qquad\qquad\qquad\qquad\qquad\qquad\label{ec3}\\
&&(U-V)\big[{\cW}_\pm(u),{\cW}_\mp(v)\big]= \frac{(q-q^{-1})}{\bar\rho(q+q^{-1})}\left({\cG}_\pm(u){\cG}_\mp(v)-{\cG}_\pm(v){\cG}_\mp(u)\right)\qquad\qquad\qquad\label{ec4}\\
&& \qquad \qquad\qquad\qquad\qquad\qquad\qquad\qquad+ \frac{1}{(q+q^{-1})} \big({\cG}_\pm(u)-{\cG}_\mp(u)+{\cG}_\mp(v)-{\cG}_\pm(v)\big)\ ,\nonumber
\end{align}
\beqa
&&
(U-V)\big[{\cG}_\pm(u),{\cG}_\mp(v)\big]=\bar\rho(q^2-q^{-2})UV\big({\cW}_\pm(u){\cW}_\mp(v)-{\cW}_\pm(v){\cW}_\mp(u)\big)\ ,\label{ec5}\\
&&U\big[{\cG}_\mp(v),{\cW}_\pm(u)\big]_q -V\big[{\cG}_\mp(u),{\cW}_\pm(v)\big]_q + \ \bar\rho \big(U{\cW}_\pm(u)-V{\cW}_\pm(v)\big)=0\ ,\label{ec6}\\
&&U\big[{\cW}_\mp(u),{\cG}_\mp(v)\big]_q -V\big[{\cW}_\mp(v),{\cG}_\mp(u)\big]_q+  \ \bar\rho \big(U{\cW}_\mp(u)-V{\cW}_\mp(v)\big)=0\ ,\label{ec7}\\
&&\big[{\cG}_\epsilon(u),{\cW}_\pm(v)\big]+\big[{\cW}_\pm(u),{\cG}_\epsilon(v)\big]=0 \ ,\quad \forall \epsilon=\pm\label{ec8}\ ,\qquad\qquad\qquad\qquad\qquad\qquad\qquad\\
&&\big[{\cG}_\pm(u),{\cG}_\pm(v)\big]=0\ ,\label{ec9}\qquad\qquad\qquad\qquad\qquad\qquad\qquad\\ 
&&\big[{\cG}_+(u),{\cG}_-(v)\big]+\big[{\cG}_-(u),{\cG}_+(v)\big]=0\ .\qquad\qquad\qquad\qquad\qquad\qquad\qquad\ \label{ec16}
\eeqa
The commutation relations among the generators of  $\bar{\cal A}_q$ are now extracted.   Inserting (\ref{c1}), (\ref{c2}) into (\ref{ec1})-(\ref{ec16}), expanding and identifying terms of same order in $U^{-k}V^{-l}$  one finds equivalently the set of defining relations (\ref{def1})-(\ref{def11}) together with the set of relations (\ref{wg1}), (\ref{wg2}) and (\ref{gg1})-(\ref{gg4}) as we now show in details. Precisely, inserting (\ref{c1}) into (\ref{ec1}), (\ref{ec3}), one gets (\ref{def4}), (\ref{def5}), respectively.  Inserting (\ref{c1}),  (\ref{c2}) into (\ref{ec4}), one gets   (\ref{def1}),  (\ref{gg3}), (\ref{gg4}).
Inserting (\ref{c1}),  (\ref{c2}) into (\ref{ec5}), one gets   (\ref{gg1}), (\ref{gg2}).  Inserting (\ref{c1}),  (\ref{c2}) into (\ref{ec6}) and (\ref{ec7}), one gets   (\ref{def2}),  (\ref{def3}) as well as  (\ref{wg1}), (\ref{wg2}).  Inserting (\ref{c1}),  (\ref{c2}) into (\ref{ec8})-(\ref{ec16}), one gets  (\ref{def5})-(\ref{def11}).  As the  relations (\ref{wg1}), (\ref{wg2}) and (\ref{gg1})-(\ref{gg4}) follow from  the defining relations  (\ref{def1})-(\ref{def11}) by Lemmas \ref{lem1}, \ref{lem12}, it follows that  the Freidel-Maillet type algebra (\ref{RE}) is isomorphic to $\bar{\cal A}_q$.
\end{proof}

\begin{rem}  The relations  (\ref{ec1})-(\ref{ec16})  coincide with the relations \cite[Lemmas 13.3,13.4]{Ter19b} in the algebra ${\cal U}_q^+$  for the identification:
\beqa
&& U \mapsto t^{-1}\ ,\quad V \mapsto s^{-1}\ ,\label{idc0}\\
&&{\cW}_\pm(u)  \mapsto  t \cW^{\mp}(t)\ ,\quad\ {\cW}_\pm(v)   \mapsto s \cW^{\mp}(s)\ ,\label{idc1}\\
&& {\cG}_+(u)  \mapsto  q^{-1}(q^2-q^{-2}) ( \cG(t)-1)\  ,\quad  {\cG}_-(u)   \mapsto  q^{-1}(q^2-q^{-2}) ( \tilde\cG(t)-1)\ ,\label{idc2}\\
&& {\cG}_+(v)   \mapsto  q^{-1}(q^2-q^{-2}) ( \cG(s)-1)\  ,\quad  {\cG}_-(v)  \mapsto  q^{-1}(q^2-q^{-2}) ( \tilde\cG(s)-1)\ ,\label{idc2}\\
&& \bar\rho \mapsto q^{-1}(q^2-q^{-2})(q-q^{-1})\ .\label{idc3}
\eeqa
\end{rem}

For completeness, let us mention that an alternative presentation of $\bar{\cal A}_q$ can be considered instead, that involves power series in $u$ in the opposite direction. Indeed, consider the system of relations (\ref{ec1})-(\ref{ec16}) with (\ref{c1})-(\ref{c2}). Applying the transformation:
\beqa
{\cW}_\pm(u)&\mapsto& - {\cW}_\mp(u^{-1}q^{-1})\ ,\qquad {\cG}_\pm(u)\mapsto -{\cG}_\pm(u^{-1}q^{-1}) \ ,\nonumber\\
u &\mapsto& u^{-1}\ ,\quad q \mapsto q^{-1}\ ,\nonumber
\eeqa
and similarly for $u \rightarrow v$,
one finds that 
\beqa
&&  K'(u)=
       \begin{pmatrix}
      u^{-1}q^{-1} \cW_-(u^{-1}q^{-1}) &  \frac{1}{k_-(q+q^{-1})}\cG_+(u^{-1}q^{-1}) + \frac{k_+(q+q^{-1})}{(q-q^{-1})} \\
     \frac{1}{k_+(q+q^{-1})}\cG_-(u^{-1}q^{-1}) + \frac{k_-(q+q^{-1})}{(q-q^{-1})}  & u^{-1}q^{-1} \cW_+(u^{-1}q^{-1}) 
      \end{pmatrix} \label{K'}
\eeqa
satisfies the Freidel-Maillet type equation: 
\begin{align} R(u/v)\ (K'(u)\otimes I\!\!I)\ (R^{(0)})^{-1}\ (I\!\!I \otimes K'(v))\
= \ (I\!\!I \otimes K'(v))\   (R^{(0)})^{-1}\ (K'(u)\otimes I\!\!I)\ R(u/v)\ .
\label{REp} \end{align}
This second presentation of $\bar{\cal A}_q$ will be used in Section \ref{secDr}.

\subsection{Central elements}
For the Freidel-Maillet type algebra  (\ref{RE}), central elements can be derived from the so-called Sklyanin determinant by analogy with \cite[Proposition 5]{Skly88}. Define $P^-_{12}=(1-P)/2$. As usual, below `$\rm tr_{12}$' stands for the trace over $\cal V_1 \otimes \cal V_2$.
\begin{prop}\label{propqdet} Let $K(u)$ be a solution of (\ref{RE}). The quantum determinant
\beqa
\Gamma(u)=\tr_{12}\big(P^{-}_{12}(K(u)\otimes I\!\!I)\ R^{(0)} (I\!\!I \otimes K(uq))\big) \ ,\ \label{gamma}
\eeqa
is  such that  $\big[\Gamma(u),(K(v))_{ij}\big]=0$. 
\end{prop}
\begin{proof} Recall the notation (\ref{notR}). Introduce the vector space ${\cal V}_0$. With respect to the tensor product  ${\cal V}_0 \otimes {\cal V}_1\otimes {\cal V}_2$, we denote:
\beqa
K_0(u)=K(u)\otimes I\!\!I\otimes I\!\!I\ ,\qquad K_1(u)= I\!\!I \otimes K(u)\otimes I\!\!I\ ,\qquad K_2(u)=I\!\!I \otimes I\!\!I \otimes K(u)\ . 
\eeqa
Consider the product $(a) \equiv K_0(v)\Gamma(u)$: 
\beqa
(a) &=& K_0(v) \tr_{12}\big(P^{-}_{12}K_1(u)\ R_{12}^{(0)} K_2(uq)\big)\ ,\nonumber\\
&=&   qK_0(v) \tr_{12}\big(P^{-}_{12}R_{01}^{(0)}R_{02}^{(0)}K_1(u)\ R_{12}^{(0)} K_2(uq)\big) \quad (\mbox{using}\quad P^{-}_{12} = q P^{-}_{12}R_{01}^{(0)}R_{02}^{(0)}) \nonumber\\
&=&   q\tr_{12}\big(P^{-}_{12} K_0(v)R_{01}^{(0)}R_{02}^{(0)}K_1(u)\ R_{12}^{(0)} K_2(uq)\big) \quad (\mbox{using}\quad \lbrack  K_0(v),P^{-}_{12}\rbrack=0)\nonumber \\
&=&   q \tr_{12}\big(P^{-}_{12} K_0(v)R_{01}^{(0)}K_1(u)R_{02}^{(0)}\ R_{12}^{(0)} K_2(uq)\big) \quad (\mbox{using}\quad \lbrack  K_1(u),R^{(0)}_{02}\rbrack=0)\nonumber \\
&=&  q \tr_{12}\big(P^{-}_{12} R^{-1}_{01}(v/u) K_1(u)R^{(0)}_{01}K_0(v)R_{01}(v/u) R_{02}^{(0)}\ R_{12}^{(0)} K_2(uq)\big)\quad (\mbox{using (\ref{RE})})\ .\nonumber
\eeqa 
Then we use  $\lbrack  K_0(v), R_{12}^{(0)} \rbrack=0$,  $\lbrack  K_2(uq), R_{01}(v/u) \rbrack=0$ and
\beqa
R_{01}(v/u) R_{02}^{(0)}\ R_{12}^{(0)}=\ R_{12}^{(0)}  R_{02}^{(0)} R_{01}(v/u) \nonumber\ 
\eeqa
to show:
\beqa
K_0(v) \tr_{12}\big(P^{-}_{12}K_1(u)\ R_{12}^{(0)} K_2(uq)\big)\ &=&   q \tr_{12}\big(P^{-}_{12} R^{-1}_{01}(v/u) K_1(u)R^{(0)}_{01}K_0(v) R_{12}^{(0)}  R_{02}^{(0)} R_{01}(v/u) K_2(uq)\big) \nonumber \\
&=&  q  \tr_{12}\big(P^{-}_{12} R^{-1}_{01}(v/u) K_1(u)R^{(0)}_{01} R_{12}^{(0)} K_0(v) R_{02}^{(0)} K_2(uq)R_{01}(v/u) \big)  \nonumber
\eeqa
Applying again (\ref{RE}) to the combination  $K_0(v) R_{02}^{(0)} K_2(uq)$ and using $R_{02}(v/uq) R_{01}^{(0)}\ R_{12}^{(0)}=\ R_{12}^{(0)}  R_{01}^{(0)} R_{02}(v/uq)$, it follows:
\beqa
(a)&=& q \tr_{12}\big(P^{-}_{12} R^{-1}_{01}(v/u) K_1(u)R^{(0)}_{01} R_{12}^{(0)}   R^{-1}_{02}(v/uq)   K_2(uq) R_{02}^{(0)} K_0(v)R_{02}(v/uq) R_{01}(v/u) \big)  \nonumber\\
&=& q \tr_{12}\big(P^{-}_{12} R^{-1}_{01}(v/u) K_1(u)    R^{-1}_{02}(v/uq)  R_{12}^{(0)}R^{(0)}_{01} K_2(uq) R_{02}^{(0)} K_0(v)R_{02}(v/uq) R_{01}(v/u) \big)  \nonumber\\
&=& q \tr_{12}\big(P^{-}_{12} R^{-1}_{01}(v/u) R^{-1}_{02}(v/uq)  K_1(u)    R_{12}^{(0)}K_2(uq) R^{(0)}_{01} R_{02}^{(0)} K_0(v)R_{02}(v/uq) R_{01}(v/u) \big) \nonumber.
\eeqa
Then, using  $P^{-}_{12}R_{02}(x/q)R_{01}(x)=  P^{-}_{12} (x^2-q^2)(x^2-q^{-2})/x^2$,  $ q P^{-}_{12}R_{01}^{(0)}R_{02}^{(0)}= P^{-}_{12}$, eq. (\ref{RE}) and the cyclicity of the trace, the last expression simplifies to:
\beqa
(a)&=&  q \tr_{12}\big(P^{-}_{12} K_1(u)   R_{12}^{(0)} K_2(uq)R^{(0)}_{01} R_{02}^{(0)}K_0(v) P_{12}^-\big)  \nonumber\\
&=&   \tr_{12}\big(P^{-}_{12} K_1(u)   R_{12}^{(0)} K_2(uq)\big) K_0(v)  \nonumber\\
&=& \Gamma(u)K_0(v)\nonumber \ .
\eeqa
\end{proof}

Now, define:
\beqa
\Gamma(u)= \frac{1}{2(q-q^{-1})}\left(\Delta(u) - \frac{2\bar\rho}{(q-q^{-1})}\right)\ .\nonumber
\eeqa
Using the entries of (\ref{K}), by Proposition \ref{propqdet} it implies
$ [\Delta(u),\cW_\pm(v)]=[\Delta(u),\cG_\pm(v)]=0$.
Using  (\ref{c1}), (\ref{c2}), it follows:
\begin{cor}\label{corcent}
\beqa
\qquad \Delta(u)&=& (q-q^{-1})u^2q^2\Big(\cW_+(u)\cW_-(uq) + \cW_-(u)\cW_+(uq)\Big) - \frac{(q-q^{-1})}{\bar\rho} \Big(\cG_+(u)\cG_-(uq) + \cG_-(u)\cG_+(uq)\Big)  \label{deltau}\\
&& \quad - \ \cG_+(u) - \cG_+(uq) - \cG_-(u) - \cG_-(uq)  \ \nonumber
\eeqa
provides a generating function for  central elements  in $\bar{\cal A}_q$.  
\end{cor}
Expanding  $\Delta(u)$  in power series of $U=qu^2/(q+q^{-1})$, the coefficients produce the central elements of  $\bar{\cal A}_q$ given by (\ref{deltadef}). Namely,  by straightforward calculations one gets:
\beqa
\Delta(u) = - \sum_{n=0}^{\infty} U^{-n-1}q^{-n-1}(q^{n+1}+q^{-n-1}) \Delta_{n+1}\ . \nonumber
\eeqa
\begin{rem} In \cite[Lemma 13.8]{Ter19b}, a generating function for central elements is given. By \cite[Corollary 8.4]{Ter19b} and  \cite[Definition 13.1]{Ter19b}, alternatively three other generating functions may be considered. For instance:
\beqa
Z^\vee(t) &=& \cG(qt) \tilde\cG(q^{-1}t)  - qt \cW^{+}(qt) \cW^{-}(q^{-1}t)\ ,\nonumber\\
\sigma(Z^\vee(t)) &=& \tilde\cG(qt) \cG(q^{-1}t)  - qt \cW^{-}(qt) \cW^{+}(q^{-1}t)\ .\nonumber
\eeqa
Using the identification (\ref{idc0})-(\ref{idc3}), the image of  the generating function $\Delta(u)$  in  the algebra ${\cal U}_q^+$ follows:
\beqa
\Delta(u) \mapsto -q^{-1}(q^2-q^{-2})  \left( Z^\vee(q^{-1}t)  + \sigma(Z^\vee(q^{-1}t) \right)\ . \nonumber
\eeqa
\end{rem}

\vspace{1mm}

\subsection{Specialization $q\rightarrow 1$}
Due to the presence of poles at $q=1$ in the off-diagonal entries of $K(u)$ in (\ref{K}),  the relations (\ref{RE}) are not suitable for  the specialization  $q\rightarrow 1$. However,  it is possible to solve this problem within  the framework of the non-standard classical Yang-Baxter algebra \cite{Cher83,STS,BV90,Sk06} in order to obtain 
an alternative presentation of $\bar{\cal A}$, besides Definition \ref{def:Acl}, viewed as a specialization $q\rightarrow 1$ of the Freidel-Maillet type algebra (\ref{RE}). Introduce the r-matrix\footnote{Note that this r-matrix can be obtained from a limiting case of a r-matrix considered in \cite{BBC}.} 
\beqa
\bar r(u,v)= \frac{1}{(u^2/v^2-1)}\begin{pmatrix}
       1&0&0&0\\
       0& -1 & 2u/v &0\\
       0& 2u/v & -1 &0\\
       0 &0&0& 1
      \end{pmatrix}\ 
\label{r12basic}
\eeqa
solution of the non-standard classical Yang-Baxter equation \cite{BV90}:
 \begin{equation}\label{nsCYBE}
  [\ {\bar r}_{13}(u_1,u_3)\ , \ {\bar r}_{23}(u_2,u_3)\ ]=[\ {\bar r}_{21}(u_2,u_1)\ , \ {\bar r}_{13}(u_1,u_3)\ ]+[\ {\bar r}_{23}(u_2,u_3)\ , \ {\bar r}_{12}(u_1,u_2)\ ]\;,
 \end{equation}
where   $\bar r_{21}(u,v) = P\bar r_{12}(u,v)P$ $(= \bar r_{12}(u,v)$  for  (\ref{r12basic})).
 Define the generating functions:
\beqa
{w}_+(u)&=&\sum_{k=0}^\infty{\tw}_{-k}U^{-k-1} \ , \quad {w}_-(u)=\sum_{k=0}^\infty{\tw}_{k+1}U^{-k-1} \ ,\label{eq:cuw}\\
\quad  {g}_+(u)&=&\sum_{k=0}^\infty{{\tg}}_{k+1}U^{-k-1}\  \ ,\quad  {g}_-(u)=\sum_{k=0}^\infty\tilde{{\tg}}_{k+1}U^{-k-1}\ \quad\mbox{with}\quad U=u^2/2\ .\label{eq:cuw1}
\eeqa
\begin{prop}
 \label{prop32}
The algebra $\bar{\cal A}$ admits a FRT presentation given by 
\begin{equation}
 B(u)= \frac{1}{2}\begin{pmatrix}
  \frac{1}{4}\ {\tg}_-(u)    & u\tw_-(u)  \\
 u \tw_+(u)  &      \frac{1}{4}\ {\tg}_+(u)    
      \end{pmatrix}\label{eq:BW}
\end{equation}
that satisfies the non-standard classical Yang-Baxter algebra
\begin{equation}\label{eq:nsYBa}
 [\ B_{1}(u)\ , \ B_{2}(v)\ ]=[\ {\bar r}_{21}(v,u)\ , \ B_{1}(u)\ ]+[\ B_{2}(v)\ , \ {\bar r}_{12}(u,v)\ ]\; .
\end{equation}
\end{prop}
\begin{proof}  
Insert (\ref{eq:BW}) into (\ref{eq:nsYBa}) with (\ref{r12basic}). Define the formal variables $U=u^2/2$ and $V=v^2/2$. One obtains equivalently:
\begin{eqnarray}
&& (U-V)\big[w_\pm(u),{w}_\mp(v)\big]= \frac{1}{2}( {g}_\pm(u)-{g}_\mp(u) + {g}_\mp(v)-{g}_\pm(v))\ ,\nonumber\\
&&(U-V)\big[{g}_\epsilon(u),{w}_\pm(v)\big]\mp \epsilon16\big(U{w}_\pm(u)-V{w}_\pm(v)\big)=0\ ,\quad \epsilon=\pm 1\ ,\nonumber\\
&& \big[{g}_\pm(u),{g}_\mp(v)\big]=0\ ,\nonumber\\
&&\big[{w}_\pm(u),{w}_\pm(v)\big]=0\ ,\quad \big[{g}_\pm(u),{g}_\pm(v)\big]=0\ .\nonumber
\end{eqnarray}
These relations are equivalent to the specialization $q\rightarrow 1$ of (\ref{ec1})-(\ref{ec16}) ($ \bar\rho \mapsto 16$). Using  (\ref{eq:cuw}), the above equations are equivalent to (\ref{po1})-(\ref{po4}).
\end{proof}
\begin{rem} For the specialization $q\rightarrow 1$, the generating function (\ref{deltau}) reduces to  $\delta(u)=-2({g}_+(u)+{g}_-(u))$. 
\end{rem}

\section{Quotients of  $\bar{\cal A}_q$ and tensor product representations}\label{secTPR}
In this section, a class of solutions - so-called `dressed' solutions - of the Freidel-Maillet type equation (\ref{RE}) are constructed and studied in details by adapting known technics of the so-called reflection equation \cite{Skly88}, see Proposition \ref{p41}. By Lemma \ref{lemKN}, it is shown that
 the entries of the dressed solutions can be written in terms of  the `truncated' generating functions (\ref{ecf1})-(\ref{ecf2}), whose   generators act on $N-$fold tensor product representations of  $U_q({sl_2})$ according to (\ref{r1})-(\ref{r4}). Realizations of $\bar{\cal A}_q$ in $U_q(sl_2)^{\otimes N }$ are obtained, see Proposition \ref{propKN}.
\subsection{Dressed solutions of the Freidel-Maillet type equation}
The starting point of the following analysis is an adaptation of \cite[Proposition 2]{Skly88}, \cite{FM91}, to the Freidel-Maillet type equation (\ref{RE}), thus we  skip the proof of the proposition below.   Let $K_0(u)$  be a solution of  (\ref{RE}). Assume there exists a pair of quantum Lax operators satisfying   the exchange relations:
\beqa R(u/v)\ (L(u)\otimes I\!\!I)\ ( I\!\!I \otimes L(v))  &=&  ( I\!\!I \otimes L(v))\  (L(u)\otimes I\!\!I)   \ R(u/v)  \ ,  \label{YBA1}\\
 R(u/v)\ (L_0(u)\otimes I\!\!I)\ ( I\!\!I \otimes L_0(v))  &=&  ( I\!\!I \otimes L_0(v))\  (L_0(u)\otimes I\!\!I)   \ R(u/v)  \label{YBA2} \ ,\\
 R^{(0)}\ (L_0(u)\otimes I\!\!I)\ ( I\!\!I \otimes L(v))  &=&  ( I\!\!I \otimes L(v))\  (L_0(u)\otimes I\!\!I)   \ R^{(0)}  \label{YBA3} \  .
 \eeqa
Using (\ref{YBA1})-(\ref{YBA3}), it is easy to show that $L_0(uv_1) K_0(u) L(u/v_1)$ for any $v_1\in \mathbb{K}^*$ is also a solution of  (\ref{RE}) (similar to \cite[Proposition 2]{Skly88}). More generally it follows\footnote{Here the index $[j]$ characterizes the `quantum space' $V_{[j]}$ on which the entries of $L(u),L_0(u)$ act. With respect to the ordering $V_{[2]}\otimes V_{[1]}$ used below for (\ref{r1})-(\ref{r4}), one has:
\beqa
((T)_{[\textsf 2]}(T')_{[\textsf 1]}(T'')_{[\textsf 2]})_{ij} =\sum_{k,\ell=1}^2  (T)_{ik}(T'')_{\ell j}\otimes (T')_{k\ell}  \ .
\eeqa
} 
%
%
\begin{prop}\label{p41} Let $K_0(u)$  be a solution of  (\ref{RE}). Let  $N$ be a positive integer and $\{v_i\}_{i=1}^{N}\in \mathbb{K}^*$. Let  $L(u),L_0(u)$ be such that (\ref{YBA1})-(\ref{YBA3}) hold. Then
\beqa
K^{(N)}(u) = (L_0(uv_N))_{[\textsf N]} 
\cdots  
 (L_0(uv_1))_{[\textsf 1]} K_0(u) (L(u/v_1)))_{[\textsf 1]} 
\cdots  
 (L(u/v_N))_{[\textsf N]}\label{dK}
\eeqa
satisfies (\ref{RE}).
\end{prop}

This proposition provides a tool for the explicit  construction of so-called `dressed' solutions of  (\ref{RE}). Below, we construct explicit examples of such solutions. To this end, we first introduce some known basic material. Recall the algebra  $U_q(sl_2)$   consists of three generators denoted $S_\pm,s_3$. They  satisfy
\beqa
[s_3,S_\pm]=\pm S_\pm \qquad \mbox{and} \qquad
[S_+,S_-]=\frac{q^{2s_3}-q^{-2s_3}}{q-q^{-1}}\ .\label{Uqsl2}
\eeqa
The central element of $U_q(sl_2)$  is the so-called Casimir operator:
\beqa
\Omega = \frac{q^{-1}q^{2s_3}+ q q^{-2s_3}}{(q-q^{-1})^2} + S_+S_- = \frac{qq^{2s_3}+ q^{-1} q^{-2s_3}}{(q-q^{-1})^2} + S_-S_+ .\label{Casimir}
\eeqa
Let $V$ be the spin-$j$ irreducible finite dimensional representation of $U_q(sl_2)$ of dimension $2j+1$. The eigenvalue  $\omega_j$ of $\Omega$ is  such that
\beqa
\omega_j \equiv \frac{ w_0^{(j)}}{(q-q^{-1})^2}\qquad \mbox{with} \qquad w_0^{(j)}=q^{2j+1}+q^{-2j-1}.\label{valCas}
\eeqa
Define the so-called quantum  Lax operators 
\beqa
&& \quad  L_0(u)=
       \begin{pmatrix}
     uq^{1/2} q^{s_3} &  0 \\
   0  &  uq^{1/2} q^{-s_3}
      \end{pmatrix}  \qquad \mbox{and} \qquad   L(u)=
       \begin{pmatrix}
      uq^{1/2} q^{s_3} -  u^{-1}q^{-1/2} q^{-s_3} & (q-q^{-1})S_- \\
  (q-q^{-1})S_+  &   uq^{1/2} q^{-s_3} -  u^{-1}q^{-1/2} q^{s_3}
      \end{pmatrix} \ .\label{Lop}
\eeqa
Recall the R-matrices (\ref{R}) and (\ref{R0}). One routinely checks that the relation (\ref{YBA1})  holds. The relations (\ref{YBA2})-(\ref{YBA3}) follow as a limiting case of  (\ref{YBA1}).  Note that the overall factor $uq^{1/2}$ in the expression of $L_0(u)$   is kept for further convenience only.
 Let $k_\pm,\bar\epsilon_\pm \in \mathbb{K}$. Define:
\beqa
 K_0(u) =  \begin{pmatrix} u^{-1} \bar\epsilon_+& \frac{k_+}{(q-q^{-1})}  \\
    \frac{k_-}{(q-q^{-1})}  &  u^{-1} \bar\epsilon_-
      \end{pmatrix}\ . \label{K0}
\eeqa
It is checked that $K_0(u)$  satisfies (\ref{RE}). 
As a basic example of dressed solution,  consider the case $N=1$ of (\ref{dK}).  Define the four operators in $U_q(sl_2)$:
\beqa
{\cW}_{0}^{(1)}&=& \ k_+v_1q^{1/2}S_+q^{s_3} + \bar\epsilon_+q^{2s_3}\ , \label{N1case1}\\
{\cW}_{1}^{(1)}&=&    \ k_-v_1q^{1/2}S_-q^{-s_3} + \bar\epsilon_- q^{-2s_3}
\ ,\\
{\cG}_{1}^{(1)}&=&  
 k_+k_-v_1^{2}\frac{( w_0^{(j_1)}-(q+q^{-1})q^{2s_3})}{(q-q^{-1})}  + (q^2-q^{-2})k_- \bar\epsilon_+ v_1 q^{-1/2} S_-q^{s_3} + (q-q^{-1}) \bar\epsilon_+ \bar\epsilon_-\ ,
\\
\tilde{\cG}_{1}^{(1)}&=&   k_+k_-v_1^{2}\frac{( w_0^{(j_1)}-(q+q^{-1})q^{-2s_3})}{(q-q^{-1})} + (q^2-q^{-2})k_+ \bar\epsilon_- v_1 q^{-1/2} S_+q^{-s_3}  + (q-q^{-1}) \bar\epsilon_+ \bar\epsilon_-
\ . \label{N1case4}
\eeqa
Computing explicitly the entries of  (\ref{dK}) for $N=1$, one finds that the dressed solution can be written as:
\beqa
&&\quad  K^{(1)}(u) =  \begin{pmatrix}  uq \cW_0^{(1)} - u^{-1}v_1^2 \bar\epsilon_+ &  \frac{\cG_1^{(1)} }{k_-(q+q^{-1})} + \frac{k_+ qu^2}{(q-q^{-1})}- \frac{k_+v_1^2 w_0^{(j_1)}}{(q^2-q^{-2})} - \frac{ \bar\epsilon_+ \bar\epsilon_-(q-q^{-1}) }{k_-(q+q^{-1})} \\ \frac{\tilde\cG_1^{(1)} }{k_+(q+q^{-1})} + \frac{k_- qu^2}{(q-q^{-1})}  - \frac{k_-v_1^2 w_0^{(j_1)}}{(q^2-q^{-2})} -\frac{ \bar\epsilon_+ \bar\epsilon_-(q-q^{-1}) }{k_+(q+q^{-1})} 
     &   uq \cW_1^{(1)} - u^{-1}v_1^2 \bar\epsilon_-
      \end{pmatrix}\  .\label{K1}
\eeqa 

\subsection{General dressed solutions}
The structure of the above solution (\ref{K1}) can be generalized to dressed solutions of arbitrary size as we now show. 
According to the ordering of the `quantum' vector spaces $V^{(N)}= V_{[\textsf N]} \otimes \cdot \cdot\cdot \otimes V_{[\textsf 2]} \otimes V_{[\textsf 1]}$, let us first define recursively the four families of operators $\{ {\cW}_{-k}^{(N)},{\cW}_{k+1}^{(N)},{\cG}_{k+1}^{(N)},\tilde{\cG}_{k+1}^{(N)}|k=0,1,...,N\}$, where  $N$ is a positive integer:
\beqa
{\cW}_{-k}^{(N)}&=& \ \frac{(q-q^{-1})}{k_-(q+q^{-1})^2}
\left(v_Nq^{1/2}S_+q^{s_3}\otimes
{\cG}_{k}^{(N-1)}\right)\  +\ q^{2s_3}\otimes {\cW}_{-k}^{(N-1)}  -\frac{v_{N}^2}{(q+q^{-1})}I\!\!I\otimes {\cW}_{-k+1}^{(N-1)}   \label{r1}\\
&& \qquad+\ \frac{v_N^2 w_0^{(j_N)}}{(q+q^{-1})^2}{\cW}_{-k+1}^{(N)}
\ ,\nonumber\\
{\cW}_{k+1}^{(N)}&=&   \frac{(q-q^{-1})}{k_+(q+q^{-1})^2}
\left(k_-v_Nq^{1/2}S_-q^{-s_3}\otimes {\tilde {\cG}}_{k}^{(N-1)}\right) +\ q^{-2s_3}\otimes {\cW}_{k+1}^{(N-1)} 
-\frac{v_{N}^2}{(q+q^{-1})}I\!\!I\otimes {\cW}_{k}^{(N-1)} \label{r2}\\
&&\qquad  +  \frac{v_N^2 w_0^{(j_N)}}{(q+q^{-1})^2}{\cW}_{k}^{(N)}
\ ,\nonumber\\
{\cG}_{k+1}^{(N)}&=&  (q^2-q^{-2})
k_-v_Nq^{-1/2}S_-q^{s_3}\otimes {\cW}_{-k}^{(N-1)}
 -\frac{v_N^{2}}{(q+q^{-1})}q^{2s_3}\otimes {\cG}_{k}^{(N-1)} 
+I\!\!I \otimes {\cG}_{k+1}^{(N-1)} \label{r3}\\
&&\qquad +\frac{v_N^2 w_0^{(j_N)}}{(q+q^{-1})^2}{\cG}_{k}^{(N)}\ ,
\nonumber\\
{\tilde{\cG}}_{k+1}^{(N)}&=&  (q^2-q^{-2})
k_+v_Nq^{-1/2}S_+q^{-s_3}\otimes {\cW}_{k+1}^{(N-1)} -\frac{v_N^{2}}{(q+q^{-1})}q^{-2s_3}\otimes {\tilde{\cG}}_{k}^{(N-1)} 
+I\!\!I \otimes {\tilde{\cG}}_{k+1}^{(N-1)} \label{r4}\\
&&\qquad +\frac{v_N^2 w_0^{(j_N)}}{(q+q^{-1})^2}{\tilde{\cG}}_{k}^{(N)}\ .\nonumber
\eeqa
Here for the special case $k=0$ we identify\,\footnote{Although the notation is ambiguous, one must keep in mind that ${{\cW}}_{k}^{(N)}|_{k=0}\neq {{\cW}}_{-k}^{(N)}|_{k=0}$\ ,${{\cW}}_{-k+1}^{(N)}|_{k=0}\neq {{\cW}}_{k+1}^{(N)}|_{k=0}$\ for any $N$.}
\beqa
{{\cW}}_{k}^{(N)}|_{k=0}\equiv 0\ ,\quad {{\cW}}_{-k+1}^{(N)}|_{k=0}\equiv 0\ ,\quad {\cG}_{k}^{(N)}|_{k=0}={\tilde{\cG}}_{k}^{(N)}|_{k=0}\equiv \frac{k_+k_-(q+q^{-1})^2}{(q-q^{-1})}I\!\!I^{(N)}\ \label{not0}
\eeqa
together with the `initial' conditions for $k\geq 1$ (the notation (\ref{alp}) is used) 
\beqa
&&{{\cW}}_{-k}^{(0)}\equiv \left(   \frac{\alpha_1}{q+q^{-1}} \right)^{k-1}\left(   \frac{\alpha_1}{q+q^{-1}} \right)_{|_{v_1=0}} {{\cW}}_{0}^{(0)} \ ,\quad
{{\cW}}_{k+1}^{(0)}\equiv \left(   \frac{\alpha_1}{q+q^{-1}} \right)^{k-1}\left(   \frac{\alpha_1}{q+q^{-1}} \right)_{|_{v_1=0}} {{\cW}}_{1}^{(0)}\ ,\label{init21}\\
&&{{\cG}}_{k+1}^{(0)}={\tilde{\cG}}_{k+1}^{(0)}\equiv \left(   \frac{\alpha_1}{q+q^{-1}} \right)^{k}{{\cG}}_{1}^{(0)} \ ,\label{init22}
\eeqa
where
\beqa
{{\cW}}_{0}^{(0)}\equiv \bar\epsilon_+ \ ,\quad {{\cW}}_{1}^{(0)}\equiv \bar\epsilon_- \qquad  \mbox{and}\qquad
{\cG}_{1}^{(0)}={\tilde{\cG}}_{1}^{(0)}\equiv  \bar\epsilon_+ \bar\epsilon_-(q-q^{-1})\ .\label{initrep0}
\eeqa

A crucial ingredient in the construction of dressed solutions by induction  from (\ref{dK})  is the existence of a set of linear relations satisfied by the operators  (\ref{r1})-(\ref{r4}). We proceed by strict analogy with \cite[Appendix B]{BK}, thus we skip most of the details of the proof.  For further convenience, introduce the notation:
\beqa
\bar\epsilon_\pm^{(N)} = (-1)^{N} \left(\prod_{k=1}^N v_k^2 \right) \bar\epsilon_\pm\ .\label{eN}
\eeqa
\begin{lem}\label{lemrel} The operators (\ref{r1})-(\ref{r4}) satisfy the linear relations: 
\beqa
&& \sum_{k=0}^{N}  c_{k}^{(N)} \cW_{-k}^{(N)} + \bar\epsilon_{+}^{(N)} =0 \ ,\qquad 
 \sum_{k=0}^{N}  c_{k}^{(N)} \cW_{k+1}^{(N)} + \bar\epsilon_{-}^{(N)} =0 \ ,\label{lr1}\\
&& \sum_{k=0}^{N}     c_{k}^{(N)}  \cG_{k+1}^{(N)} =0 \ ,\qquad \qquad
 \sum_{k=0}^{N}   c_{k}^{(N)}  \tilde{\cG}_{k+1}^{(N)} =0 \  \label{lr2}
\eeqa
with\footnote{For the elementary symmetric polynomials in the variables $\{x_i|i=1,...,n\}$, we use the notation:
\beqa
 \textsf{e}_{k}(x_1,x_2,...,x_n) = \sum_{1\leq j_1 < j_2 < \cdots < j_k \leq n } \! \! \! \! \! \! \! x_{j_1}x_{j_2}\cdots x_{j_k} \ .\nonumber
\eeqa}  \ $c_{k}^{(N)}= (-1)^{N-k-1}(q+q^{-1})^{k}  \textsf{e}_{N-k}(\alpha_1,\alpha_2,\cdots, \alpha_N)$,
\beqa
\alpha_1 = \frac{v_1^2w_0^{(j_1)}}{(q+q^{-1})} +  \frac{ \bar\epsilon_+ \bar\epsilon_-(q-q^{-1})^2 }{k_+k_-(q+q^{-1})}\ ,\quad
 \alpha_k = \frac{v_k^2w_0^{(j_k)}}{(q+q^{-1})} \quad \mbox{for}\quad k=2,...,N\ .\label{alp}
\eeqa
\end{lem}
\begin{proof} For $N=1,2$, the four relations (\ref{lr1})-(\ref{lr2}) are explicitly checked. Then we proceed by induction. 
\end{proof}

The result below is obtained after some straightforward calculations similar to those performed in \cite{BK,BK1}, thus we just sketch the proof. Introduce the `truncated'  generating functions:
\begin{eqnarray}
&& \cW_+^{(N)}(u)  =\sum_{k= 0}^{N-1} f_{k+1}^{(N)}(u) \cW_{-k}^{(N)}\ ,\quad  \cW_-^{(N)}(u)  = \sum_{k= 0}^{N-1} f_{k+1}^{(N)}(u) \cW_{k+1}^{(N)} \label{ecf1} \\
&&  \cG_+^{(N)}(u)  =\sum_{k= 0}^{N-1} f_{k+1}^{(N)}(u) \cG_{k+1}^{(N)}\ ,\quad  \cG_-^{(N)}(u)  = \sum_{k= 0}^{N-1} f_{k+1}^{(N)}(u) \tilde{\cG}_{k+1}^{(N)}\;\label{ecf2}
\end{eqnarray}
where
\begin{eqnarray}
f_{k}^{(N)}(u)=\sum_{p=k}^{N} (-1)^{N-p}(q+q^{-1})^{p-1}\textsf{e}_{N-p}(\alpha_1,\alpha_2,...,\alpha_N) U^{p-k} \quad \mbox{with} \quad U=qu^2/(q+q^{-1})\  . \label{fqn}
\end{eqnarray}
\begin{lem}\label{lemKN}  Dressed solutions of the form (\ref{dK}) can be written as:
\beqa
  K^{(N)}(u)&=&   \begin{pmatrix}
      uq \cW_+^{(N)}(u) + u^{-1} \bar\epsilon_+^{(N)} &  \frac{1}{k_-(q+q^{-1})}\cG_+^{(N)}(u) + \frac{k_+(q+q^{-1})}{(q-q^{-1})} f_0^{(N)}(u) \\
   \frac{1}{k_+(q+q^{-1})}\cG_-^{(N)}(u) + \frac{k_-(q+q^{-1})}{(q-q^{-1})}f_0^{(N)}(u)    & uq \cW_-^{(N)}(u) + u^{-1} \bar\epsilon_-^{(N)}
      \end{pmatrix} \label{eq:REfinitebis}
\eeqa
with (\ref{ecf1})-(\ref{ecf2})  and (\ref{eN}). 
\end{lem}
\begin{proof} For $N=1$, one checks that (\ref{eq:REfinitebis})  coincides with (\ref{K1}). Then, we proceed by induction. Assume $K^{(N)}(u)$  is of the form (\ref{eq:REfinitebis}) for $N$ fixed. We compute  $((L_0(uv_{N+1}))_{[\textsf N+1]}  K^{(N)}(u) (L(u/v_{N+1}))_{[\textsf N+1]})_{ij}$ for $i,j=1,2$. For instance, consider the entry  $(11)_{N+1}$. Explicitly, it reads:
\beqa
(11)_{N+1} &=& uq\left( (q-q^{-1}) v_{N+1}q^{1/2}S_+q^{s_3}\otimes \left(   \frac{1}{k_-(q+q^{-1})}\cG_+^{(N)}(u) + \frac{k_+(q+q^{-1})}{(q-q^{-1})} f_0^{(N)}(u)    \right)  + q^{2s_3}\otimes \bar\epsilon_+^{(N)}\right. \nonumber\\
&& \left.\qquad   + (u^2q q^{2s_3} -v_{N+1}^2 ) \otimes  \cW_+^{(N)}(u)     \right) - u^{-1}v_{N+1}^2  \bar\epsilon_+^{(N)}\ .  \nonumber
\eeqa
Inserting  (\ref{ecf1}), (\ref{ecf2}) and using the definitions (\ref{r1})-(\ref{r4}), (\ref{eN}) for  $ N \rightarrow N+1$, after some simple operations and reorganizing all terms one gets :
\beqa
(11)_{N+1} &=& uq \left( \sum_{k=0}^{N-1}  \left(  (q+q^{-1})f_{k}^{(N)}(u) - \alpha_{N+1}f_{k+1}^{(N)}(u) \right)  \cW_{-k}^{(N+1)} + (q+q^{-1})f_{N}^{(N)}(u)   \cW_{-N}^{(N+1)}   \right) +  u^{-1} \bar\epsilon_+^{(N+1)}  \nonumber\\
&& \quad +  q^{2s_3} \otimes \left(\underbrace{  \sum_{k=0}^{N-1}  \left(  qu^2f_{k+1}^{(N)}(u) -(q+q^{-1})f_{k}^{(N)}(u) \right)  \cW_{-k}^{(N+1)}   - (q+q^{-1})f_{N}^{(N)}(u) \cW_{-N}^{(N+1)} + \bar\epsilon_+^{(N)} }_{\equiv\Gamma(u)}  \right)\ .\nonumber
  \eeqa
Identifying $(11)_{N+1}$ with $(K^{(N+1)}(u))_{11}$ leads to a set of constraints. They read:
\beqa
(q+q^{-1})f_{k}^{(N)}(u) - \alpha_{N+1}f_{k+1}^{(N)}(u)  &=& f_{k+1}^{(N+1)}(u) \quad \mbox{for}\quad k=0,...,N-1\ ,
\label{co1}\\
(q+q^{-1})f_{N}^{(N)}(u) &=& f_{N+1}^{(N+1)}(u) \ \label{co2}
\eeqa
and $\Gamma(u)=0$. The solution of the constraints (\ref{co1})-(\ref{co2}) is given by (\ref{fqn}). Using this expression, one finds that $\Gamma(u)$ coincides with the l.h.s of the first equation in (\ref{lr1}). By Lemma \ref{lemrel}, it follows $\Gamma(u)=0$, so $(11)_{N+1}=(K^{(N+1)}(u))_{11}$. By similar arguments, one shows $(ij)_{N+1}=(K^{(N+1)}(u))_{ij}$ using (\ref{lr1}), (\ref{lr2}).
\end{proof}

\subsection{Realizations of $\bar{\cal A}_q$ in $U_q(sl_2)^{\otimes N }$}
According to previous results, dressed solutions of the form (\ref{eq:REfinitebis}) automatically generate the finite set of operators (\ref{r1})-(\ref{r4}). In this section, we show (\ref{r1})-(\ref{r4})  extends to $k\in {\mathbb N}$ and provide realizations of $\bar{\cal A}_q$ in $U_q(sl_2)^{\otimes N }$.  To this aim, we need a  generalization  of Lemma \ref{lemrel}.
\begin{lem}\label{lemrel2} For any  $p\in {\mathbb N}$, the operators (\ref{r1})-(\ref{r4}) satisfy the linear relations: 
\beqa
&& \sum_{k=0}^{N}  c_{k}^{(N)} \cW_{-k-p}^{(N)} + \delta_{p,0} \bar\epsilon_{+}^{(N)} =0 \ ,\qquad 
 \sum_{k=0}^{N}  c_{k}^{(N)} \cW_{k+1+p}^{(N)} + \delta_{p,0}\bar\epsilon_{-}^{(N)} =0 \ ,\label{lrp1}\\
&& \sum_{k=0}^{N}     c_{k}^{(N)}  \cG_{k+1+p}^{(N)} =0 \ ,\qquad \qquad
 \sum_{k=0}^{N}   c_{k}^{(N)}  \tilde{\cG}_{k+1+p}^{(N)} =0\ . \  \label{lrp2}
\eeqa
\end{lem}
\begin{proof} For $p=0$ the four relations hold by Lemma \ref{lemrel}. For  $N=1$ and any $p \geq 1$, the four relations are checked using (\ref{init21}), (\ref{init22}). Then we proceed by induction on $N$. 
\end{proof}

Define $\bar{\cal A}_q^{(N)}$ as the algebra generated by  $\{ {\cW}_{-k}^{(N)},{\cW}_{k+1}^{(N)},{\cG}_{k+1}^{(N)},\tilde{\cG}_{k+1}^{(N)}|k\in {\mathbb N}\}$. We are now in position to give the main result of this section.
\begin{prop}\label{propKN}  The map $\bar{\cal A}_q \rightarrow  \bar{\cal A}_q^{(N)} $ given by:
\beqa
{\tW}_{-k}\mapsto{\cW}_{-k}^{(N)}\ ,\quad {\tW}_{k+1}\mapsto{\cW}_{k+1}^{(N)} \ ,\quad {\tG}_{k+1}\mapsto{\cG}_{k+1}^{(N)}\ ,\quad {\tilde{\tG}}_{k+1}\mapsto{\tilde{\cG}}_{k+1}^{(N)}\  \ \nonumber
\eeqa
with (\ref{r1})-(\ref{r4}) for $k\in {\mathbb N}$ and (\ref{rho})  is a surjective homomorphism. 
\end{prop}
\begin{proof}  Consider the image of (\ref{K}) such that the generators  in (\ref{c1}), (\ref{c2}) map to (\ref{r1})-(\ref{r4}). For instance, one  has:
\beqa
{\cW}_+(u) \mapsto \sum_{k\in {\mathbb N}}{\cW}_{-k}^{(N)}U^{-k-1} &=& \sum_{k=0}^{N-1}{\cW}_{-k}^{(N)}U^{-k-1} +  \sum_{k=N}^{\infty}{\cW}_{-k}^{(N)}U^{-k-1} \ .\label{eqWp}
\eeqa
Using  (\ref{lrp1}):
\beqa
\sum_{k=N}^{\infty}{\cW}_{-k}^{(N)}U^{-k-1} &=& \sum_{p=0}^{\infty}{\cW}_{-N-p}^{(N)}U^{-N-p-1}
=- \frac{1}{c_N^{(N)}}\sum_{p=0}^{\infty} \sum_{k=0}^{N-1}c_k^{(N)}{\cW}_{-k-p}^{(N)}U^{-N-p-1}   - \frac{\bar\epsilon_+^{(N)}}{c_N^{(N)}} U^{-N-1}\nonumber \\
 &=&  - \frac{1}{c_N^{(N)}} \left(\sum_{k=0}^{N-1}c_k^{(N)}U^{k-N}\right){\cW}_+(u)  +  \frac{1}{c_N^{(N)}}   \sum_{k=1}^{N-1} \sum_{p=0}^{k-1}  c_k^{(N)}{\cW}_{-p}^{(N)}U^{-N-p-1+k}      - \frac{\bar\epsilon_+^{(N)}}{c_N^{(N)}} U^{-N-1} \nonumber \\
 &=&  - \frac{1}{c_N^{(N)}} \left(\sum_{k=0}^{N-1}c_k^{(N)}U^{k-N}\right){\cW}_+(u)
  +  \frac{1}{c_N^{(N)}}   \sum_{k=0}^{N-1}  U^{-k-1} {\cW}_{-k}^{(N)} \left( \sum_{p=k+1}^{N-1}  c_p^{(N)}  U^{p-N}\right)       - \frac{\bar\epsilon_+^{(N)}}{c_N^{(N)}} U^{-N-1} \nonumber \ .
\eeqa
Replacing  the last expression into (\ref{eqWp}) and using
 (\ref{ecf1}), (\ref{ecf2}) and (\ref{fqn}), one gets:
\beqa
f_0^{(N)}(u){\cW}_+(u) \mapsto    {\cW}_+^{(N)}(u)  + u^{-2}q^{-1}\bar\epsilon_+^{(N)} \ .\nonumber
\eeqa
Similarly, using  (\ref{lrp1}),  (\ref{lrp2}) one finds:
\beqa
f_0^{(N)}(u){\cW}_-(u) \mapsto    {\cW}_-^{(N)}(u)  + u^{-2}q^{-1}\bar\epsilon_-^{(N)} \ ,\quad
f_0^{(N)}(u){\cG}_\pm(u) \mapsto    {\cG}_\pm^{(N)}(u) \ .\nonumber
\eeqa
It follows  $f_0^{(N)}(u) K(u) \mapsto K^{(N)}(u)$. Thus, the operators  (\ref{r1})-(\ref{r4}) for $k\in {\mathbb N}$ generate a quotient of the algebra $\bar{\cal A}_q$ by the  relations   (\ref{lrp1}), (\ref{lrp2}).
\end{proof}

\begin{rem} For the specialization  $q\rightarrow 1$ in  (\ref{r1})-(\ref{r4}),  realizations of  $\bar{\cal A}$ in  $U(sl_2)^{\otimes N }$ are obtained.
\end{rem}

\section{ The algebra $\bar{\cal A}_q$, alternating subalgebras of $U_q(\widehat{gl_2})$ and root vectors}\label{secDr}

Recall that the quantum  affine Kac-Moody algebra $U_q(\widehat{sl_2})$ admits a Drinfeld second presentation denoted $U_q^{Dr}$ with generators $\{{\tx}_k^{\pm}, \tho_{\ell},\tK^{\pm 1}, C^{\pm 1/2}|k\in {\mathbb Z},\ell\in {\mathbb Z}\backslash
\{0\} \}$ \cite{Dr,Beck,Ji}. For $q\rightarrow 1$, this presentation specializes to the universal enveloping algebra of $\widehat{sl_2}$ with generators   $\{{x}_k^{\pm}, h_{k},c|k\in {\mathbb Z} \}$ - called the Cartan-Weyl presentation - see e.g. \cite[top of page 566]{Beck}. 
 According to (\ref{AqpdefUqp})  (similarly  (\ref{Apdefsl2p})), a natural question concerns the interpretation of $\bar{\cal A}_q$   in terms of subalgebras of $U_q^{Dr}$ (and similarly for $\bar{\cal A}$ in terms of subalgebras of $\widehat{sl_2}$).
Although this problem may look complicated at first sight for $q\neq 1$,  it is  solved using the framework of Freidel-Maillet algebras combined with the results of Ding-Frenkel \cite{DF93}, as shown in this section.  In this section, we fix ${\mathbb K} =  {\mathbb C}$.\vspace{1mm}

We start with the simplified situation $q\rightarrow 1$, see Definition \ref{defalt} and Proposition \ref{pDrgl2c}.

\subsection{ The  algebra   $\bar{\cal A}$ and `alternating' subalgebras of $\widehat{gl_2}$}
The affine general Lie algebra $\widehat{gl_2}$ admits a presentation of Serre-Chevalley type and Cartan-Weyl type, closely related with the presentations of the affine Lie algebra $\widehat{sl_2}$   \cite{Kac,GO}. Consider the presentation of 
 Cartan-Weyl type for $\widehat{gl_2}$.
 In the definition below, $[.,.]$ denotes the Lie bracket.
\begin{defn}(Cartan-Weyl presentation $\widehat{gl_2}^{CW}$) The affine general Lie algebra $\widehat{gl_2}$ over ${\mathbb C}$  is generated by  $\{x_k^\pm,\epsilon_{1,k},$ $\epsilon_{2,k},c|k\in {\mathbb Z}\}$ subject to the relations:
\beqa
&& \big[\epsilon_{i,k},\epsilon_{j,\ell} \big]  =  kc \delta_{i,j} \delta_{k+\ell,0}\ ,\label{CW1} \\
&& \big[\epsilon_{1,k},x^\pm_\ell   \big]  =  \pm  x^\pm_{k+\ell} \ ,\label{CW21}\\
&& \big[\epsilon_{2,k},x^\pm_\ell   \big]  =  \mp  x^\pm_{k+\ell} \ ,\label{CW22}\\
&& \big[ x^+_k, x^-_\ell   \big]  = \epsilon_{1,k+\ell}- \epsilon_{2,k+\ell} + \delta_{k+\ell,0} kc \ ,\label{CW3}\\
&& \big[ x^\pm_k, x^\pm_{k\pm1}  \big] =0\  \label{CW4}
 \eeqa
and  $c $ is central.
\end{defn}
Note the automorphism $\theta$  such that:
\beqa
\theta: &&   x^\pm_k \mapsto  x^\mp_k \ ,\quad \epsilon_{1,k} \mapsto  \epsilon_{2,k}\ ,\quad \epsilon_{2,k} \mapsto  \epsilon_{1,k} \ ,\quad c \mapsto c\ .\label{thet}
\eeqa
\vspace{1mm}
Let
\beqa
h_k=\epsilon_{1,k} -\epsilon_{2,k}\ .\label{defh}
\eeqa
The subalgebra generated by $\{x^\pm_k,h_k,c|k\in {\mathbb Z}\}$, denoted  $\widehat{sl_2}^{CW}$,   is isomorphic to  the affine Lie algebra  $\widehat{sl_2}$. The commutation relations are given by   (\ref{CW3}), (\ref{CW4}) with (\ref{defh}) and
\beqa
&& \big[ h_k,h_\ell   \big]  =  \delta_{k+\ell,0} 2kc\ ,\label{CWsl21} \\
&& \big[ h_k,x^\pm_\ell   \big]  =  \pm 2 x^\pm_{k+\ell} \ .\label{CW2sl2}
\eeqa
Recall the Serre-Chevalley presentation  $\widehat{sl_2}^{SC}$ in Appendix \ref{apA}.
\begin{rem} An isomorphism $\widehat{sl_2}^{SC} \rightarrow \widehat{sl_2}^{CW}$ is given by:
\beqa
k_0 \mapsto -h_0-c\ , \quad k_1 \mapsto h_0 \ ,\quad
 e_1 \mapsto x_0^+\ , \quad e_0 \mapsto x_1^-\ ,\quad
f_1 \mapsto x_0^-\ , \quad f_0 \mapsto x_{-1}^+\ , \quad c \mapsto -c\ . \nonumber
\eeqa
\end{rem}

In view of (\ref{Apdefsl2p}), we now study the relation between $\bar{\cal A}$ and  $\widehat{gl_2}$.  
Isomorphisms between certain subalgebras of $\widehat{gl_2}$ and $\bar{\cal A}$ can be  identified  through a direct comparison of the defining relations  (\ref{CW1})-(\ref{CW4}) and (\ref{po1})-(\ref{po4}). However, although not necessary for $q= 1$, to prepare the analysis for $q\neq 1$ in the next section it is  instructive to exhibit these isomorphisms using the FRT presentation of  $U(\widehat{gl_2})$,  which follows from $U(\widehat{sl_2})$'s one\footnote{We expect this presentation appears in the literature, although we could not find a reference.
Here it is taken from \cite{BBC}.}. \vspace{1mm}

Introduce the following classical  (traceless) r-matrix for an indeterminate $z\neq 1$ associated with  $\widehat{sl_2}$:
\begin{equation}\label{def:r}
 r(z)=\frac{1}{z-1}\begin{pmatrix}
       -\frac{1}{2}(z+1)&0&0&0\\
       0&\frac{1}{2}(z+1)& -2&0\\
       0&-2z& \frac{1}{2}(z+1) &0\\
        0&0&0&-\frac{1}{2}(z+1)
      \end{pmatrix} \ .
\end{equation}
Note that  $r_{12}(z)=-r_{21}(1/z)=-r_{12}(z)^{t_1t_2}$. It satisfies the classical Yang-Baxter equation
\begin{equation}\label{eq:CYBE}
  [\ r_{13}(z_1/z_3)\ , \ r_{23}(z_2/z_3)\ ]=[\ r_{13}(z_1/z_3)+  r_{23}(z_2/z_3)\ , \ r_{12}(z_1/z_2)\ ]\;.
\end{equation}
 For simplicity, we keep the same notation for the generators of $U(\widehat{sl_2})$ and  $\widehat{sl_2}$. Defining:
\begin{eqnarray}
 T^{+}(z)&=&   \begin{pmatrix}
          h_0/2 &2  x^-_0 \\
       0 &   -h_0/2
                   \end{pmatrix}   + \sum_{k\geq 1}z^{k}\begin{pmatrix}
           h_{k} &2x^-_{k} \\
        2 x^+_{k}&   -h_{k}
                   \end{pmatrix},\label{Tp}
                   \\
 T^{-}(z)&=& \begin{pmatrix}
         - h_0/2 & 0 \\
       -2x^+_0 &   h_0/2
                   \end{pmatrix}   + \sum_{k\geq 1}z^{-k} \begin{pmatrix}
       -  h_{- k} & -2x^-_{-k}\\
        -2 x^+_{-k} &   h_{- k} 
                   \end{pmatrix},\label{Tm}
\end{eqnarray}
one checks that the relations\footnote{We denote $r'(z)= \frac{d}{d}r(z)$\ . } 
\begin{eqnarray}
\null[ T^{\pm}(z),c]&=&0\;,\label{Tg}\\
\null[ T_1^{\pm}(z),T^{\pm}_2(w)]&=&[ T_1^{\pm}(z)+T^{\pm}_2(w),r_{12}(z/w)]\;,\label{rpp}\\
\null[ T_1^{+}(z),T^{-}_2(w)]
&=&[ T_1^{+}(z)+T^{-}_2(w),r_{12}(z/w)]-2c\,r'_{12}(z/w)z/w\ ,\label{rpm}
\;
\end{eqnarray}
are equivalent to the relations \eqref{CW3}, \eqref{CW4},  \eqref{CWsl21},  \eqref{CW2sl2}, where  $\big[.,.\big]$ now denotes the usual commutator $\big[.,.\big]_1$.
The FRT presentation for $U(\widehat{gl_2})$ is obtained from (\ref{Tp}), (\ref{Tm}) as follows. Define the $2\times 2$ matrix  
\beqa
 H^\pm(z) = \pm \left( \frac{1}{2}(\epsilon_{1,0}+ \epsilon_{2,0}) + \sum_{k\geq 1} z^{\pm k}(\epsilon_{1,\pm k}+ \epsilon_{2,\pm k}) \right)I\!\!I\ . \nonumber
\eeqa
The corresponding pair of Lax operators for $U(\widehat{gl_2})$ is given by    $T_{\widehat{gl_2}}^\pm(z)= T^\pm(z) + H^\pm(z) $, and satisfy classical Yang-Baxter relations that follow from (\ref{Tg})-(\ref{rpm}).\vspace{1mm}

We now relate $\bar{\cal A}$ to certain subalgebras of $\widehat{gl_2}$ using the   FRT presentation. By straightforward computation, it is found that 
\beqa
B(u) \mapsto \tilde B^-(u) = - T_{\widehat{gl_2}}^-(u^2) - t_0 \quad \mbox{or}\quad B(u) \mapsto \tilde B^+(u) = T_{\widehat{gl_2}}^+(u^{-2}) - t_0\  \label{BDr}
\eeqa
with   $t_0=diag(\epsilon_{1,0},\epsilon_{2,0}) $, satisfy the non-standard classical Yang-Baxter equation (\ref{eq:nsYBa}) for the identification $\bar r(u,v)= - r(u^2/v^2)-r_0$, where $r_0=diag(1/2,-1/2,-1/2,1/2) $. In particular, let us consider the first map in (\ref{BDr}). Applying a similarity transformation:
\beqa
B^-(u) = -  M(u)\tilde B^-(u)^t M(u)^{-1} \quad \mbox{with} \quad 
 M(u)&=&   \begin{pmatrix}
          0 & -u \\
     1   &  0
                   \end{pmatrix}  \nonumber
\eeqa
one finds for instance that 
\beqa
B^-(u) =   \begin{pmatrix} 
        0  & 0 \\
     2u^{-1}x_0^+ & 0 
                   \end{pmatrix} +  \sum_{k\geq 1}u^{-2k} \begin{pmatrix}
       2 \epsilon_{1,- k} & 2ux^-_{-k}\\
        2u^{-1} x^+_{-k} &   2 \epsilon_{2,- k}  \label{Bmu}
                   \end{pmatrix}
 \eeqa
satisfies  (\ref{eq:nsYBa}) for the symmetric $r$-matrix  (\ref{r12basic}). Similarly, from the second map in  (\ref{BDr}) one gets a second solution of  (\ref{eq:nsYBa})  with (\ref{r12basic}):
\beqa
B^+(u) =   \begin{pmatrix} 
       0 &  2u^{-1}x_0^-  \\
    0 & 0
                   \end{pmatrix} +  \sum_{k\geq 1}u^{-2k} \begin{pmatrix}
        2 \epsilon_{1, k} &   2u^{-1} x^-_{k} \\
         2ux^+_{k} &   2 \epsilon_{2, k}  \label{Bpu}
                   \end{pmatrix} \ .
 \eeqa

 According to the structure of the matrices (\ref{Bmu}),  (\ref{Bpu}) and the automorphism (\ref{thet}), different subalgebras that combine half of the positive/negative root vectors, together with  half of the imaginary root vectors are now introduced. 
\begin{defn}\label{defalt}
\beqa
\widehat{gl_2}^{\triangleright,\pm} &=& \{x_{k}^\pm, x^\mp_{k+1}, \epsilon_{1,k+1},\epsilon_{2,k+1}|k\in {\mathbb N}\} \ ,\label{gsl2rpm}\\
\widehat{gl_2}^{\triangleleft,\pm}&=& \{x_{-k}^\pm, x^\mp_{-k-1}, \epsilon_{1,-k-1},\epsilon_{2,-k-1}|k\in {\mathbb N}\} \ .\label{gsl2lpm}
\eeqa
We call $\widehat{gl_2}^{\triangleright,\pm}$ and $\widehat{gl_2}^{ \triangleleft,\pm}$ the
 right and left alternating subalgebras of  $\widehat{gl_2}$.  The subalgebra generated by $\{\epsilon_{1,0},\epsilon_{2,0},c\}$ is denoted  $\widehat{gl_2}^{\diamond}$.
\end{defn}
Inserting (\ref{Bmu}) (resp. (\ref{Bpu})) into (\ref{eq:nsYBa}), the relations satisfied by the generators $\{x_{\pm k}^\pm,\epsilon_{1,\pm \ell},\epsilon_{2,\pm \ell}\}$ are extracted. They are  identical to the defining relations of the subalgebra $\widehat{gl_2}^{\triangleleft,+}$ (resp.  $\widehat{gl_2}^{ \triangleright,-}$). Thus,  FRT presentations for  $\widehat{gl_2}^{\triangleright,-}$ and  $\widehat{gl_2}^{\triangleleft,+}$ are given respectively by (\ref{Bpu}), (\ref{Bmu}) satisfying (\ref{eq:nsYBa}). Applying the automorphism (\ref{thet}) to (\ref{Bpu}), (\ref{Bmu}), one gets the FRT presentations of $\widehat{gl_2}^{\triangleright,+}$ and  $\widehat{gl_2}^{\triangleleft,-}$, respectively. \vspace{1mm}

In particular, combining above results with those of Section \ref{secFM} it follows:
\begin{prop}\label{pDrgl2c}   There exists an algebra isomorphism $\bar{\cal A}   \rightarrow  U( \widehat{gl_2}^{\triangleright,+})  $ (resp. $\bar{\cal A}   \rightarrow   U( \widehat{gl_2}^{\triangleleft,-})$) such that:
\beqa
&& {\tw}_{-k} \mapsto 2^{1-k} x^-_{k+1} \  ,\quad{\tw}_{k+1} \mapsto 2^{1-k} x^+_{k}  \ ,\ \quad \quad  {\tg}_{k+1}\mapsto 2^{3-k}\epsilon_{1,k+1}\ ,\quad \quad \tilde{\tg}_{k+1} \mapsto  2^{3-k}\epsilon_{2,k+1} \   \nonumber   \label{mapc1} \\
(\mbox{resp.}   &&   {\tw}_{-k} \mapsto 2^{1-k} x^-_{-k} \  ,\ \ \quad{\tw}_{k+1} \mapsto 2^{1-k} x^+_{-k-1}  \ ,\quad   {\tg}_{k+1}\mapsto 2^{3-k}\epsilon_{1,-k-1}\ ,\quad \tilde{\tg}_{k+1} \mapsto  2^{3-k}\epsilon_{2,-k-1}  \label{mapc2}\ . \nonumber
\eeqa
\end{prop}
\begin{proof} 
Identify  $\theta(B^\pm(u))$ for (\ref{Bpu}), (\ref{Bmu}), to  (\ref{eq:BW}). 
\end{proof}
Observe that the elements $\delta_{k+1}^\pm =\epsilon_{1,\pm (k+1)}+\epsilon_{2,\pm ( k+1)} $ are central. If we denote   $\textsc{z}^\pm=\{ \delta_{k+1}^\pm\}_{k\in {\mathbb N}}$ and introduce  the alternating subalgebras  $\widehat{sl_2}^{\triangleright,+}=\{ x^+_{k}, x_{k+1}^-, h_{k+1}|k\in {\mathbb N}\}$ (resp. $\widehat{sl_2}^{ \triangleleft,-}=\{x_{-k}^-, x^+_{-k-1}, h_{-k-1}|k\in {\mathbb N}\}$), in addition to (\ref{Apdefsl2p}) one has the decompositions $\widehat{gl_2}^{\triangleright,+}= \widehat{sl_2}^{\triangleright,+} \oplus \textsc{z}^+ $ and   $\widehat{gl_2}^{\triangleleft,-}= \widehat{sl_2}^{\triangleleft,-} \oplus \textsc{z}^- $. So, the images become:
\beqa
&&{\tg}_{k+1}\mapsto 2^{2-k}(h_{k+1} +  \delta_{k+1}^+)\ ,\quad \tilde{\tg}_{k+1} \mapsto  2^{2-k}(-h_{k+1} +   \delta_{k+1}^+) \\
(\mbox{resp.}   &&    {\tg}_{k+1}\mapsto 2^{2-k}(h_{-k-1} +  \delta_{k+1}^-)\ ,\quad \tilde{\tg}_{k+1}  \mapsto  2^{2-k}(-h_{-k-1}+\delta_{k+1}^-) )  \ .
\eeqa 

In the next section, by analogy we use the Freidel-Maillet type presentation given in Section \ref{secFM} to derive $q$-analogs of the isomorphisms of Proposition \ref{pDrgl2c}.

\subsection{ The  algebra   $\bar{\cal A}_q$ and `alternating' subalgebras of $U_q(\widehat{gl_2})$}
The Drinfeld second presentation  \cite{Jing,FMu} and FRT presentation of  $U_q(\widehat{gl_2})$ \cite{RS,DF93} are first reviewed,  see Definition \ref{def:UqDrgl2} and  Theorem \ref{def:UqRS}. Then,  `alternating' subalgebras  of $U_q(\widehat{gl_2})$ that can be viewed as $q$-analogs of (\ref{gsl2rpm}),  (\ref{gsl2lpm}) are identified, see Definition \ref{defaltq}. Using the  Ding-Frenkel isomorphism \cite{DF93}, K-matrices $K^\pm(u)$ (or $K'^+(u)$) that satisfy the Freidel-Maillet type equation  (\ref{RE}) (or (\ref{REp})) are constructed using a dressing procedure, see Lemmas \ref{Kmu}, \ref{Kpu} or \ref{K'pu}.   By a direct comparison of the K-matrix (\ref{K})  (resp. (\ref{K'})) to the K-matrix  $K^-(u)$ (resp. $K'^+(u)$),  explicit  isomorphisms from  $\bar{\cal A}_q$ to alternating subalgebras of $U_q(\widehat{gl_2})$ are derived, see Propositions \ref{map1}, \ref{map2}. For the first generators, Examples \ref{exi1}, \ref{exi2} are given.  

\subsubsection{Drinfeld second presentation and FRT presentation of $U_q(\widehat{gl_2})$} 
In this subsection, we review some necessary material.
For the quantum affine Kac-Moody algebra $U_q(\widehat{sl_2})$, there are two standard presentations: the {\it Drinfeld-Jimbo} presentation denoted $U_q^{DJ}$ and the  {\it Drinfeld (second) presentation} denoted $U_q^{Dr}$, see e.g. \cite[p. 392]{CPb}, \cite{Da14}. For  $U_q(\widehat{gl_2})$, an analog of Drinfeld second presentation is  known \cite{Jing,FMu}.
\begin{defn}\label{def:UqDrgl2}
  The quantum affine algebra   $U_q(\widehat{gl_2})$ is isomorphic to the associative algebra over ${\mathbb C}(q)$ with generators $\{{\tx}_k^{\pm}, \tep_{1,\ell},\tep_{2,\ell},
\tK^{\pm 1} |k\in {\mathbb Z},\ell\in {\mathbb Z}\backslash
\{0\} \}$, central elements $C^{\pm 1/2}$ and the following relations:  
\beqa
&&  C^{1/2}C^{-1/2}=1\ ,\quad \tK\tK^{-1}=\tK^{-1}\tK=1 \ ,\label{gl1}\\
&& \big[ \tep_{i,k},\tep_{j,\ell}\big]= \frac{\big[ k \big]_q}{k} \frac{C^k-C^{-k}}{q-q^{-1}} \delta_{i,j}\delta_{k+\ell,0} \ ,\quad  \tK\tep_{i,k} = \tep_{i,k} \tK \ ,\\
&& \big[ \tep_{1,k}, \tx^\pm_\ell \big]  = \pm \frac{\big[ k\big]_q}{k} C^{\mp |k|/2} q^{|k|/2} \tx^\pm_{k+\ell}\ ,\\
&& \big[ \tep_{2,k}, \tx^\pm_\ell \big]  = \mp \frac{\big[ k\big]_q}{k} C^{\mp |k|/2} q^{-|k|/2} \tx^\pm_{k+\ell}\ ,\\
&& \tK \tx^\pm_k \tK^{-1}    =  q^{\pm 2} \tx^\pm_{k} \ ,\label{gl2}\\
&& \tx^\pm_{k+1} \tx^\pm_\ell  -q^{\pm 2} \tx^\pm_{\ell} \tx^\pm_{k+1} = q^{\pm 2} \tx^\pm_{k} \tx^\pm_{\ell+1} - \tx^\pm_{\ell+1} \tx^\pm_k \ ,\label{gl3}\\     
&& \big[ \tx^+_k, \tx^-_\ell   \big]  = \frac{(C^{(k-\ell)/2} \psi_{k+\ell} - C^{-(k-\ell)/2}\phi_{k+\ell})}{q-q^{-1}}\ ,\label{gl4}
\eeqa
where the $\psi_{k}$ and $\phi_{k}$  are defined by the following equalities of formal power series in the indeterminate $z$:
\beqa
\psi(z)&=&\sum_{k=0}^\infty \psi_{k} z^{-k} = \tK \exp\left( (q-q^{-1}) \sum_{k=1}^{\infty} \tho_k  z^{-k} \right) \ ,\label{psi}\\
 \phi(z)&=& \sum_{k=0}^\infty \phi_{-k} z = \tK^{-1} \exp\left(- (q-q^{-1}) \sum_{k=1}^{\infty} \tho_{-k}  z \right)\ ,\label{phi}
\eeqa
where we denote:
\beqa
\tho_k =  q^{|k|/2}\tep_{1,k}-  q^{-|k|/2}\tep_{2,k} \ .\label{thok}
\eeqa
\end{defn}

Note that there exists a $q$-analog of the automorphism (\ref{thet})   such that:
\beqa
\theta: &&   \tx^\pm_k \mapsto  \tx^\mp_k \ ,\quad \cal E_{1,k} \mapsto  \cal E_{2,k}\ ,\quad \cal E_{2,k} \mapsto  \cal E_{1,k} \ ,\quad \tK \mapsto \tK\ ,\quad C \mapsto C^{-1}, \ \quad q \mapsto q^{-1}\ .\label{thetq}
\eeqa
In addition, there exists an automorphism:
\beqa
\nu: &&   \tx^+_k \mapsto  \tK\tx^+_k \ ,\quad  \tx^-_k \mapsto  \tx^-_k \tK^{-1} \ ,\quad \cal E_{1,k} \mapsto  \cal E_{1,k}\ ,\quad \cal E_{2,k} \mapsto  \cal E_{2,k}  ,\quad   \tK \mapsto \tK\ ,\quad C^{1/2} \mapsto C^{1/2}\ .\label{nu}
\eeqa

The associative subalgebra generated by  $\{{\tx}_k^{\pm}, \tho_{\ell},
\tK^{\pm 1},C^{\pm 1/2} |k\in {\mathbb Z},\ell\in {\mathbb Z}\backslash
\{0\} \}$ is isomorphic to the quantum affine algebra  $U_q(\widehat{sl_2})$, known in the literature as the  Drinfeld second presentation  $U_q^{Dr}$. The corresponding defining relations are given by (\ref{gl1}), (\ref{gl2})-(\ref{gl4}) and
\beqa
&&\big[ \tho_k,\tho_\ell   \big]  =  \delta_{k+\ell,0}\frac{1}{k}\big[ 2k \big]_q \frac{C^k-C^{-k}}{q-q^{-1}} \ ,\label{hh} \\
&& \big[ \tho_k, \tx^\pm_\ell \big]  = \pm \frac{1}{k}\big[ 2k\big]_q C^{\mp |k|/2} \tx^\pm_{k+\ell}\ .\label{hx}
\eeqa

\begin{rem} Recall  the defining relations of  $U_q^{DJ}$ in  Appendix \ref{apA}. An isomorphism $U_q^{DJ}\rightarrow U_q^{Dr}$ is given by (see e.g \cite[p. 393]{CPb}):
\beqa
\qquad
 K_0 \mapsto C\tK^{-1}\ , \quad K_1 \mapsto \tK \ ,\quad
 E_1 \mapsto  \tx_0^+\ , \quad E_0 \mapsto  \tx_1^-\tK^{-1}\ ,\quad
 F_1 \mapsto \tx_0^-\ , \quad F_0 \mapsto \tK\tx_{-1}^+\ .\label{isoCP}
\eeqa
\end{rem}

Note that it is still an open problem to find the complete Hopf algebra isomorphism between $U_q^{DJ}$ and $U_q^{Dr}$. Only partial information is known, see e.g. \cite[Section 4.4]{CP}.

Extending previous works \cite{FRT89,RS},  for  the quantum affine Lie algebra of type $A$   such as $U_q(\widehat{gl_n})$ a FRT presentation has been obtained in \cite{DF93}. For type $B,C,D$, see \cite{JLM,JLMBD}. The explicit isomorphism between  the Drinfeld second presentation of $U_q(\widehat{gl_2})$ and FRT presentation  given in \cite{DF93}  is now recalled. Define:
\begin{equation}\label{def:r}
\tilde R(z)=\begin{pmatrix} 1
      &0&0&0\\
       0&\frac{z-1}{zq-q^{-1}}& \frac{ z(q-q^{-1})}{zq-q^{-1}}&0\\
       0&  \frac{(q-q^{-1})}{zq-q^{-1}} & \frac{ z-1}{zq-q^{-1}} &0\\
        0&0&0&1
      \end{pmatrix} \ 
\end{equation}
which satisfies the Yang-Baxter equation (\ref{YB}).
Note that  $\tilde R_{12}(z)=\tilde R_{21}^{t_1t_2}(z)$. The above $R$-matrix is related to the symmetric $R$-matrix (\ref{R}) through the similarity transformations:
\beqa
\qquad \left(\frac{u}{v}q-\frac{v}{u}q^{-1}\right)^{-1} R_{12}(u/v) &=& \cal M(u)_1  \cal M(v)_2 \tilde R_{12}(u^2/v^2) \cal M(v)_2^{-1} \cal M(u)_1^{-1}\ ,\label{simil}\\
&=& {\cal M}(u)_1^{-1}   {\cal M}(v)_2^{-1} \tilde R_{21}(u^2/v^2) {\cal M}(v)_2 {\cal M}(u)_1 \quad \mbox{with} \quad 
 {\cal M}(u)=   \begin{pmatrix}
         u^{-1/2} & 0 \\
    0   &  u^{1/2} 
                   \end{pmatrix} \ .\nonumber
\eeqa
\begin{thm}\label{def:UqRS} (see \cite{RS,DF93}) $U_q(\widehat{gl_2})$ admits a FRT presentation given by a unital associative algebra with generators  $\{{\tx}_k^{\pm}, \tk^+_{j,-\ell}, \tk^-_{j,\ell}, q^{\pm c/2} |k\in {\mathbb Z},\ell\in {\mathbb N},j=1,2 \}$. The  generators  $q^{\pm c/2}$ are central and mutally inverse. Define:
\beqa
&&  L^\pm(z)=
       \begin{pmatrix}
     \tk_1^\pm(z)  &    \tk_1^\pm(z)  \tf^\pm(z)    \\
   \te^\pm(z) \tk_1^\pm(z)     &   \tk_2^\pm(z) +   \te^\pm(z)  \tk_1^\pm(z)    \tf^\pm(z)   
      \end{pmatrix} \ \label{Lpm}
\eeqa
in terms of  the generating functions in the indeterminate $z$:
\beqa
{\te}^+(z)&=& (q-q^{-1})\sum_{k=0}^\infty q^{k(c/2-1)} {\tx}_{-k}^- z^{k} \ , \quad {\te}^-(z)=-(q-q^{-1})\sum_{k=1}^\infty q^{k(c/2+1)}{\tx}^-_{k}z^{-k} \ ,\label{eq:cuDre}\\
{\tf}^+(z)&=&(q-q^{-1})\sum_{k=1}^\infty q^{-k(c/2+1)}{\tx}^+_{-k}z^{k} \ , \quad {\tf}^-(z)=-(q-q^{-1})\sum_{k=0}^\infty q^{-k(c/2-1)}{\tx}^+_{k}z^{-k} \ ,\label{eq:cuDrf}\\
\quad  {\tk}_{j}^+(z)&=&\sum_{k=0}^\infty{{\tk}}^+_{j,-k}z^{k}\  \ ,\quad \qquad \qquad {\tk}_{j}^-(z)=\sum_{k=0}^\infty{{\tk}}^-_{j,k}z^{-k}\ ,\quad j=1,2\ .\label{eq:cuDrk}
\eeqa
The defining relations are the following:
\beqa
\tk^+_{i,0}\tk^-_{i,0}&=& \tk^-_{i,0}\tk^+_{i,0}=1 \ ,\\ 
 \tilde  R(z/w)\ (L^\pm(z)\otimes I\!\!I)\ ( I\!\!I \otimes L^\pm(w))  &=&  ( I\!\!I \otimes L^\pm(w))\  (L^\pm(z)\otimes I\!\!I)   \  \tilde R(z/w)  \ ,  \label{YBApm1}\\
\tilde R(q^{c}z/w)\ (L^+(z)\otimes I\!\!I)\ ( I\!\!I \otimes L^-(w))  &=&  ( I\!\!I \otimes L^-(w))\  (L^+(z)\otimes I\!\!I)   \ \tilde R(q^{-c}z/w)  \ .  \label{YBApm2}
 \eeqa
For (\ref{YBApm1}), the expansion direction of $ \tilde R(z/w)$ can be chosen in $z/w$ or $w/z$, but for 
(\ref{YBApm2}) the expansion direction is only in $z/w$.
$U_q(\widehat{gl_2})$ is a Hopf algebra. The coproduct $\Delta$  is defined by:
\beqa
\Delta(L^\pm(z)) = (L^\pm(zq^{\pm(1\otimes c/2)}))_{[\textsf 1]} (L^\pm(zq^{\mp(c/2\otimes 1)}))_{[\textsf 2]} \label{coprodUqgl2}
\eeqa
and its antipode is $S(L^\pm(z))=L^\pm(z)^{-1}$.
\end{thm}
\begin{rem} The inverse quantum Lax operators (\ref{Lpm}) are \cite[eq. (4.9)]{DF93}:
\beqa
&& ( L^\pm(z))^{-1}=
       \begin{pmatrix}
    (\tk_1^\pm(z))^{-1} +   \tf^\pm(z)  (\tk_2^\pm(z))^{-1}    \te^\pm(z)    &   -   \tf^\pm(z) (\tk_2^\pm(z))^{-1}   \\
  -    (\tk_2^\pm(z))^{-1} \te^\pm(z)  &  (\tk_2^\pm(z))^{-1} 
      \end{pmatrix} \ \label{Lpminv}.
\eeqa
\end{rem}

The explicit isomorphism between the FRT presentation of Theorem \ref{def:UqRS} and Drinfeld second presentation of $U_q(\widehat{gl_2})$ of Definition \ref{def:UqDrgl2} is given in \cite[Section 4]{Jing}. Introduce the generating functions \cite{DF93}:
\beqa
\tx^\pm(z) = \sum_{k\in {{\mathbb Z}}}  \tx^\pm_k z^{-k} \  .\label{xcu}
\eeqa
In terms of (\ref{eq:cuDre}), (\ref{eq:cuDrf}), one has:
\beqa
\tx^+(z) = (q-q^{-1})^{-1}\left( \tf^+(q^{c/2+1}z) -  \tf^-(q^{-c/2+1}z) \right)\ , \nonumber \\
\tx^-(z) = (q-q^{-1})^{-1}\left( \te^+(q^{-c/2+1}z) -  \te^-(q^{c/2+1}z) \right)\   \nonumber
\eeqa
and
\beqa
C^{1/2}=q^{c/2}\ . \nonumber
\eeqa
The generating functions $\{\tk_i^\pm(z)\}_{i=1,2}$ are related with the generators $\{\tep_{i,k}\}_{i=1,2}$ as follows  \cite[Section 4]{Jing} (see also \cite{FMu}):
\beqa
\tk_i^\pm(z)=\tk_{i,0}^{\pm } \exp\left( \pm (q-q^{-1})\sum_{n= 1}^\infty a_{i,\mp n} z^{\pm n}\right) \  \label{kpmz}
\eeqa
where the new generators
\beqa
a_{1,m}&=& q^{m}\left(q^{|m|/2}\tep_{1,m}-q^{-|m|/2}\tep_{2,m} \right) +a_{2,m} \ ,\label{a1m}\\
 a_{2,m} &=& q^{2m + |m|/2}\left(  \frac{|m|}{m}\frac{\tep_{1,m}+q^{|m|}\tep_{2,m}}{(1+q^{2|m|})^{1/2}} + \tep_{2,m}\right)\ \label{a2m}\ ,
\eeqa
are introduced. The generators $\tk_{i,0}^{\pm }$ are such that $\big[\tk_{i,0}^{\pm },a_{j,m}\big]=\big[\tk_{i,0}^{\epsilon},\tk_{j,0}^{\epsilon' }\big]=0$ for any $i,j$ and
\beqa
\tk_{2,0}^-(\tk_{1,0}^-)^{-1}=\tK\  ,\qquad \tk_{2,0}^+(\tk_{1,0}^+)^{-1}=\tK^{-1}\  . \label{kk}
\eeqa

The commutation relations of  $U_q(\widehat{gl_2})$ presented in terms of the generators $\{a_{i,m}|i=1,2\}$ are given in \cite[Section 4]{Jing}. Although not reported here, for further analysis some of those are displayed in Appendix \ref{apB}.\vspace{1mm}

In the context of the FRT presentation of $U_q(\widehat{gl_2})$ \cite{DF93}, the explicit exchange relations between the  generating functions (\ref{eq:cuDre})-(\ref{eq:cuDrk}) are extracted from  (\ref{YBApm1}),  (\ref{YBApm2}) inserting (\ref{Lpm}). We refer the reader to \cite[p. 288-292]{DF93} for details. In particular, for the following analysis, we will need the asymptotics of some of the exchange relations displayed in \cite{DF93}. Considering the limits $\tk^+_j(0)$  and $\tk^-_j(\infty)$ of (\ref{eq:cuDrk}), from \cite[eqs. (4.24), (4.25), (4.40), (4.41)]{DF93} one gets for instance:
\beqa
\tk_{1,0}^\pm \te^\pm(w) (\tk_{1,0}^\pm )^{-1}= q^{\mp 1} \te^\pm(w) \ ,\quad \tk_{1,0}^\pm \tf^\pm(w) (\tk_{1,0}^\pm )^{-1}= q^{\pm 1} \tf^\pm(w) \ ,\label{redkef1}\\
(\tk_{2,0}^\pm)^{-1} \te^\pm(w)\tk_{2,0}^\pm= q^{\mp 1} \te^\pm(w) \ ,\quad 
(\tk_{2,0}^\pm)^{-1}  \tf^\pm(w)\tk_{2,0}^\pm= q^{\pm 1} \tf^\pm(w) \ ,\label{redkef2}
\eeqa
and from \cite[eqs. (4.13), (4.14), (4.17)]{DF93} one gets:
\beqa
\tk_{i,0}^\pm \tk_{j}^\pm(w)= \tk_j^\pm(w) \tk_{i,0}^\pm\ , \quad \tk_{i,0}^\pm \tk_i^\pm(w)= \tk_i^\pm(w) \tk_{i,0}^\pm \ , \quad i\neq j =1,2\ .\label{redkk}
\eeqa

 To prepare the discussion in further sections, the description of the known embedding $U_q(\widehat{sl_2})\hookrightarrow U_q(\widehat{gl_2})$ is now recalled. First, central elements of $U_q(\widehat{gl_2})$ are constructed using the FRT presentation. Following \cite[Section 2.6]{FMu}, define the generating functions:
\beqa
y^\pm(z)=\tk_1^\mp(q^{-1}z)\tk_2^\mp(qz)\ .\label{ypmz}
\eeqa
By \cite[eq. (4.17)]{DF93}, note that the ordering of the factors in (\ref{ypmz}) is irrelevant. 
Using the other exchange relations in \cite{DF93}, one finds $\big[y^\pm(z),\te^\epsilon(w)\big]=\big[y^\pm(z),\tf^\epsilon(w)\big]= \big[y^\pm(z),\tk_1^\epsilon(w)\big]=\big[y^\pm(z),\tk_2^\epsilon(w)\big] =0$ for $\epsilon=\pm $ and any $z,w$. 
\begin{prop} (see \cite{FMu}) The coefficients of the generating function $y^\pm(z)$ are central elements of $U_q(\widehat{gl_2})$. 
\end{prop}
\begin{cor}\label{cor:cent} The elements 
\beqa
\tk^\mp_{1,0}\tk^\mp_{2,0}\ \quad \mbox{and}\quad \gamma_m = q^m a_{1,m} + q^{-m}a_{2,m} \quad \mbox{for}\quad m\in {\mathbb Z}^*\  \label{expgam}
\eeqa
are central in $U_q(\widehat{gl_2})$.
\end{cor}
\begin{proof} Insert (\ref{kpmz}) into (\ref{ypmz}). Identify the coefficients of the resulting power series $y^\pm(z)$.
\end{proof}
Note that $\big[U_q(\widehat{gl_2}),y\big]=0$ for $y=\tk^\pm_{1,0}\tk^\pm_{2,0},\gamma_m$ can be independently checked  using  (\ref{a1m}), (\ref{a2m}) and the commutation relations (\ref{comam1})-(\ref{comam4}).

\begin{rem} In terms of the generators $\tho_m$ (\ref{thok}) and central elements $\gamma_m$  (\ref{expgam}), the new generators $a_{1,m},a_{2,m}$ entering in (\ref{kpmz}) decompose as:
\beqa
a_{1,m}= \frac{q^m}{1+q^{2m}}(\tho_m + \gamma_m)\ ,\label{a1mbis}\qquad
a_{2,m}=  \frac{q^m}{1+q^{-2m}}(-\tho_m + q^{-2m} \gamma_m)\ \label{a2mbis}\ .
\eeqa
\end{rem}
It is known that the elements (\ref{expgam}) and $C^{\pm 1/2}$ generate\footnote{I thank  N. Jing for communications on this point. Note that the analogs of $y^{\pm}(z)$ are known for higher rank affine Lie algebras of type A,B,C,D \cite{FMu,JLM,JLMBD}.} the center of $U_q(\widehat{gl_2})$. The following arguments are described in \cite{FMu} (see also \cite{JLM}). Denote ${\cal C}$ the subalgebra generated by (\ref{expgam}). One has the embedding $U_q^{Dr} \otimes \cal C   \hookrightarrow  U_q(\widehat{gl_2})$\ .
Furthermore, define ${U'}_q^{Dr}$ as the extension of $U_q^{Dr}$ by $q^{\pm 1/2},K^{\pm 1/2}$, and define $\cal C'$ as the extension of  $\cal C$ by   $(\tk^\pm_{1,0}\tk^\pm_{2,0})^{1/2}$. Then, one has the inverse embedding $
 U_q(\widehat{gl_2})    \hookrightarrow  {U'}_q^{Dr}  \otimes  \cal C'$. It follows that $U_q(\widehat{gl_2})$ and   ${U'}_q^{Dr}\otimes \cal C$ are ``almost'' isomorphic. So, one has the tensor product decomposition:
\beqa
 U_q(\widehat{gl_2}) \cong U_q^{Dr}\otimes \cal C \ . \label{isophiDr}
\eeqa
For more details, see e.g. \cite[Proposition 2.3, Corollary 2.4]{JLM}.  The  explicit  isomorphism $\varphi_{D}:  U_q(\widehat{gl_2})\rightarrow U_q^{Dr} \otimes \cal C $ is constructed along these lines. In view of these comments, $U_q^{Dr}$ can be considered as the quotient of the Drinfeld type presentation of $U_q(\widehat{gl_2})$ by the relations
\beqa
 y^\pm(z)=1\ \qquad \Longleftrightarrow \qquad  k^\pm_{1,0}k^\pm_{2,0}=1 \quad \mbox{and}\quad \gamma_m=0\ \quad \forall m \in {\mathbb Z}^*\ .\label{quoDr}
\eeqa
Below, we will use the surjective homomorphism $\gamma_D:  U_q(\widehat{gl_2})  \rightarrow  U_q^{Dr} $ using the presentation of Theorem \ref{def:UqRS}. Recall (\ref{xcu}) and (\ref{kpmz}). Using   (\ref{a2mbis}) and setting (\ref{quoDr}), for instance one has:
\beqa
&&\gamma_D( q^{c/2})  \mapsto C^{1/2}   \ ,\qquad \gamma_D(\tx^\pm(z)) \mapsto  \tx^\pm(z) \ ,\label{gamD1}\\
&&\gamma_D( a_{1,m})\mapsto \frac{1}{q^m+q^{-m}}\tho_m \ ,\qquad \gamma_D(a_{2,m}) \mapsto -\frac{q^{2m}}{q^m+q^{-m}}\tho_m\ ,\label{gamD2} \\
&&\gamma_D(\tk_{2,0}^\mp(\tk_{1,0}^\mp)^{-1}) \mapsto  \tK^{\pm 1} \label{gamD3}\ .
\eeqa
Thus, the FRT presentation of $U_q(\widehat{sl_2})$ is obtained as a corollary of \cite[Main Theorem]{DF93}. It is  given by the image of  (\ref{YBApm1}),  (\ref{YBApm2}) with (\ref{Lpm})  via $\gamma_D$.

\subsubsection{Alternating subalgebras   $U_q(\widehat{gl_2})^{\triangleright,\pm}$ and $U_q(\widehat{gl_2})^{\triangleleft,\pm}$  and  K-matrices}
By analogy with the analysis of previous section, we  need  to identify  $q$-deformed analogs of the ``classical'' alternating subalgebras (\ref{gsl2rpm}), (\ref{gsl2lpm}). For instance, consider the elements:
\beqa
C^{- k/2}\tK^{-1}\tx_{k}^+\ ,\quad  C^{(k+1)/2}\tx^-_{k+1}\ ,\quad  \tep_{1,k+1}\ ,\quad\tep_{2,k+1} \quad \mbox{for} \quad k\in {\mathbb N}\ .\label{elemalt}
\eeqa
Using the defining relations of $U_q(\widehat{gl_2})$, for $k,\ell \in {\mathbb N}$ one finds:
\beqa
&& \big[ \tep_{i,k},\tep_{j,\ell}\big]= 0\ , \nonumber\\ 
&& \big[ \tep_{1,k},  C^{- \ell/2}\tK^{-1}\tx^+_\ell \big]  =  \frac{\big[ k\big]_q}{k}q^{k/2} C^{-(k+\ell)/2}\tK^{-1}\tx^+_{k+\ell}\ ,\qquad
 \big[ \tep_{2,k},  C^{- \ell/2}\tK^{-1}\tx^+_\ell \big]  =  -\frac{\big[ k\big]_q}{k}  q^{-k/2}  C^{-(k+\ell)/2}\tK^{-1}\tx^+_{k+\ell}\ ,\nonumber\\
&&\big[ \tep_{1,k},  C^{(\ell+1)/2}\tx^-_{\ell+1} \big]  = -  \frac{\big[ k\big]_q}{k}  q^{k/2}  C^{(k+\ell+1)/2}\tx^-_{k+\ell+1}\ ,\qquad
 \big[ \tep_{2,k},  C^{(\ell+1)/2}\tx^-_{\ell+1} \big]  =   \frac{\big[ k\big]_q}{k}  q^{-k/2} C^{(k+\ell+1)/2}\tx^-_{k+\ell+1}\ .\nonumber
\eeqa
Furthermore, the relations (\ref{gl3}) are left invariant by the action of $ C^{- (k+\ell+1)/2}\tK^{-2}$ for $(++)$ or the action of  $ C^{(k+\ell+1)/2}$ for $(--)$. Also, using (\ref{gl2}), (\ref{gl4}) one finds:
\beqa
\big[ C^{- k/2}\tK^{-1}\tx_{k}^+\ ,\quad  C^{(\ell+1)/2}\tx^-_{\ell+1}\big] = \frac{1}{q-q^{-1}} \tK^{-1} \psi_{k+\ell +1}
 + (q^2-1) \left( C^{- k/2}\tK^{-1}\tx_{k}^+ \right)  \left(C^{(\ell+1)/2}\tx^-_{\ell+1} \right)\ .\nonumber
 \eeqa
According to (\ref{psi}), $\tK^{-1} \psi_{k}$ only depends on $\tho_k $ so it is a combination of $\tep_{1,k},\tep_{2,k}$. Thus, we conclude that
the elements (\ref{elemalt}) form a subalgebra of  $U_q(\widehat{gl_2})$. Other subsets of elements are similarly considered, which form different subalgebras. It follows:
\begin{defn}\label{defaltq}
\beqa
U_q(\widehat{gl_2})^{\triangleright,\pm} &=& \{C^{\mp k/2}\tK^{-1}\tx_{k}^\pm, C^{\pm (k+1)/2}\tx^\mp_{k+1}, \tep_{1,k+1},\tep_{2,k+1}|k\in {\mathbb N}\} \ ,\nonumber\\
U_q(\widehat{gl_2})^{\triangleleft,\pm}&=& \{C^{\mp k/2}\tx_{-k}^\pm, C^{\pm (k+1)/2} \tx^\mp_{-k-1}\tK, \tep_{1,-k-1},\tep_{2,-k-1}| k\in {\mathbb N}\} \ . \nonumber
\eeqa
 We call $U_q(\widehat{gl_2})^{\triangleright,\pm}$ and $U_q(\widehat{gl_2})^{ \triangleleft,\pm}$ the right and left 
alternating subalgebras of  $U_q(\widehat{gl_2})$.  The subalgebra generated by $\{\tK^{\pm 1},C^{\pm 1/2}\}$ is denoted  $U_q(\widehat{gl_2})^{\diamond}$.
\end{defn}

In each alternating subalgebra introduced above, the center is characterized as follows. Consider for instance $U_q(\widehat{gl_2})^{\triangleright,\pm}$. Its center is the subalgebra of $\cal C$ generated by some of the coefficients of the generating function $y^+(z)$ as defined in (\ref{ypmz}).
\begin{rem}\label{Ct} The center $\cal C^\triangleright$ (resp. $\cal C^\triangleleft$ ) of  $U_q(\widehat{gl_2})^{\triangleright,\pm}$ (resp.  $U_q(\widehat{gl_2})^{ \triangleleft,\pm}$) is generated by  
 $\gamma_m$ (resp. 
 $\gamma_{-m}$) with $m\in {\mathbb N}^*$.
\end{rem}

For   $U_q(\widehat{sl_2})$,  it is known that given a certain ordering the elements  $\{{\tx}_k^{\pm}, \tho_{\ell},\tK^{\pm 1}, C^{\pm 1/2} 
|k\in {\mathbb Z},\ell\in {\mathbb Z}\backslash
\{0\} \}$ generate a PBW basis, see \cite[Proposition 6.1]{Beck} with \cite[Lemma 1.5]{BCP}.  According to (\ref{thok}), with a minor modification in the Cartan sector associated with the decomposition of $\tho_k$ into $\tep_{1,k} , \tep_{2,k}$, a PBW basis for $U_q(\widehat{gl_2})$ is obtained. If one considers the subalgebra $U_q(\widehat{gl_2})^{\triangleright,+}$, let us choose the ordering:
\beqa
C^{1/2}\tx_1^- <  C\tx_2^- < \cdots < \tep_{1,1} < \tep_{1,2} < \cdots < \tep_{2,1} < \tep_{2,2} < \cdots < C^{-1/2} \tK^{-1}\tx_1^+ <  \tK^{-1} \tx_0^+ \ , \nonumber
\eeqa
whereas for the  subalgebra $U_q(\widehat{gl_2})^{\triangleleft,-}$ we choose the ordering:
\beqa
\tx_{0}^- < C^{1/2} \tx_{-1}^- < \cdots < \tep_{1,1} < \tep_{1,2} < \cdots < \tep_{2,1} < \tep_{2,2} < \cdots <C^{-1}\tx_{-2}^+ \tK < C^{-1/2}  \tx_{-1}^+\tK \ . \nonumber
\eeqa
It follows:
\begin{prop}\label{pbwalt} The vector space   $U_q(\widehat{gl_2})^{\triangleright,+}$  (resp. $U_q(\widehat{gl_2})^{\triangleleft,-}$)  has a linear basis consisting of the products $x_1x_2\cdots x_n$ $(n\in {\mathbb N})$ with $x_i\in U_q(\widehat{gl_2})^{\triangleright,+}$  (resp. $x_i\in U_q(\widehat{gl_2})^{\triangleleft,-}$)   such that  $x_1 \leq x_2 \leq \cdots \leq x_n$. 
\end{prop}
Using the automorphism (\ref{thetq}), PBW bases for $U_q(\widehat{gl_2})^{\triangleright,-}$ and $U_q(\widehat{gl_2})^{\triangleleft,+}$ are similarly obtained.\vspace{1mm}

We now turn to the construction of K-matrices satisfying the Freidel-Maillet type equations (\ref{RE}) or (\ref{REp}), whose entries are formal power series in the elements of alternating subalgebras. Assume there exists  a matrix $\tilde K^0$ with scalar entries and
two quantum Lax operators $ L(z) ,  L^{0}$ such that the following relations hold ($\tilde R_{21}(z)= P\tilde R_{12}(z)P$):
\beqa \tilde R_{12}(z/w) \  \tilde K^0_1 \ R^{(0)}\ \tilde K^0_2\
&=& \  \tilde K^0_2  \ R^{(0)} \ \tilde K^0_1 \ \tilde R_{21}(z/w)\  ,\label{RKzinit} \\
\tilde R_{12}(z/w)    L_1(z)  L_2(w)   &=& L_2(w)   L_1(z) \tilde R_{12}(z/w) \ ,\label{RtLpLp}\\
\tilde R_{21}(z/w)   (L^{0})_1(L^{0})_2 &=& (L^{0})_2(L^{0})_1  \tilde R_{21}(z/w) \ ,\label{RL0L0}\\
(L^{0})_1R^{(0)}  L_2(w)  &=&  L_2(w)   R^{(0)} (L^{0})_1\ ,\label{L0R0L}\\
L_1(z) R^{(0)}(L^{0})_2  &=& (L^{0})_2R^{(0)} L_1(z)  \  .\label{LR0L0}
\eeqa
Adapting \cite[Proposition 2]{Skly88},   using the above relations one finds that  :
\beqa
\tilde K(z) \mapsto L(z) \tilde K^0 L^{0}\label{Ktz}
\eeqa
satisfies the following Freidel-Maillet type equation (for a non-symmetric R-matrix)
\begin{align} \tilde R_{12}(z/w)\ (\tilde K(z)\otimes I\!\!I)\ R^{(0)}\ (I\!\!I \otimes \tilde K(w))\
= \ (I\!\!I \otimes \tilde K(w))\  R^{(0)}\ (\tilde K(z)\otimes I\!\!I)\ \tilde R_{21}(z/w)\  .\label{RKz}
\end{align}

An example built from the FRT presentation for  $U_q(\widehat{gl_2})$ of Theorem \ref{def:UqRS} is obtained as follows. For the choices 
\beqa
L(z) \mapsto  L^-(z)  \quad \mbox{and} \quad L^{0}\mapsto  L^{-,0}= diag( (\tk_{2,0}^-)^{-1}, (\tk^-_{1,0})^{-1})\ ,\label{Lm0i} 
\eeqa
 eq. (\ref{RtLpLp}) holds and using the exchange relations  (\ref{redkef1})-(\ref{redkk}) it is checked that eqs. (\ref{RL0L0})-(\ref{LR0L0}) hold.
Also, for the choice 
\beqa
\tilde K^{0} =  \begin{pmatrix} 0 & \frac{k_+ (q+q^{-1})}{(q-q^{-1})}  \\
    \frac{k_- (q+q^{-1})}{(q-q^{-1})}  & 0
      \end{pmatrix}\  \label{Ktilde0}
\eeqa
 it is checked that eq. (\ref{RKzinit})  holds. It follows 
\beqa
\tilde K(z) \mapsto \tilde K^-(z) = L^-(z) \tilde K^{0}  L^{-,0}\label{Ktmz}
\eeqa
satisfies (\ref{RKz}).
Note that  eq. (\ref{RKz}) is left invariant under the transformation $(z,w)\mapsto (\lambda z,\lambda w)$ for any $\lambda\in{\mathbb C}^*$.\vspace{1mm}

A solution of (\ref{RE}) associated with the symmetric $R$-matrix (\ref{R}) is readily obtained using the similarity transformation (\ref{simil}).
\begin{lem}\label{Kmu} The dressed $K$-matrix
\beqa
\qquad  K^-(u)=
       \begin{pmatrix}
 u^{-1}\left( \frac{k_-(q+q^{-1})}{q-q^{-1}} \tk_1^-(qu^2)\tf^-(qu^2)  (\tk^-_{2,0})^{-1} \right)  &   \frac{k_+(q+q^{-1})}{q-q^{-1}}  \tk_1^-(qu^2) (\tk^-_{1,0})^{-1}  \\
 \frac{k_-(q+q^{-1})}{q-q^{-1}}\left( \tk_2^-(qu^2) +   \te^-(qu^2)  \tk_1^-(qu^2)    \tf^-(qu^2)   \right)  (\tk^-_{2,0})^{-1}    &    u\left( \frac{k_+(q+q^{-1})}{q-q^{-1}}\te^-(qu^2) \tk^-_1(qu^2) (\tk^-_{1,0})^{-1} \right)  & 
      \end{pmatrix} \ \nonumber
\eeqa
satisfies the Freidel-Maillet type equation (\ref{RE}). 
\end{lem}
\begin{proof}The $K$-matrix $\tilde K^-(z)$ defined by 
 (\ref{Ktmz}) satisfies (\ref{RKz}). Applying the transformation  (\ref{simil}) to (\ref{RKz}) and defining
\beqa
 K^-(u)=  \cal M(u) \tilde K^- (qu^2)    {\cal M}(u)\ ,  \nonumber
\eeqa
 the claim follows.
\end{proof}

Another solution of (\ref{RE}) is obtained as follows.   Assume there exists
two quantum Lax operators $ L(z) ,  L^{0}$ such that the relations  (\ref{L0R0L}), (\ref{LR0L0}) and
\beqa
\tilde R_{21}(z/w)    L_1(z)  L_2(w)   &=& L_2(w)   L_1(z) \tilde R_{21}(z/w) \ ,\nonumber\\
\tilde R_{12}(z/w)   (L^{0})_1(L^{0})_2 &=& (L^{0})_2(L^{0})_1  \tilde R_{12}(z/w) \  \nonumber
\eeqa
are satisfied. 
%
%
It is straightforward to check that 
\beqa
L(z) \mapsto  (L^{+}(z^{-1}))^{-1}  \quad \mbox{and} \quad L^{0}\mapsto  L^{+,0} = diag(\tk_{2,0}^+,\tk_{1,0}^+)\  \label{Lp0i} 
\eeqa
obey the above set of relations. Then
\beqa
\tilde K(z) \mapsto \tilde K^+(z) =    L^{+,0}\tilde K^{0}   (L^{+}(z^{-1})^{-1})  \label{Ktpz}
\eeqa
satisfies  (\ref{RKz}). Using this result combined with the similarity transformation (\ref{simil}), it follows:
\begin{lem}\label{Kpu} The dressed $K$-matrix
\beqa
 K^+(u)=
  \begin{pmatrix}
    u^{-1}\left( -\frac{k_+(q+q^{-1})}{q-q^{-1}} \tk^+_{2,0}  \tk^+_2(1/qu^2)^{-1} \te^+(1/qu^2) \right)  &   \frac{k_+(q+q^{-1})}{q-q^{-1}}  \tk^+_{2,0}  \tk^+_2(1/qu^2)^{-1}  \\
 \frac{k_-(q+q^{-1})}{q-q^{-1}}\tk^+_{1,0}\left(   \tk^+_1(1/qu^2)^{-1} +   \tf^+(1/qu^2) \tk^+_2(1/qu^2)^{-1} \te^+(1/qu^2)\right)     &   u\left( -\frac{k_-(q+q^{-1})}{q-q^{-1}}\tk^+_{1,0}\tf^+(1/qu^2)\tk^+_2(1/qu^2)^{-1}    \right)
     \end{pmatrix} \ \nonumber
\eeqa
satisfies the Freidel-Maillet type equation (\ref{RE}). 
\end{lem}

For completeness, a K-matrix satisfying  (\ref{REp}) is now constructed along the same lines. To this aim, we consider the set of relations (\ref{RKzinit})-(\ref{LR0L0}) with the substitution:
\beqa
R^{(0)}  \rightarrow (R^{(0)})^{-1}\ .
\eeqa
For the choices
\beqa
L(z) \mapsto  L^+(z)  \quad \mbox{and} \quad L^{0}\mapsto  L'^{+,0}= diag( (\tk_{2,0}^+)^{-1}, (\tk^+_{1,0})^{-1})\ ,\label{Lp0bis} 
\eeqa
one finds that
\beqa
\tilde K(z) \mapsto \tilde K'^+(z) = L^+(z) \tilde K^{0}  L'^{+,0}\label{K'tpz}
\eeqa
satisfies  (for the non-symmetric R-matrix)
\begin{align} \tilde R_{12}(z/w)\ (\tilde K(z)\otimes I\!\!I)\  (R^{(0)})^{-1}\ (I\!\!I \otimes \tilde K(w))\
= \ (I\!\!I \otimes \tilde K(w))\   (R^{(0)})^{-1}\ (\tilde K(z)\otimes I\!\!I)\ \tilde R_{21}(z/w)\  .\label{RK'z}
\end{align}
Using (\ref{simil}), it follows:
\begin{lem}\label{K'pu} The dressed $K$-matrix
\beqa
\qquad  K'^+(u)=
       \begin{pmatrix}
 u^{-1}\left( \frac{k_-(q+q^{-1})}{q-q^{-1}} \tk_1^+(qu^2)\tf^+(qu^2)  (\tk^+_{2,0})^{-1} \right)  &   \frac{k_+(q+q^{-1})}{q-q^{-1}}  \tk_1^+(qu^2) (\tk^+_{1,0})^{-1}  \\
 \frac{k_-(q+q^{-1})}{q-q^{-1}}\left( \tk_2^+(qu^2) +   \te^+(qu^2)  \tk_1^+(qu^2)    \tf^+(qu^2)   \right)  (\tk^+_{2,0})^{-1}    &    u\left( \frac{k_+(q+q^{-1})}{q-q^{-1}}\te^+(qu^2) \tk^+_1(qu^2) (\tk^+_{1,0})^{-1} \right)  & 
      \end{pmatrix} \ \nonumber
\eeqa
satisfies the Freidel-Maillet type equation (\ref{REp}). 
\end{lem}

The entries of the K-matrices are formal power series in the elements of the alternating subalgebras. Consider for instance the entry $(K^-(u))_{11}$. One has:
\beqa
(K^-(u))_{11} &=&  u^{-1}q\left( \frac{k_-(q+q^{-1})}{q^2-1} \tk_1^-(qu^2)\underbrace{\tf^-(qu^2)  (\tk^-_{2,0})^{-1}}_{= q  (\tk^-_{2,0})^{-1} \tf^-(qu^2) } \right) \qquad \mbox{by} \quad (\ref{redkef2})\nonumber\\
&=&   u^{-1}q\left( \frac{k_-(q+q^{-1})}{q-q^{-1}} \!\!\!\!\!\!\!\! \!\!\!\!\!\!\!\!\!\!\!\!\underbrace{\tk_1^-(qu^2)(\tk^-_{2,0})^{-1}}_{=\tK^{-1}  \exp\left( -(q-q^{-1}) \sum_{n=1}^\infty a_{1,n} (qu^2)^{- n}\right)  } \!\!\!\!   \!\!\!\!\!\!\!\! \!\!\!\!\!\!\!\! \tf^-(qu^2)  \right)     \qquad \mbox{by} \quad (\ref{kpmz})\nonumber\ .
\eeqa
Inserting  (\ref{eq:cuDrf}),  one gets:
\beqa
(K^-(u))_{11} &=&  uq \left( -k_-(q^2+1) \exp\left( -(q-q^{-1}) \sum_{n=1}^\infty a_{1,n} (qu^2)^{- n}\right) \sum_{k=0}^\infty q^{k} C^{-k/2}\tK^{-1}{\tx}_{k}^+ (qu^2)^{-k-1}     \right) \nonumber\ .
\eeqa
According to Definition \ref{defaltq} and (\ref{a1m}),  (\ref{a2m}), we conclude $(K^-(u))_{11} \in U_q(\widehat{gl_2})^{\triangleright,+} \otimes {\mathbb C}[[u^{2}]]$. Studying similarly the other entries and repeating the same analysis for $K^+(u)$ and  $K'^+(u)$, one finds:
\beqa
 &&(K^-(u))_{ij} \in U_q(\widehat{gl_2})^{\triangleright,+}\otimes {\mathbb C}[[u^2]]\ ,\quad (K^+(u))_{ij} \in U_q(\widehat{gl_2})^{\triangleleft,-}\otimes {\mathbb C}[[u^2]]\  ,\label{Kiju}\\
\quad\mbox{and}\quad && (K'^+(u))_{ij} \in U_q(\widehat{gl_2})^{\triangleleft,-}\otimes {\mathbb C}[[u^2]]\ .\nonumber
\eeqa

\subsubsection{Isomorphisms relating  $\bar{\cal A}_q$ and the alternating subalgebras   $U_q(\widehat{gl_2})^{\triangleright,\pm}$ and $U_q(\widehat{gl_2})^{\triangleleft,\pm}$}
Recall the Freidel-Maillet type presentation for $\bar{\cal A}_q$ of Theorem \ref{thm1}.  A direct comparison between the K-matrix (\ref{K}) and the K-matrices $K^\pm(u)$ previously derived provides explicit maps from  $\bar{\cal A}_q$ to the alternating subalgebras of $U_q(\widehat{gl_2})$. Recall the  generating functions (\ref{c1}),  (\ref{c2}) of the algebra $\bar{\cal A}_q$.
\begin{prop} \label{map1}  There exists an isomorphism from   $\bar{\cal A}_q$ to $U_q(\widehat{gl_2})^{\triangleright,+}$ such that: 
\beqa
&&{\cW}_+(u)\mapsto-k_-(q^2+1) \exp\left( -(q-q^{-1}) \sum_{n=1}^\infty a_{1,n} (qu^2)^{- n}\right) \sum_{k=0}^\infty q^{k} C^{-k/2}\tK^{-1}{\tx}_{k}^+ (qu^2)^{-k-1} \ ,\label{mapAgl21}\\
&&{\cW}_-(u)\mapsto-k_+(q^{-2}+1) \left(\sum_{k=0}^\infty q^{k+1} C^{(k+1)/2} {\tx}_{k+1}^- (qu^2)^{-k-1} \right)
\exp\left( -(q-q^{-1}) \sum_{n=1}^\infty a_{1,n} (qu^2)^{- n}\right) \ ,\\
&&{\cG}_+(u)\mapsto \frac{\bar\rho}{q-q^{-1}}\left(  \exp\left( -(q-q^{-1}) \sum_{n=1}^\infty a_{1,n} (qu^2)^{- n}\right) -1 \right)\ ,\\
&&{\cG}_-(u)\mapsto \frac{\bar\rho}{q-q^{-1}}\left(  \exp\left( -(q-q^{-1}) \sum_{n=1}^\infty a_{2,n} (qu^2)^{- n}\right) -1 \right)\ \label{mapAgl24}\\\
&&\qquad  +\  \bar\rho(q-q^{-1})  \sum_{k,\ell =0}^\infty q^{k+\ell+2} C^{(k-\ell+1)/2}  {\tx}_{k+1}^- \tK^{-1}    \exp\left( -(q-q^{-1}) \sum_{n=1}^\infty a_{1,n} (qu^2)^{- n}\right)
{\tx}_{\ell}^+ (qu^2)^{-k-\ell-1} \ .\nonumber
\eeqa
\end{prop}
\begin{proof} As previously discussed, using (\ref{eq:cuDre}),  (\ref{eq:cuDrf})  and (\ref{kpmz}), the entries of $K^-(u)$ are power series in $qu^2$. Identifying (\ref{K}) with $K^-(u)$, one gets the above homomorphism  $\bar{\cal A}_q \rightarrow U_q(\widehat{gl_2})^{\triangleright,+}$  through identifying the generating functions. It remains to  show that it is an isomorphism. Firstly,  by analogy with $U_q(\widehat{sl_2})$ \cite[page 289]{CPb},  $U(\widehat{gl_2})$ with defining relations (\ref{CW1})-(\ref{CW4}) is known as  the specialization $q\rightarrow 1$ of $U_q(\widehat{gl_2})$. So, the subalgebra $U_q(\widehat{gl_2})^{\triangleright,+}$ specializes to $U(\widehat{gl_2})^{\triangleright,+}$ with  (\ref{gsl2rpm}).
 Secondly, by Proposition \ref{pDrgl2c} $\cal A\cong U(\widehat{gl_2})^{\triangleright,+}$. Thirdly, by Proposition \ref{mapspe}  $\cal A$ is the specialization of $\bar{\cal A}_q$ at $q\rightarrow 1,\bar\rho \rightarrow 16$. All together, we conclude that the map above is an isomorphism. 
\end{proof}
Identifying the leading terms of the power series, one finds for instance: 
\begin{example}\label{exi1} The image in $U_q(\widehat{gl_2})^{\triangleright,+}$ of the first generators of $ \bar{\cal A}_q$ is such that:
\beqa
\tW_0 &\mapsto& -k_-q \tK^{-1}\tx_0^+ \ ,\qquad \ \ 
\tW_1  \mapsto -k_+C^{1/2} \tx_1^- \ ,\nonumber\\
\tG_1 &\mapsto& - \frac{\bar\rho}{ q+q^{-1}} a_{1,1}\ ,\qquad
\tilde{\tG}_1 \mapsto - \frac{\bar\rho}{ q+q^{-1}} a_{2,1} + \frac{\bar\rho(q-q^{-1})}{(q+q^{-1})}q^2C^{1/2} \tx_1^- \tK^{-1}\tx_0^+\ .\nonumber
\eeqa
\end{example}

As a second example, recall the Freidel-Maillet type presentation (\ref{REp}) for $\bar{\cal A}_q$ with (\ref{K'}). In this case, the K-matrix (\ref{K'}) is compared with the K-matrix $K'^+(u)$ of Lemma \ref{K'pu}. It follows
\begin{prop}  \label{map2}
There exists an isomorphism from   $\bar{\cal A}_q$ to $U_q(\widehat{gl_2})^{\triangleleft,-}$ such that: 
\beqa
{\cW}_+(u^{-1}q^{-1})&\mapsto& k_+(q+q^{-1})\sum_{k=0}^\infty q^{-k} C^{k/2} {\tx}_{-k}^- (qu^2)^{k+1}
 \exp\left( (q-q^{-1}) \sum_{n=1}^\infty a_{1,-n} (qu^2)^{ n}\right) \ ,\nonumber\\
{\cW}_-(u^{-1}q^{-1})&\mapsto&k_-(q+q^{-1})\exp\left( (q-q^{-1}) \sum_{n=1}^\infty a_{1,-n} (qu^2)^{ n}\right)
 \left(\sum_{k=0}^\infty q^{-k+1}C^{-(k+1)/2} {\tx}_{-k-1}^+ \tK (qu^2)^{k+1} \right)
 \ ,\nonumber\\
{\cG}_+(u^{-1}q^{-1})&\mapsto& \frac{\bar\rho}{q-q^{-1}}\left(  \exp\left( (q-q^{-1}) \sum_{n=1}^\infty a_{1,-n} (qu^2)^{ n}\right) -1 \right)\ ,\nonumber\\
{\cG}_-(u^{-1}q^{-1})&\mapsto& \frac{\bar\rho}{q-q^{-1}}\left(  \exp\left( (q-q^{-1}) \sum_{n=1}^\infty a_{2,-n} (qu^2)^{ n}\right) -1 \right)\ \nonumber\\
&& +\  \bar\rho(q-q^{-1})  \sum_{k,\ell =0}^\infty   q^{-k-\ell} C^{(k-\ell-1)/2} {\tx}_{-k}^-    \exp\left( (q-q^{-1}) \sum_{n=1}^\infty a_{1,-n} (qu^2)^{ n}\right)
{\tx}_{-\ell-1}^+\tK (qu^2)^{k+\ell+1} \ .\nonumber
\eeqa
%
%
%
\end{prop}
\begin{example}\label{exi2} The image in $U_q(\widehat{gl_2})^{\triangleleft,-}$ of the first generators of  $\bar{\cal A}_q$ is such that:
\beqa
\tW_0 &\mapsto& k_+\tx_0^- \ ,\qquad \qquad \quad 
\tW_1 \mapsto k_- q C^{-1/2}  \tx_{-1}^+\tK \ ,\nonumber\\
\tG_1 &\mapsto&  \frac{\bar\rho}{ q+q^{-1}} a_{1,-1}\ ,\qquad
\tilde{\tG}_1 \mapsto  \frac{\bar\rho}{ q+q^{-1}} a_{2,-1} + \frac{\bar\rho(q-q^{-1})}{(q+q^{-1})}C^{-1/2}\tx_0^- \tx_{-1}^+ \tK\ .\nonumber
\eeqa
\end{example}
So, the alternating subalgebra $U_q(\widehat{gl_2})^{\triangleright,+}$ (resp.  $U_q(\widehat{gl_2})^{ \triangleleft,-}$) admits a Freidel-Maillet type presentation given by the K-matrix $K^-(u)$ (resp. $K'^+(u)$)  satisfying  eq. (\ref{RE}) (resp.  eq. (\ref{REp})).
Using the automorphism (\ref{thetq}), a presentation for $U_q(\widehat{gl_2})^{\triangleright,-}$ (resp.  $U_q(\widehat{gl_2})^{ \triangleleft,+}$) can be obtained as well.\vspace{1mm}

Finally, let us introduce the alternating subalgebras of $U_q^{Dr}$. 
\begin{defn}\label{defaltsl2q}
\beqa
U_q^{Dr,\triangleright,\pm} &=& \{C^{\mp k/2}\tK^{-1}\tx_{k}^\pm, C^{\pm (k+1)/2}\tx^\mp_{k+1}, \tho_k|k\in {\mathbb N}\} \ ,  \nonumber\\
U_q^{Dr,\triangleleft,\pm}&=& \{C^{\mp k/2}\tx_{-k}^\pm, C^{\pm (k+1)/2} \tx^\mp_{-k-1}\tK,\tho_k| k\in {\mathbb N}\} \ . \nonumber
\eeqa
We call $U_q^{Dr,\triangleright,\pm}$ and $U_q^{ Dr, \triangleleft,\pm}$ the
 right and left 
alternating subalgebras of  $U_q^{Dr}$.  The subalgebra generated by $\{\tK^{\pm 1},C^{\pm 1/2}\}$ is denoted  $U_q^{Dr,\diamond}$.
\end{defn}
As a corollary of  (\ref{isophiDr}) and Remark  \ref{Ct},  one has the tensor product decompositions:
\beqa
U_q(\widehat{gl_2})^{\triangleright,\pm} \cong U_q^{Dr,\triangleright,\pm}  \otimes \cal C^\triangleright\ ,\qquad U_q(\widehat{gl_2})^{\triangleleft,\pm} \cong U_q^{Dr,\triangleleft,\pm} \otimes \cal C^\triangleleft \ . \nonumber
\eeqa
Recall (\ref{expgam}).
\begin{rem} The alternating subalgebra $ U_q^{Dr,\triangleright,\pm}$ (resp. $U_q^{Dr,\triangleleft,\pm}$) is the quotient of  $U_q(\widehat{gl_2})^{\triangleright,\pm}$ (resp. $U_q(\widehat{gl_2})^{\triangleleft,\pm}$) by the ideal generated from the relations \{$\gamma_{m+1}=0| \ \forall m\in {\mathbb N}\}$ (resp. \{$\gamma_{-m-1}=0| \ \forall m\in {\mathbb N}\}$).
\end{rem}

We conclude this section with some comments. Using the isomorphism of Propositions \ref{map1}, the image of the generating function $\Delta(u) \in \cal Z \otimes {\mathbb C} \big[\big[u^2\big]\big]$ defined by  (\ref{deltau}) gives a generating function in $ \cal C^\triangleright \otimes {\mathbb C} \big[\big[u^2\big]\big]$ that looks more complicated than (\ref{ypmz}). In the context of FRT/Sklyanin/Freidel-Maillet  type presentations, this is not surprising as 
 $\Delta(u)$ and $y^\pm(qu^2)$ are built from different quantum determinants (see e.g. \cite{Skly88} for details). However, as a consistency check one can compare the leading orders of both power series. For instance, let us compute the image in $U_q(\widehat{gl_2})^{\triangleright,+}$ of $\Delta_1$   given  by (\ref{delta1}) using the expressions of Example \ref{exi1}. After simplifications using (\ref{gl2}), (\ref{gl4}), it reduces to:
\beqa
\Delta_1 = -\frac{2}{(q+q^{-1})^2}(qa_{1,1} + q^{-1}a_{2,1})\ , \nonumber
\eeqa 
which produces $\gamma_1$  (see  (\ref{expgam}) for $m=1$).\vspace{1mm}

\subsection{The comodule algebra homomorphism $\delta: \ \bar{\cal A}_q \rightarrow  U_q(\widehat{gl_2})^{\triangleright,+,0} \otimes \bar{\cal A}_q$} 
At the end of Section \ref{sec2}, a coaction map $\langle \tW_0,\tW_1 \rangle \rightarrow U_q^{DJ,+,0} \otimes \langle \tW_0,\tW_1 \rangle$ has been given. In this subsection, we study further the comodule  algebra structure of  $\bar{\cal A}_q$ using the FRT presentation of Theorem \ref{def:UqRS}. A coaction  formula for all the generators of $\bar{\cal A}_q$ is derived as follows. Recall the coproduct formulae  for the quantum Lax operators (\ref{coprodUqgl2}). Take the K-matrix (\ref{Ktmz}) and define the new K-matrix: 
\beqa
\Delta(L^-(z)) \tilde K^{0}  \Delta'(L^{-,0}) =   (L^-(zq^{-(1\otimes \frac{c}{2})}))_{[\textsf 1]} 
\left( \underbrace{ (L^-(zq^{( \frac{c}{2}\otimes 1)}))_{[\textsf 2]}   \tilde K^{0}  (L^{-,0})_{[\textsf 2]}}_{=(\tilde K^-(z q^{( \frac{c}{2}\otimes 1)}))_{[\textsf 2]}} \right)
   (L^{-,0}))_{[\textsf 1]}   \label{deltaKtmz}\ .
\eeqa
%
%
%
By construction, it satisfies  (\ref{RKz}) for the non-symmetric R-matrix (\ref{def:r}). Using the invariance of (\ref{RKz}) under shifts in the ratio $z/w$, it follows that
\beqa
\delta(\tilde K^-(z)) = (L^-(z))_{[\textsf 1]}  (\tilde K^-(z))_{[\textsf 2]} 
 ( L^{-,0})_{[\textsf 1]}   \nonumber
\eeqa
solves (\ref{RKz}). More generally, starting from any K-matrix satisfying (\ref{RKz}) and following standard arguments \cite{Skly88} different types of coactions can be constructed from the FRT presentation. Using (\ref{simil}), for a symmetric R-matrix for instance it yields to:
\begin{prop} The Freidel-Maillet type presentation (\ref{RE})  of  $\bar{\cal A}_q$  associated with   R-matrix  (\ref{R}) and K-matrix (\ref{K}) admits a  comodule  algebra structure. The left coaction is given by: 
\beqa
\delta( K^-(u)) = \left( \cal M(u)   L^- (qu^2)  \cal M(u)^{-1}  \right)_{[\textsf 1]}  
(K^-(u))_{[\textsf 2]}
(L^{-,0})_{[\textsf 1]}\ .\label{deltaKmu}
\eeqa
\end{prop}
A right coaction map is similarly obtained by analogy with (\ref{Ktpz}).
 Now, recall the generating functions (\ref{c1}), (\ref{c2}). Also, define $U_q(\widehat{gl_2})^{\triangleright,+,0}$  as the alternating subalgebra $U_q(\widehat{gl_2})^{\triangleright,+}$ extended by $\tK,\tK^{-1}$. 
\begin{lem} \label{coprodform} There exists a left comodule algebra homomorphism $\delta:  \bar{\cal A}_q \rightarrow  U_q(\widehat{gl_2})^{\triangleright,+,0} \otimes \bar{\cal A}_q$ such that:
\beqa
\delta(\cW_+(u)) &\mapsto&  ( qu^2)^{-1}q
 \tk^-_{1}(qu^2)  (\tk^-_{2,0})^{-1} \tf^-(qu^2) \otimes  \left( \frac{1}{k_+(q+q^{-1})}        \cG_-(u)     + \frac{k_-(q+q^{-1})}{(q-q^{-1})}   I\!\!I    \right) \nonumber\\
&&  + \  \tk^-_{1}(qu^2)  (\tk^-_{2,0})^{-1}   \otimes  \cW_+(u) \ ,\nonumber\\
\qquad \delta(\cW_-(u) )&\mapsto&   q^{-1}\te^-(qu^2) \tk^-_{1}(qu^2)  (\tk^-_{1,0})^{-1} \otimes 
 \left( \frac{1}{k_-(q+q^{-1})}        \cG_+(u)     + \frac{k_+(q+q^{-1})}{(q-q^{-1})}   I\!\!I    \right)\nonumber\\
&&+\ \left(    \tk^-_2(qu^2)(\tk^-_{1,0})^{-1} + \  q^{-1} \te^-(qu^2) \tk_{1}^-(qu^2) (\tk^-_{1,0})^{-1}\tf^-(qu^2)  \right) \otimes \cW_-(u)\ ,\nonumber
\eeqa
\beqa
\delta(\cG_+(u)) &\mapsto&   \tk^-_{1}(qu^2)  (\tk^-_{1,0})^{-1} \otimes \cG_+(u)  
+ \frac{\bar\rho} {q-q^{-1}}\left( \tk^-_{1}(qu^2)  (\tk^-_{1,0})^{-1} -1 \right)    \otimes  I\!\!I\nonumber\\
&&+\  k_-(q+q^{-1})  \tk^-_{1}(qu^2)  (\tk^-_{1,0})^{-1} \tf^-(qu^2) \otimes \cW_-(u) \ ,\nonumber\\
\delta(\cG_-(u)) &\mapsto&  \left( \tk^-_{2}(qu^2)  (\tk^-_{2,0})^{-1} +     q\te^-(qu^2) k_1^-(qu^2)(\tk^-_{2,0})^{-1} \tf^-(qu^2)         \right)\otimes \cG_-(u) \nonumber\\
&&+ \frac{\bar\rho} {q-q^{-1}}\left(     \tk^-_{2}(qu^2)  (\tk^-_{2,0})^{-1} +     q\te^-(qu^2) k_1^-(qu^2)(\tk^-_{2,0})^{-1} \tf^-(qu^2)       -1  \right)  \otimes  I\!\!I\nonumber\\
&&+\  k_+qu^2(q+q^{-1}) \te^-(qu^2) \tk^-_{1}(qu^2)  (\tk^-_{2,0})^{-1}  \otimes \cW_+(u) \nonumber
 \ .\nonumber
\eeqa
\end{lem}
\begin{proof}
Compute (\ref{deltaKmu}) using (\ref{Lpm}), (\ref{simil}) and  (\ref{K}) .  Compare the entries of the resulting matrix to $\delta(K(u))$  with (\ref{K}).
\end{proof}
Expanding the power series on both sides of the above equations using (\ref{c1}), (\ref{c2}), (\ref{eq:cuDre})-(\ref{eq:cuDrk}) with (\ref{kpmz}), (\ref{kk}), one gets the image by $\delta$ of the generators of  $\bar {\cal A}_q$. This generalizes example (\ref{Hisopart}). 
\begin{example}\label{excop}
\beqa
\delta(\tW_0) &=& -k_-q\tK^{-1} \tx_0^+ \otimes I\!\!I + \tK^{-1} \otimes \tW_0 \ ,\nonumber\\
\delta(\tW_1) &=& -k_+C^{1/2}\tx_1^- \otimes I\!\!I + \tK \otimes \tW_1 \ .\nonumber
\eeqa
\end{example}
If we define similarly $U_q(\widehat{gl_2})^{\triangleleft,-,0}$, note that a right coaction map  $\bar{\cal A}_q \rightarrow  \bar{\cal A}_q \otimes U_q(\widehat{gl_2})^{\triangleleft,-,0} $ can be derived along the same lines.

\subsection{Relation between the generators of $\bar {A}_q$ and root vectors of $U_q(\widehat{sl_2})$}
Let $\alpha_0,\alpha_1$ denote the simple roots of $\widehat{sl_2}$ and $\delta=\alpha_0+\alpha_1$ be the minimal positive imaginary root. Let $\cal R = \{ n\delta + \alpha_0, n\delta+\alpha_1,m\delta| n\in {\mathbb Z}, m\in {\mathbb Z}\backslash \{0\} \} $ be the root system of $\widehat{sl_2}$ and $\cal R^+ = \{ n\delta + \alpha_0, n\delta+\alpha_1,m\delta| n\in {\mathbb N}, m\in {\mathbb N}\backslash \{0\}\}$ denote the positive root system.  Recall $U_q^{DJ,+}$ denote the subalgebra generated by 
\beqa
E_{\alpha_1}\equiv E_1 \ ,\qquad E_{\alpha_0}\equiv E_0\ . \nonumber
\eeqa
 Using Lusztig's braid group action with generators $T_0,T_1$ such that $T_{i}:U_q(\widehat{sl_2}) \rightarrow U_q(\widehat{sl_2})$, root vectors $E_\beta \in U_q^{DJ,+}$ for every $\beta \in \cal R^+$ are defined \cite{Da,Beck}. Namely, for real root vectors $ n\delta + \alpha_0, n\delta+\alpha_1$ with $n\in {\mathbb N}$ one chooses
\beqa
E_{n\delta + \alpha_0} = (T_0\Phi)^n (E_0)\qquad \mbox{and} \qquad E_{n\delta + \alpha_1} = (T_0\Phi)^{-n} (E_1)\ . \nonumber
\eeqa
Here $\Phi: U_q(\widehat{sl_2}) \rightarrow U_q(\widehat{sl_2})$ denotes the automorphism defined by:
\beqa
\Phi(X_0)=X_1\ ,\qquad \Phi(X_1)=X_0\ \qquad \mbox{for}\qquad X =E,F,K^{\pm 1} \ . \nonumber
\eeqa
For the imaginary root vectors, following  \cite{Beck,BCP} they are defined through the functional equation (note that $\big[ E_{n\delta}, E_{m\delta}\big]=0$ for any $n,m$):
\beqa
&& \exp \left((q-q^{-1})   \sum_{k=1}^\infty E_{k\delta} z^k\right) = 1 +  (q-q^{-1})\sum_{k=0}^\infty  \tilde\psi_k z^k \ 
\quad \mbox{with}\quad \tilde{\psi}_k = E_{k\delta -\alpha_1} E_{\alpha_1} - q^{-2} E_{\alpha_1}E_{k\delta-\alpha_1} \ . \nonumber
\eeqa

For the negative root system denoted $\cal R^-$, similarly one defines the root vectors $F_{\beta}\in U_q^{DJ,-}$ for every $\beta \in \cal R^-$ \cite{Da}. The root vectors of $U_q^{DJ,+}$ and $U_q^{DJ,-}$ are related  as follows (see \cite[Theorem 2]{Da}):
\beqa
F_\beta = \Omega(E_{\beta}) \qquad \forall \beta\in \cal R^+\ ,\label{mapOm}
\eeqa
where $\Omega$ is an antiautomorphism of $U_q(\widehat{sl_2})$ such that
\beqa
\Omega(E_i)=F_i\ ,\quad \Omega(F_i)=E_i\ ,\quad \Omega(K_i)=K_i^{-1} \quad \mbox{for $i=1,2$}\ , \quad\Omega(C)=C^{-1}\ \mbox{and}\quad \Omega(q)=q^{-1}\ .\nonumber
\eeqa

\vspace{1mm}

The explicit relation between Drinfeld generators and root vectors has been given in \cite[Section 4]{Beck} (see also \cite[Lemma 1.5]{BCP}). For  $U_q(\widehat{sl_2})$, according to above definitions one has the correspondence:
\beqa
\tx^+_k &=& E_{k\delta + \alpha_1}\ ,\qquad
\tx^-_{k+1} =- C^{-k-1}\tK E_{k\delta + \alpha_0}\ ,\qquad \qquad\tho_{k+1} = C^{-(k+1)/2} E_{(k+1)\delta}\ ,\label{imr1}\\
\tx^-_{-k} &=& F_{k\delta + \alpha_1}\ ,\qquad
\tx^+_{-k-1} = -  F_{k\delta + \alpha_0}\tK^{-1} C^{k+1}\ ,\qquad  \tho_{-k-1} = C^{(k+1)/2} F_{(k+1)\delta}\label{imr2}
\eeqa
for $k\in {\mathbb N}$. 
From (\ref{hx}), one gets the following relations in terms of the root vectors \cite[Section 3]{Da}:
\beqa
\big[E_\delta,  E_{k\delta+\alpha_1}\big]= (q+q^{-1}) E_{(k+1)\delta + \alpha_1}\ ,\qquad
\big[ E_{k\delta+ \alpha_0}, E_\delta \big]= (q+q^{-1}) E_{(k+1)\delta + \alpha_0} \ . \label{relroot}
\eeqa
By induction, root vectors can be written as polynomials in $E_1,E_0$. For instance:
\beqa
E_\delta&=& E_0E_1 - q^{-2}E_1E_0 \ ,\nonumber\\
E_{\delta+\alpha_0}&=& \frac{1}{q+q^{-1}}\left(E_0^2E_1 - (1+ q^{-2})E_0E_1E_0 +q^{-2}E_1E_0^2 \right)\ ,\nonumber\\  
E_{\delta+\alpha_1}&=& \frac{1}{q+q^{-1}}\left(E_0E_1^2 - (1+ q^{-2})E_1E_0E_1 +q^{-2}E_1^2E_0 \right)\ .\nonumber  
\eeqa

We now relate the root vectors to the generators of alternating subalgebras. For convenience, compute  the image of $U_q^{Dr,\triangleright,+}$ (see Definition \ref{defaltsl2q}) by the automorphism $\nu$ (\ref{nu}) using (\ref{thok}). This alternating subalgebra is denoted $(U_q^{Dr,\triangleright,+})^\nu$. Using (\ref{imr1}), in terms of root vectors the generators of $(U_q^{Dr,\triangleright,+})^\nu$ read:
\beqa
C^{-k/2}\tK^{-1}\tx_k^+ &\stackrel{\nu}\mapsto&  \qquad   C^{-k/2}\tx_k^+    =  C^{-k/2} E_{k\delta + \alpha_1} \ ,\label{xkp1}\\
C^{(k+1)/2}\tx_{k+1}^- &\stackrel{\nu}\mapsto&   C^{(k+1)/2}\tx_{k+1}^-\tK^{-1}   = -  q^{-2} C^{-(k+1)/2}E_{k\delta+ \alpha_0} \ ,\label{xkm2}\\
\tho_{k+1}&\stackrel{\nu} \mapsto& \qquad \qquad \tho_{k+1} = C^{-(k+1)/2} E_{(k+1)\delta}\ .\label{hp3}
\eeqa

As an application of  Proposition \ref{map1}, a set of functional relations relating the  generators of $\bar{A}_q$ to the  root vectors of $U_q^{DJ,+}$ (or similarly for $U_q^{DJ,-}$) is easily derived. Recall the surjective homomorphism  $\gamma:\bar{\cal A}_q \rightarrow \bar{ A}_q \cong  U_q^{DJ,+}$, see (\ref{mapgam}). Consider the image of the generating functions (\ref{c1}),  (\ref{c2}) via $\gamma$.
\begin{prop}\label{prop:Aroot} The  isomorphism $\iota: \bar{A}_q \rightarrow U_q^{DJ,+}$ is such that:
\beqa
\gamma({\cW}_+(u))&\mapsto&-k_-q(q+q^{-1}) \exp\left( -(q-q^{-1}) \sum_{n=1}^\infty \frac{1}{(q^n+q^{-n})}E_{n\delta} (qu^2)^{- n}\right) \sum_{k=0}^\infty q^{k}E_{k\delta+\alpha_1} (qu^2)^{-k-1} \ ,\nonumber\\
\gamma({\cW}_-(u))&\mapsto&k_+q^{-1}(q+q^{-1})\left(\sum_{k=0}^\infty q^{k-1} E_{k\delta+\alpha_0} (qu^2)^{-k-1} \right)
\exp\left( -(q-q^{-1}) \sum_{n=1}^\infty \frac{1}{(q^n+q^{-n})}E_{n\delta} (qu^2)^{- n}\right)
 \ ,\nonumber\\
\gamma({\cG}_+(u))&\mapsto& \frac{\bar\rho}{(q-q^{-1})}\left( \exp\left( -(q-q^{-1}) \sum_{n=1}^\infty \frac{1}{(q^n+q^{-n})}E_{n\delta} (qu^2)^{- n}\right) - 1\right)
\ ,\nonumber\\
\gamma({\cG}_-(u))&\mapsto& \frac{\bar\rho}{(q-q^{-1})}\left( 
 \exp\left((q-q^{-1}) \sum_{n=1}^\infty \frac{q^{2n}}{(q^n+q^{-n})}E_{n\delta} (qu^2)^{- n}\right) 
-1 \right)\ \nonumber\\
&&\!\!\!\!\!\!\!\!\!\!\!\!\!\!\!\!\!\!\!\! +\  \bar\rho(q-q^{-1}) \sum_{k,\ell =0}^\infty q^{k+\ell} E_{k\delta+\alpha_0} \exp\left( -(q-q^{-1}) \sum_{n=1}^\infty \frac{1}{(q^n+q^{-n})}E_{n\delta} (qu^2)^{- n}\right) 
 E_{\ell\delta+\alpha_1} (qu^2)^{-k-\ell-1} \ .\nonumber
\eeqa
\end{prop}
\begin{proof} Recall the surjective homomorphism $\gamma_D$ which acts as (\ref{gamD1})-(\ref{gamD3}). Consider its restriction to $U_q(\widehat{gl_2})^{\triangleright,+}$, applied
to the r.h.s. of  (\ref{mapAgl21})-(\ref{mapAgl24}). The resulting expressions are now in $U_q^{Dr,\triangleright,+}\otimes {\mathbb C}[[u^2]]$. Then, studying the relations satisfied by $\{C^{-k/2}\tK^{-1}\tx_k^+,C^{(k+1)/2}\tx_{k+1}^-,\tho_{k+1}\}$ one finds that they are equivalent to the defining relations of the quotient of $U_q^{Dr,\triangleright,+}$ by $C=1$. Apply $\nu$ and use the identification given in the r.h.s of (\ref{xkp1})-(\ref{hp3}) for $C=1$. 
\end{proof}
 Expanding the above power series, for instance set   $k_+ \rightarrow q^2$, $k_- \rightarrow -q^{-1}$ (which gives $\bar\rho=-q(q+q^{-1})^2$)  in these expressions. It follows:
\beqa
&&W_0 \mapsto E_1\ ,\qquad W_1 \mapsto E_0\ ,\qquad
G_1   \mapsto qE_\delta\ ,\qquad \mbox{(note that}\ \   \tilde{G}_1   \mapsto -q^3E_\delta + (q^3-q^{-1})E_0E_1)\ ,\label{exi3}\\
&&W_{-1}  \mapsto  \frac{1}{(q+q^{-1})^2}\left( -(q-q^{-1})E_\delta E_1 + (q^2+1)E_{\delta+\alpha_1}\right)\ ,\label{exi4}\\
&&W_{2}  \mapsto  \frac{1}{(q+q^{-1})^2}\left( -(q-q^{-1}) E_0E_\delta + (q^2+1)E_{\delta+\alpha_0}\right)\ .\label{exi5}
\eeqa
By construction, $(U_q^{Dr,\triangleright,+})^\nu/_{C=1}  \cong U_q^{DJ,+}$.
Using (\ref{mapOm}), an isomorphism $ \bar{A}_q \rightarrow  U_q^{Dr,\triangleleft,-}/_{C=1}  \cong U_q^{DJ,-}$ is obtained from  the above expressions.
\vspace{1mm}

The  inverse of the map $\iota$ is now considered. We want to solve the positive root vectors $E_{n\delta+\alpha_1},E_{n\delta+\alpha_0},E_{n\delta}$  in terms of the generators  $W_{-k},W_{k+1},G_{k+1}$.
 Although we do not have the explicit inverse map between generating functions,  the images of the root vectors in $\bar A_q$ can be obtained  recursively from Proposition \ref{prop:Aroot}.  For instance,
\beqa
&& E_1\mapsto  W_0\ ,\qquad E_0 \mapsto  W_1\ ,\qquad
E_\delta  \mapsto   q^{-1}G_1W_0 \ ,\qquad    \ ,\label{exib1}\\
&&E_{\delta+\alpha_1}  \mapsto  \frac{(q-q^{-1})}{(q+q^{-1})}q^{-2}G_1W_0
+ (1+q^{-2})W_{-1}\ ,\label{exib4}\\
&&E_{\delta+\alpha_0}  \mapsto   \frac{(q-q^{-1})}{(q+q^{-1})}q^{-2}W_1G_1
+ (1+q^{-2})W_{2} \ .\label{exib5}
\eeqa
Of course, these expressions could be given in a different ordering (see Theorem \ref{pbwAbar}) using (\ref{def3}) for $k=0$.\vspace{1mm}

Finally, let us point that several relations mixing both sets of generators can be readily obtained using (\ref{RE}) combined with Proposition \ref{prop:Aroot}. Namely, define the image of the K-matrix (\ref{K}) by $\iota$ as:
\beqa
K^\iota(u) = \iota( K(u))\ .\label{Kgam}
\eeqa
Consider  the pair of K-matrices $\{K(u),K^\iota(v)\}$. They satisfy:
\begin{align} R(u/v)\ (K(u)\otimes I\!\!I)\ R^{(0)}\ (I\!\!I \otimes K^\iota(v))\
= \ (I\!\!I \otimes K^\iota(v))\  R^{(0)}\ (K(u)\otimes I\!\!I)\ R(u/v)\ 
\label{REmixed}
 \end{align}
with (\ref{R}). If we  define the generating functions $\cW_\pm(v)^{\iota,\gamma}= \iota \circ \gamma (\cW_\pm(v))$, $\cG_\pm(v)^{\iota,\gamma}= \iota \circ \gamma (\cG_\pm(v))$, from (\ref{ec1})-(\ref{ec16}) one extracts the set of functional relations  associated with (\ref{REmixed}).
\begin{rem} In \cite[Section 11]{Ter19}, the relation between Damiani's  PBW basis and the alternating PBW basis for $\bar A_q$ has been studied in details  within the framework of the $q$-shuffle algebra.  In particular, various relations  mixing  both sets of generators have been obtained.
\end{rem}

\section{The alternating presentation of $U_q(\widehat{sl_2})$ from $U_q^{DJ}$}
Define the alternating subalgebra $\bar{A}_q^\triangleright \cong (U_q^{Dr,\triangleright,+})^\nu/_{C=1} $  (resp.  $\bar{A}_q^\triangleleft\cong  U_q^{Dr,\triangleleft,-}/_{C=1}$) as the image of  $ \bar{A}_q$ by $\iota$ (resp. $\Omega \circ \iota$) (see  Proposition \ref{prop:Aroot}) for $k_+ \rightarrow q^2$, $k_- \rightarrow -q^{-1}$. 
For convenience, let us denote the generators of $ \bar{A}_q^\triangleright$ (resp. $ \bar{A}_q^\triangleleft$)  by  $\{W_{-k}^\triangleright,W_{k+1}^\triangleright,G_{k+1}^\triangleright,\tilde G_{k+1}^\triangleright|k\in{\mathbb N}  \}$  (resp. $\{W_{-k}^\triangleleft,W_{k+1}^\triangleleft,G_{k+1}^\triangleleft,\tilde G_{k+1}^\triangleleft|k\in{\mathbb N}  \}$) . According to (\ref{exi3}):
\beqa
W_{0}^\triangleright=E_1\ ,\quad W_{1}^\triangleright=E_0 \ ,\quad W_{0}^\triangleleft=F_1\ ,\quad W_{1}^\triangleleft=F_0 \ . 
\eeqa
Recall   Proposition \ref{prop:AqUDJ} and  $U_q^{DJ,0}= \{K_0,K_1\}$.  By construction, one gets the tensor product decomposition:
\beqa
U_q(\widehat{sl_2})\cong\bar{A}_q^\triangleright \otimes  U_q^{DJ,0} \otimes  \bar{A}_q^\triangleleft \ . \label{newdec}
\eeqa
Moreover, by Theorem \ref{pbwAbar} an `alternating' PBW basis for $U_q(\widehat{sl_2})$ readily follows from the results of \cite{Ter19,Ter19b}.
\begin{thm}\label{thmfin} A PBW basis for $U_q(\widehat{sl_2})$ is obtained by its alternating right and left generators
\beqa
\{W^\triangleright_{-k}\}_{k\in {\mathbb N}}\ ,\quad \{G^\triangleright_{\ell+1}\}_{\ell\in {\mathbb N}}\ ,\quad \{W^\triangleright_{n+1}\}_{n\in {\mathbb N}} \ ,\quad \{W^\triangleleft_{-r}\}_{r\in {\mathbb N}}\ ,\quad \{G^\triangleleft_{s+1}\}_{s\in {\mathbb N}}\ ,\quad \{W^\triangleleft_{t+1}\}_{t\in {\mathbb N}} \nonumber
\eeqa
and $K_0,K_1$
in any linear order $<$  that satisfies 
\beqa
W^\triangleright_{-k}<   G^\triangleright_{\ell+1}<  W^\triangleright_{n+1} < K_0 < K_1 <  W^\triangleleft_{r+1}<   G^\triangleleft_{s+1}<  W^\triangleleft_{-t}\ ,\qquad  k,\ell,n,r,s,t \in {\mathbb N}\ . \nonumber
\eeqa
\end{thm}
The transition matrix from the alternating PBW basis of Theorem \ref{thmfin} to Damiani's PBW basis for $U_q(\widehat{sl_2})$  \cite[Theorem 2]{Da} is determined by Proposition \ref{prop:Aroot} and using the antiautomorphism $\Omega$  (\ref{mapOm}).\vspace{1mm}
\vspace{5mm}

\noindent{\bf Acknowledgments:}  I am very grateful to  Naihuan Jing,   Stefan Kolb and Paul Terwilliger for  discussions. In particular, some results were obtained motivated by questions from Paul Terwilliger. I also thank  Nicolas Cramp\'e for gratefully sharing a MAPLE code and discussions.
 P.B.  is supported by C.N.R.S. 
\vspace{0.2cm}

\begin{appendix}

\section{Drinfeld-Jimbo  presentation of $U_q(\widehat{sl_2})$}\label{apA}
\vspace{2mm}
\subsection{Drinfeld-Jimbo presentation $U_q^{DJ}$}
Define the extended Cartan matrix $\{a_{ij}\}$ ($a_{ii}=2$,\ $a_{ij}=-2$ for $i\neq j$). The quantum affine algebra $U_{q}(\widehat{sl_2})$ over ${\mathbb C}(q)$ is generated by  $\{E_j,F_j,K_j^{\pm 1}\}$, $j\in \{0,1\}$ which satisfy the defining relations
\beqa 
K_iK_j=K_jK_i\ , \quad K_iK_i^{-1}=K_i^{-1}K_i=1\ , \quad
K_iE_jK_i^{-1}= q^{a_{ij}}E_j\ ,\quad
K_iF_jK_i^{-1}= q^{-a_{ij}}F_j\ ,\quad
[E_i,F_j]=\delta_{ij}\frac{K_i-K_i^{-1}}{q-q^{-1}}\
\nonumber\eeqa
together with the $q-$Serre relations  ($i\neq j$)
\beqa
 \big[E_i, \big[E_i, \big[E_i,E_j \big]_{q} \big]_{q^{-1}} \big]&=&0\ ,\label{defUqDJp}\\
 \big[F_i, \big[F_i, \big[F_i,F_j \big]_{q} \big]_{q^{-1}} \big]&=& 0\ . \label{defUqDJm}
\eeqa
The  product $C=K_0K_1$ is the central element of the algebra. The
Hopf algebra structure is ensured by the existence of a
comultiplication $\Delta$
, antipode ${\cal S}$
and a counit ${\cal E}$
with
\beqa \Delta(E_i)&=& 1 \otimes E_i + E_i \otimes K_i
\ ,\label{coprod} \\
 \Delta(F_i)&=&F_i \otimes 1 +   K_i^{-1}\otimes F_i\ ,\nonumber\\
 \Delta(K_i)&=&K_i\otimes K_i\ ,\nonumber
\eeqa
%
%
%
\beqa {\cal S}(E_i)=-E_iK_i^{-1}\ ,\quad {\cal S}(F_i)=-K_iF_i\ ,\quad {\cal S}(K_i)=K_i^{-1} \qquad {\cal S}({1})=1\
\label{antipode}\nonumber\eeqa
and\vspace{-0.3cm}
\beqa {\cal E}(E_i)={\cal E}(F_i)=0\ ,\quad {\cal
E}(K_i)=1\ ,\qquad {\cal E}(1)=1\
.\label{counit}\nonumber\eeqa
More generally, one defines the $N-$coproduct $\Delta^{(N)}: \
U_{q}(\widehat{sl_2}) \longrightarrow
U_{q}(\widehat{sl_2}) \otimes \cdot\cdot\cdot \otimes
U_{q}(\widehat{sl_2})$ as
\beqa \Delta^{(N)}\equiv (id\times \cdot\cdot\cdot \times id
\times \Delta)\circ \Delta^{(N-1)}\ \label{coprodN}\eeqa
for $N\geq 3$ with $\Delta^{(2)}\equiv \Delta$,
$\Delta^{(1)}\equiv id$.
Note that the opposite coproduct $\Delta'$ can be similarly defined with $\Delta'\equiv \sigma
\circ\Delta$ where the permutation map $\sigma(x\otimes y
)=y\otimes x$ for all $x,y\in U_{q}(\widehat{sl_2})$ is used.\vspace{2mm}

\subsection{Serre-Chevalley presentation $\widehat{sl_2}^{SC}$} In the definition below, $[.,.]$ denotes the Lie bracket.
The affine algebra $\widehat{sl_2}$ over ${\mathbb C}$ is generated by  $\{e_j,f_j,k_j\}$, $j\in \{0,1\}$ which satisfy the defining relations
\beqa 
\big[k_i,k_j\big]=0\ ,\quad 
\big[k_i,e_j\big]=a_{ij} e_j \ ,\quad \big[k_i,f_j\big]=- a_{ij} f_j \ ,\quad 
\big[e_i,f_j\big]=\delta_{i,j}k_i\
\nonumber\eeqa
together with the Serre relations ($i\neq j$)
\beqa
 \big[e_i, \big[e_i, \big[e_i,e_j \big]\big] \big]&=&0\ ,\label{defSCp}\\
 \big[f_i, \big[f_i, \big[f_i,f_j \big] \big] \big]&=& 0\ . \label{defSCm}
\eeqa
The  sum $c=k_0+k_1$ is the central element of the algebra. \vspace{1mm}

For $U(\widehat{sl_2}^{SC})$, as usual $[x,y] \rightarrow xy-yx$.

\section{Some defining relations of Gao-Jing presentation of $U_q(\widehat{gl_2})$ }\label{apB}
\vspace{2mm}
We refer the reader to \cite[Theorem 4.16]{Jing}. From Definition \ref{def:UqDrgl2} and (\ref{a1m}), (\ref{a2m}), the following commutation relations are derived: 
\beqa
\big[a_{i,m}, a_{i,n}\big]&=&0\ , i=1,2 \label{comam1}\ ,\\
\big[a_{2,m}, a_{1,n}\big]
&=&-\frac{\big[m\big]}{m}[mc]q^{-m}\delta_{m+n,0}
\ ,\\
\big[a_{1,m}, x^{\pm}_{n}\big]
&=&\pm\frac{\big[m\big]}{m}q^{\mp|m|c/2} x^{\pm}_{m+n}\ ,\\
\big[a_{2,m}, x^{\pm}_{n}\big]&=&\mp\frac{\big[m\big]}{m}q^{2m\mp|m|c/2} x^{\pm}_{m+n}\ .\label{comam4}
\eeqa

\end{appendix}

\vspace{1cm}

\end{document}